\documentclass[a4paper,12pt]{article}
\usepackage{currfile,datetime}
\usepackage{cite}
\usepackage{amsthm,amsmath,amssymb,bbm,geometry,epsfig,listings,
paralist,
enumerate,comment,nicefrac,verbatim,multirow,xcolor}
\usepackage{mathrsfs}  
\usepackage{mathtools}
\usepackage{xparse}

\usepackage[colorlinks,linkcolor=red,citecolor=green]{hyperref}

\usepackage[capitalise,compress]{cleveref} 

\crefname{equation}{}{}

\newtheorem{lemma}{Lemma}[section]

\newtheorem{proposition}[lemma]{Proposition}
\newtheorem{theorem}[lemma]{Theorem}

\newtheorem{corollary}[lemma]{Corollary}

\newtheorem{setting}[lemma]{Setting}

\crefname{subsection}{Subsection}{Subsections}
\crefname{enumi}{item}{items}
\creflabelformat{enumi}{(#2\textup{#1}#3)}

\DeclareFontEncoding{LS1}{}{}
\DeclareFontSubstitution{LS1}{stix}{m}{n}
\DeclareMathAlphabet{\mathscr}{LS1}{stixscr}{m}{n}

\newcommand{\defeq}{\curvearrowleft}
\newcommand{\stdNormal}{\mathbf{z}}

\newcommand{\vast}[2]{\left#2 \rule{0pt}{#1}\kern-.25ex\right.}
\newcommand{\sfN}{\mathsf{n}}
\newcommand{\R}{\mathbbm{R}}
\newcommand{\N}{\mathbbm{N}}
\newcommand{\Z}{\mathbbm{Z}}
\newcommand{\1}{\mathbbm{1}}
\renewcommand{\P}{\mathbbm{P}}
\newcommand{\E}{\mathbbm{E}}

\newcommand{\sqrtT}[1]{{\scriptstyle\sqrt{#1}}}

\newcommand{\funcH}[2]{{\left\vert\kern-0.25ex\left\vert\kern-0.25ex\left\vert #1     \right\vert\kern-0.25ex\right\vert\kern-0.25ex\right\vert}_{2,#2}}

\newcommand{\funcN}[2]{{\left\vert\kern-0.25ex\left\vert\kern-0.25ex\left\vert #1     \right\vert\kern-0.25ex\right\vert\kern-0.25ex\right\vert}_{1,#2}}

\newcommand{\lpspace}[1]{{#1}}

\NewDocumentCommand{\fabs}{sO{}m}{%
  {\IfBooleanTF{#1}
    {\fabsaux{\left|}{\right|}{#3}}
    {\fabsaux{#2|}{#2|}{#3}}}
}
\makeatletter
\newcommand{\fabsaux}[3]{\mathpalette\fabsaux@i{{#1}{#2}{#3}}}
\newcommand{\fabsaux@i}[2]{\fabsaux@ii#1#2}
\newcommand{\fabsaux@ii}[4]{%
  \sbox\z@{$\m@th#1#2#4#3$}%
  \sbox\tw@{$\m@th\|$}%
  \mathopen{\hbox to\wd\tw@{\hss\vrule height \ht\z@ depth \dp\z@ width .3\wd\tw@\hss}}%
  \mkern-2mu #4 \mkern-2mu 
  \mathclose{\hbox to\wd\tw@{\hss\vrule height \ht\z@ depth \dp\z@ width .3\wd\tw@\hss}}%
}
\makeatother

\NewDocumentCommand{\ffabs}{som}{%
  {\IfBooleanTF{#1}
    {\fabsaux{\left|}{\right|}{#3}}
    {\IfNoValueTF{#2}
      {\fnabsaux{|}{|}{#3}}
      {\fabsaux{#2|}{#2|}{#3}}
    }
  }
}
\makeatletter
\newcommand{\fnabsaux}[3]{\mathpalette\fnabsaux@i{{#1}{#2}{#3}}}
\newcommand{\fnabsaux@i}[2]{\fnabsaux@ii#1#2}
\newcommand{\fnabsaux@ii}[4]{%
  \sbox\z@{$\m@th#1#2#4#3$}%
  \sbox\tw@{$\m@th\|$}%
  \mathopen{\hbox to\wd\tw@{\hss\vrule height .8\ht\z@ depth .5\dp\z@ width .3\wd\tw@\hss}}%
  \mkern-2mu #4 \mkern-2mu 
  \mathclose{\hbox to\wd\tw@{\hss\vrule height .8\ht\z@ depth .5\dp\z@ width .3\wd\tw@\hss}}%
}
\makeatother

\DeclarePairedDelimiter{\lp}{\lVert}{\rVert}

\newcommand{\blp}[1]{\lp[\big]{#1}}
\newcommand{\bblp}[1]{\lp[\Big]{#1}}
\newcommand{\bbblp}[1]{\lp[\bigg]{#1}}

\newcommand{\unif}{\mathfrak{r}}

\title
{Overcoming the curse of dimensionality\\
 in the numerical approximation of\\
  backward stochastic differential equations}
\author{Martin Hutzenthaler$^{1}$, Arnulf Jentzen$^{2,3}$, \\
Thomas Kruse$^{4}$, and Tuan Anh Nguyen$^{5}$\bigskip\\
\small{$^1$ Faculty of Mathematics, University of Duisburg-Essen,}\\
\small{Essen, Germany; e-mail: \texttt{martin.hutzenthaler}\textcircled{\texttt{a}}\texttt{uni-due.de}}\\
\small{$^2$ Applied Mathematics M\"unster, Faculty of Mathematics and Computer Science,}\\ 
\small{University of M\"unster, M\"unster, Germany; e-mail: \texttt{ajentzen}\textcircled{\texttt{a}}\texttt{uni-muenster.de}}\\
\small{$^3$School of Data Science and Shenzhen Research Institute of Big Data,}\\
\small{The Chinese University of Hong Kong,}\\
\small{Shenzhen, China; e-mail: \texttt{ajentzen}\textcircled{\texttt{a}}\texttt{cuhk.edu.ch}}\\
\small{$^4$ Institute of Mathematics, University of Gie{\ss}en,}\\
\small{Gie{\ss}en, Germany; e-mail: \texttt{thomas.kruse}\textcircled{\texttt{a}}\texttt{math.uni-giessen.de}}\\
\small{$^5$ Faculty of Mathematics, University of Duisburg-Essen,}\\
\small{Essen, Germany; e-mail: \texttt{tuan.nguyen}\textcircled{\texttt{a}}\texttt{uni-due.de}}
}

\begin{document}

\maketitle
\begin{abstract}
Backward stochastic differential equations (BSDEs) belong nowadays to the most frequently studied equations in stochastic analysis and computational stochastics. 
BSDEs in applications are often nonlinear and high-dimensional. In nearly all cases such nonlinear high-dimensional BSDEs cannot be solved explicitly
and it has been and still is a very active topic of research to design and analyze numerical approximation methods to approximatively solve nonlinear high-dimensional BSDEs. Although there are a large number of research articles in the scientific literature which
analyze numerical approximation methods for nonlinear BSDEs, until today there has been no numerical approximation method in the scientific literature which has been proven to overcome the curse of dimensionality in the numerical approximation of nonlinear BSDEs in the sense that the number of computational operations of the numerical approximation method to approximatively compute one sample path of the BSDE solution grows at most polynomially in both the reciprocal $\nicefrac{ 1 }{ \varepsilon }$
of the prescribed approximation accuracy $\varepsilon \in (0,\infty)$ and the dimension $d\in \N=\{1,2,3,\ldots\}$ of the BSDE.
It is the key contribution of this article to overcome this obstacle by introducing a new Monte Carlo-type numerical approximation
method for high-dimensional BSDEs and by proving that this Monte Carlo-type numerical approximation method does indeed 
overcome the curse of dimensionality in the approximative computation 
of solution paths of BSDEs.
\end{abstract}

\newpage

\tableofcontents

\section{Introduction}\label{sec:intro}
\newcommand{\introkap}{\kappa}
\newcommand{\sigmaAlgebra}{\sigma}
\newcommand{\intOp}{\mathscr{L}}
\newcommand{\thetaBar}{\theta}
\newcommand{\F}{\mathbb{F}}
\renewcommand{\epsilon}{\varepsilon}
\newcommand{\FEU}[1]{\mathcal{C}_{#1}}
\newcommand{\FEY}[1]{{\mathfrak{C}}_{#1}}
\newcommand{\Yappr}{\mathscr{Y}}

Backward stochastic differential equations (BSDEs) have been introduced by Pardoux \& Peng in 1990 (see \cite{PardouxPeng1990}) and belong nowadays to the most frequently studied equations in stochastic analysis and computational stochastics. 
One central reason for the high interest in studying BSDEs is
their numerous occurrence in relevant real life problems.
In particular, BSDEs appear in the approximative valuation of financial products such as financial derivative contracts (see, e.g., \cite{ElKarouiPengQuenez1997,crepey2013financial,delong2013backward}), BSDEs arise in the solution of stochastic optimal control problems (see, e.g., \cite{touzi2012optimal, pham2009continuous, yong1999stochastic}), and BSDEs are strongly linked to nonlinear partial differential equations (PDEs) which themselves arise naturally in many applications (see, e.g., \cite{PardouxPeng1992,peng1991probabilistic,pardoux1999bsdes,pardoux2014stochastic}). 

BSDEs in applications are often nonlinear and high-dimensional where, e.g., in the approximative valuation of financial products the dimension of the BSDE essentially corresponds to the number of financial assets in the associated hedging portfolio, where, e.g., in stochastic optimal control problems the dimension of the BSDE is determined by the dimension of the state space of the stochastic control problem, and where, e.g., in the case of the connection of BSDEs and PDEs the dimension of the BSDE coincides with the dimension of the associated nonlinear PDE. 

In nearly all cases nonlinear high-dimensional BSDEs cannot be solved explicitly
and it has been and still is a very active topic of research to design and analyze numerical approximation methods to approximatively solve nonlinear high-dimensional BSDEs.
Standard numerical approximation methods for nonlinear BSDEs in the scientific literature suffer under the so-called curse of dimensionality 
(cf., e.g., Bellman \cite{Bellman}, Novak \& Wozniakowski \cite[Chapter~1]{NovakWozniakowski2008I}, and Novak \& Ritter \cite{MR1485004})
in the sense that the number of computational operations of the numerical approximation method to approximatively compute one sample path of the BSDE solution grows at least exponentially in the reciprocal $\nicefrac{1}{\varepsilon}$ of the prescribed approximation accuracy $\varepsilon \in (0,\infty)$ or the dimension $d\in \N=\{1,2,3,\ldots\}$ of the BSDE and it is a key objective in computational stochastics to design and analyze numerical approximation methods which overcome the curse of dimensionality in the numerical approximation of BSDEs. 

Since BSDEs have been introduced by Pardoux \& Peng in 1990 (see \cite{PardouxPeng1990}), 
a large number of numerical approximation methods for nonlinear BSDEs have been proposed and analyzed in the scientific literature.
In particular, we refer, for example, to
\cite{gobet2007error, BouchardTouzi2004,
      ChassagneuxRichou2015, 
      CrisanManolarakis2010,
      LionnetDosReisSzpruch2015, 	      
      Pham2015, 
      Turkedjiev2015,
      Zhang2004}
for numerical approximation methods for BSDEs based on one-step temporal 
discretizations of BSDEs, 
we refer, for example, to 
\cite{zhao2010stable, Chassagneux2014,
      ChassagneuxCrisan2014, teng2020multi, zhao2014new}
for numerical approximation methods for BSDEs based on multi-step temporal 
discretizations of BSDEs, 
we refer, for example, to 
\cite{bender2012least,
GobetLemorWarin2005,
GobetLemor2008Numerical, 
LemorGobetWarin2006, 
GobetTurkedjiev2016,
      GobetTurkedjiev2016MathComp,  
      RuijterOosterlee2015,
      GobetLopezSalasTurkedjiev2016
}
for numerical approximation methods for BSDEs based on suitable projections on function spaces,
we refer, for example, to 
\cite{CrisanManolarakis2012, 
      CrisanManolarakis2014,
      chassagneux2020cubature,
      de2015cubature
}
for cubature-based numerical approximation methods for BSDEs,
we refer, for example, to 
\cite{BouchardTouzi2004, CrisanManolarakisTouzi2010, hu2011malliavin}
for numerical approximation methods for BSDEs based on Malliavin calculus,
we refer, for example, to 
\cite{BallyPages2003,bally2003error, delarue2008interpolated,DelarueMenozzi2006}
for numerical approximation methods for BSDEs based on quantization algorithms,
we refer, for example, to 
\cite{ChangLiuXiong2016,le2019forward,le2018monte, le2017particle}
for numerical approximation methods for BSDEs based on density representations of particle systems,
we refer, for example, to 
\cite{ChassagneuxRichou2016,
richou2011numerical,richou2012markovian, imkeller2010results}
for numerical approximation methods for quadratic BSDEs, 
we refer, for example, to 
\cite{BenderDenk2007,bender2008time}
for numerical approximation methods for BSDEs based on Picard iterations and 
the least squares Monte Carlo method, 
we refer, for example, to 
\cite{GobetLabart2010, LabartLelong2013}
for numerical approximation methods for BSDEs based on Picard iterations and adaptive control variates,
we refer, for example, to 
\cite{FuZhaoZhou,
      ZhangGunzburgerZhao2013} 
for numerical approximation methods for BSDEs based on sparse grid approximations, 
we refer, for example, to 
\cite{BriandLabart2014, 
      GeissLabart2016}
for numerical approximation methods for BSDEs based on Wiener chaos expansions, 
we refer, for example, to 
\cite{zhao2006new,zhao2009error, zhao2012generalized}
for numerical approximation methods for BSDEs based on the theta-scheme,
we refer, for example, to 
\cite{cvitanic2005steepest}
for numerical approximation methods for BSDEs based on steepest descent algorithms,
we refer, for example, to 
\cite{chevance1997numerical, briand2019donsker, briand2001donsker, geiss2020random, geiss_labart_luoto_2020,MaProtterSanMartin2002}
for numerical approximation methods for BSDEs based on discrete time approximations of Brownian motions,
we refer, for example, to 
\cite{RuijterOosterlee2016, HuijskensRuijterOosterlee2016}
for numerical approximation methods for BSDEs based on Fourier expansions,
we refer, for example, to 
\cite{bender2017iterative, Bender2015Primal, bender2018pathwise}
for numerical approximation methods for BSDEs based on the primal-dual method,
we refer, for example, to 
\cite{Labordere2012,
      HenryLabordereOudjaneTanTouziWarin2016,
      LabordereTanTouzi2014,
      RasulovRaimoveMascagni2010,
      warin2017variations, agarwal2020branching,McKean1975,
 SkorohodBranchingDiffusion1964, 
  Watanabe1965Branching}
for numerical approximation methods for BSDEs based on
branching diffusion representations of PDEs,
we refer, for example, to 
\cite{DouglasMaProtter, MaProtterYong1994, MilsteinTretyakov2006,MilsteinTretyakov2007,MaYong1999}
for numerical approximation methods for BSDEs based on the four-step-scheme,
we refer, for example, to 
\cite{abbas2020conditional, abbas2020conditional2}
for numerical approximation methods for BSDEs based on conditional Monte Carlo learning for diffusion processes,
and we refer, for example, to \cite{chen2019deep,EHanJentzen2017CMStat,FujiiTakahashiTakahashi2017,HenryLabordere2017} and the references mentioned in the overview articles \cite{beck2020overview,han2020algorithms} for deep learning-based approximation methods for BSDEs.


Although there are a large number of research articles in the scientific literature which
analyze numerical approximation methods for nonlinear BSDEs, until today there has been no numerical approximation method in the scientific literature which has been proven to overcome the curse of dimensionality in the numerical approximation of nonlinear BSDEs in the sense that the number of computational operations of the numerical approximation method to approximatively compute one sample path of the BSDE solution grows at most polynomially in both the reciprocal $\nicefrac{ 1 }{ \varepsilon }$
of the prescribed approximation accuracy $\varepsilon \in (0,\infty)$ and the dimension $d\in \N=\{1,2,3,\ldots\}$ of the BSDE. This concept is also referred to as \emph{polynomial tractability} in the 
scientific literature (see, e.g., Novak \& Wozniakowski \cite[Definition~4.44]{NovakWozniakowski2008I}).

It is the key contribution of this article to overcome this obstacle by introducing a new Monte Carlo-type numerical approximation
method for high-dimensional BSDEs and by proving that this Monte Carlo-type numerical approximation method does indeed 
overcome the curse of dimensionality in the approximative computation 
of solution paths of BSDEs. Remarkably, this article even demonstrates that 
the introduced Monte Carlo-type numerical approximation method 
approximates solution paths of BSDEs with essentially 
the same computational complexity that is used by 
standard Monte Carlo methods for the approximative computation of integrals. More specifically, 
the main result of this article, \cref{t26} in \cref{sec:compcompl} below, 
proves that the introduced Monte Carlo-type numerical approximation method 
approximates solution paths of BSDEs with a computational effort 
which grows at most polynomially in the dimension $d\in \N$ of the driving Brownian motion 
and essentially at most quadratically in the reciprocal of the prescribed 
approximation accuracy.

The Monte Carlo-type numerical approximation method for BSDEs proposed in this article (see \eqref{eq:approx_intro} below) is based on full-history recursive multilevel Picard approximation methods \cite{EHutzenthalerJentzenKruse2016, HJKNW2018, EHutzenthalerJentzenKruse2017} (in the following we abbreviate \emph{full-history recursive multilevel Picard} by MLP) and on the multilevel approach in Heinrich \cite{h98,Heinrich01}. MLP approximations have previously been shown to overcome the curse of dimensionality in the case of a number of semilinear PDE problems (cf. \cite{EHutzenthalerJentzenKruse2016, HJKNW2018, becker2020numerical, EHutzenthalerJentzenKruse2017, hutzenthaler2019overcoming, giles20019generalised, HJKN20, HutzenthalerKruse2017, hjk2019overcoming, beck2019overcoming, beck2020overcoming,beck2020nonlinear}) and this is also the key ingredient in this article to overcome the curse of dimensionality in the numerical approximation of solution paths of BSDEs.

To briefly sketch the contribution of this article within this introductory section, we now present in the following result, \cref{h01} below, a special case of \cref{t26}, the main result of this article. Below \cref{h01} we explain in words the statement of \cref{h01} as well as the mathematical objects appearing in \cref{h01}.
\begin{theorem}\label{h01}
Let  $T, \delta \in (0,\infty)$, 
 $  \Theta = \bigcup_{ n \in \N }\! \Z^n$,
$f\in C^2( \R,\R)$, 
let 
$g_d\in C^1( \R^d,\R)$, $d\in\N$, satisfy $\sup_{d\in\N}\sup_{x=(x_1,x_2,\ldots, x_d)\in\R^d}\bigl(|f(x_1)|+|f'(x_1)|+|f''(x_1)|+|g_d(x)|+
\sum_{i=1}^{d}| \tfrac{\partial g_d}{\partial x_i}(x)|^2\bigr)<\infty$,
let
$(\Omega, \mathcal{F}, \P, (\F_t)_{t\in[0,T]})$ 
be a filtered probability space,
let $\unif^\theta\colon \Omega\to[0,1]$, $\theta\in \Theta$, be i.i.d.\ random variables,
assume for all $t\in (0,1)$ that $\P(\unif^{0}\le t)=t$,
let $W^{d,\theta}=(W^{d,\theta, 1},W^{d,\theta, 2},\ldots,W^{d,\theta, d})\colon [0,T]\times\Omega\to\R^d$, $d\in\N$, $\theta \in \Theta$, be 
independent standard 
$(\F_t)_{t\in[0,T]}$-Brownian motions,
assume that 
$
(\unif^\theta)_{\theta\in\Theta}$ and
$(W^{d,\theta})_{(d,\theta)\in\N\times \Theta}$
 are independent,
let
$ 
  {{U}}_{ n,M}^{d,\theta } \colon [0, T] \times \R^d \times \Omega \to \R
$, 
$d,M,n\in\N_0$, $\theta\in\Theta$, satisfy
for all 
$d,M \in \N$, $n\in \N_0$, $\theta\in\Theta $, 
$ t \in [0,T]$, $x\in\R^d $
 that 
{\small
\begin{equation}\label{eq:mlp_intro}
\begin{split}
&  {{U}}_{n,M}^{d,\theta}(t,x)
=(T-t)f(0)\1_{\N}(n)+
  \frac{\1_{\N}(n)}{M^n}
 \sum_{i=1}^{M^n} 
      {g}_d\bigl(x+W_{T-t}^{d,(\theta,0,-i)}\bigr)
 \\
& +
  \sum_{\ell=1}^{n-1}\Biggl[ \frac{(T-t)}{M^{n-\ell}}
    \sum_{i=1}^{M^{n-\ell}}
      \bigl(f\circ {{U}}_{\ell,M}^{d,(\theta,\ell,i)}-f\circ {{U}}_{\ell-1,M}^{d,(\theta,-\ell,i)}\bigr)
      \bigl(t+(T-t)\unif^{(\theta,\ell,i)},x+W_{(T-t)\unif^{(\theta,\ell,i)}}^{d,(\theta,\ell,i)}\bigr)
    \Biggr],
\end{split}
\end{equation}}%
let $\lfloor \cdot \rfloor_M \colon \R \to \R$, $ M \in \N $, and 
$\lceil \cdot \rceil_M \colon  \R \to \R$, $ M \in \N $, 
satisfy for all $M \in \N$, $t \in [0,T]$ that
$\lfloor t \rfloor_M = \max( ([0,t]\backslash \{T\}) \cap \{ 0, \frac{ T }{ M }, \frac{ 2T }{ M }, \ldots \} )$
and 
$\lceil t \rceil_M = \min(((t,\infty) \cup \{T\})\cap  \{ 0, \frac{ T }{ M }, \frac{ 2T }{ M }, \ldots \} )$,
let $\Yappr^{d,n,M}
\colon [0,T]\times\Omega\to\R $,
 $d,n,M\in\N$, 
satisfy  for all $d,n,M\in\N$, $t\in[0,T]$  that
\begin{align}\label{eq:approx_intro}
\Yappr^{d,n,M}_{t}
&= \sum_{\ell=0}^{n-1}\biggl[
\left[ \tfrac{ \lceil t \rceil_{M^{l+1}} - t }{ ( T / M^{ l + 1 } ) } \right]
U^{d,\ell}_{n-\ell,M}(\lfloor t \rfloor_{M^{l+1}}, W_{\lfloor t \rfloor_{M^{l+1}}}^{d,0})+
\left[\tfrac{ t-\lfloor t \rfloor_{M^{l+1}} }{ ( T / M^{ l + 1 } ) }\right]
U^{d,\ell}_{n-\ell,M}(\lceil t \rceil_{M^{l+1}}, W_{\lceil t \rceil_{M^{l+1}}}^{d,0})
\nonumber
\\
&\quad-\1_{\N}(\ell)\Bigl(
\left[ \tfrac{ \lceil t \rceil_{M^{l}} - t }{ ( T / M^{ l  } ) } \right]U^{d,\ell}_{n-\ell,M}(\lfloor t \rfloor_{M^{l}}, W^{d,0}_{\lfloor t \rfloor_{M^{l}}})+
\left[ \tfrac{ t-\lceil t \rceil_{M^{l}} }{ ( T / M^{ l  } ) } \right]
U^{d,\ell}_{n-\ell,M}(\lceil t \rceil_{M^{l}}, W^{d,0}_{\lceil t \rceil_{M^{l}}})
\Bigr)\biggr],
\end{align}
and for every 
$ d,n,M\in\N$
let
$\FEY{d,n,M}\in\N_0$ be the number of realizations of scalar random variables, the number of function evaluations of $f$, 
and the number of function evaluations of $g_d$
which are used to compute one realization of
$(\Yappr^{d,n,M}_{kT/M^n})_{k\in \{0,1,\ldots, M^n\}} $ (cf.\ \eqref{d01b} for a precise definition),
let $\mathbf{Y}^d=(Y^d,Z^{d,1},Z^{2,d},\ldots,Z^{d,d})\colon [0,T]\times\Omega\to \R^{d+1}
$, $d\in \N$, 
be $(\F_{t})_{t\in[0,T]}$-predictable stochastic processes,
assume for all $d\in \N$ that
$
\int_0^T  \E \bigl[|Y_s^d|+\textstyle\sum_{j=1}^{d}|Z_s^{d,j}|^2\bigr]\,ds<\infty
$,
and assume that for all $d\in \N$, $t\in[0,T]$ it holds $\P$-a.s. that
\begin{equation}\label{eq:bsde_intro}
Y^d_t=g_d(W^{d,0}_T)+\int_t^T f(Y^d_s)\,ds-\sum_{j=1}^{d}\int_t^T Z_s^{d,j}\, dW_s^{d,0,j}.
\end{equation}
Then
there exist 
$c\in \R$ and
$\sfN\colon \N \times (0,1]\to \N$
such that
 for all $d\in\N$, $\varepsilon\in  (0,1]$ it holds that
$\sup_{t\in [0,T]}
(\E[
|
\Yappr^{d,\sfN(d,\epsilon),\sfN(d,\epsilon)}_t-Y_t^d|^2])^{1/2}
\leq\epsilon
$ and
$
\FEY{d,\sfN(d,\epsilon),\sfN(d,\epsilon)}
\leq c
d^{c}\epsilon^{-(2+\delta)}.
$
\end{theorem}
\cref{h01} is an immediate consequence from \cref{cor:sec4} in \cref{sec:compcompl} below. 
\cref{cor:sec4}, in turn, follows from \cref{t26}, which is the main result of this article. 
In the following we add some comments on the mathematical objects appearing in \cref{h01} above. 

In \eqref{eq:bsde_intro} in \cref{h01} we specify the BSDEs whose solution processes 
we intend to approximate in \cref{h01}. 
The strictly positive real number $T \in (0,\infty)$ in the first line of \cref{h01} describes the time horizon of the BSDEs in \eqref{eq:bsde_intro}. 
The function $f \colon \R \to \R$ in the first line of \cref{h01}
specifies the driver (the nonlinearity) of the BSDEs in \eqref{eq:bsde_intro}. 
The quadrupel $(\Omega, \mathcal{F}, \P, (\F_t)_{t\in[0,T]})$ in the third line of \cref{h01}
is the filtered probability space on which the BSDEs in \eqref{eq:bsde_intro} are formulated. 
In \cref{h01} we do not assume that the filtererd probability space 
$(\Omega, \mathcal{F}, \P, (\F_t)_{t\in[0,T]})$ satisfies the \emph{usual conditions} in the sense that 
for all $t \in [0,T)$ it holds that $\{ A\in \mathcal F\colon \P(A)=0 \} \subseteq \F_t 
= \cap_{ s \in (t,T] } \F_s$.
The $(\F_{t})_{t\in[0,T]}$-predictable stochastic processes
$Y^d\colon [0,T]\times\Omega\to \R
$, $d\in \N$, 
in the last but fifth line 
of \cref{h01} are the solution processes of the BSDEs in \eqref{eq:bsde_intro}. 
 
In \eqref{eq:mlp_intro}--\eqref{eq:approx_intro} in \cref{h01} we specify 
the Monte Carlo-type approximation algorithm which we propose 
to approximate the solution processes of the BSDEs in \eqref{eq:bsde_intro}. 
To formulate the proposed Monte Carlo-type approximation algorithm 
in \eqref{eq:mlp_intro}--\eqref{eq:approx_intro} we need, roughly speaking, sufficiently many independent 
random quantities which are indexed over a sufficiently large index set. 
This sufficiently large index set is provided through 
the set  $  \Theta = \bigcup_{ n \in \N }\! \Z^n$ in the first line of \cref{h01}. 
The i.i.d.\ random variables $\unif^\theta\colon \Omega\to[0,1]$, $\theta\in \Theta$,
in the third line of \cref{h01}
and the independent standard $(\F_{t})_{t\in[0,T]}$-Brownian motions $W^{d,\theta}\colon [0,T]\times\Omega\to\R^d$, $d\in\N$, $\theta \in \Theta$, in the fourth line 
of \cref{h01} provide the random quantities which we employ to formulate 
the BSDEs in \eqref{eq:bsde_intro} and the proposed Monte Carlo-type approximation 
algorithm in \eqref{eq:mlp_intro}--\eqref{eq:approx_intro}. 

More formally, observe that the independent standard Brownian motions  
$W^{ d, 0 }\colon [0,T]\times\Omega\to\R^d$, $d\in\N$, 
in the fourth line of \cref{h01}  are the driving standard Brownian motions 
in the BSDEs in \eqref{eq:bsde_intro} 
and observe that 
the i.i.d.\ random variables $\unif^\theta\colon \Omega\to[0,1]$, $\theta\in \Theta$,
in the third line of \cref{h01}  
and the independent standard Brownian motions $W^{d,\theta}\colon [0,T]\times\Omega\to\R^d$, $d\in\N$, $\theta \in (\Theta \backslash \{0\})$,
in the fourth line of \cref{h01} are the random quantities 
which we use as random input sources to formulate  
the proposed Monte Carlo-type approximation algorithm in \eqref{eq:mlp_intro}--\eqref{eq:approx_intro}. 
Note that the assumption in third line of \cref{h01} that for all
$t\in (0,1)$ it holds that $\P(\unif^{0}\le t)=t$
ensures that for all $\theta \in \Theta$ it holds that 
$\unif^{ \theta }$ is an on $[0,1]$ continuous uniformly distributed random variable. 

The functions $g_d \colon \R^d \to \R$, $d \in \N$, 
in the first line of \cref{h01} and the independent standard Brownian motions 
$W^{ d, 0 }\colon [0,T]\times\Omega\to\R^d$, $d\in\N$, 
in the fourth line of \cref{h01} determine the terminal conditions of 
the BSDEs in \eqref{eq:bsde_intro}. More precisely, note that \eqref{eq:bsde_intro} in \cref{h01}
ensures that for all $d \in \N$ it holds $\P$-a.s.\ that 
$Y^d_T = g_d( W^{ d,0}_T )$. 
In \cref{h01} we assume that the driver $f \colon \R \to \R$ 
in the first line of \cref{h01}
and the functions $g_d \colon \R^d \to \R$, $d \in \N$, in the first line of 
\cref{h01} satisfy some regularity hypotheses. 
More formally, observe that the assumption $\sup_{d\in\N}\sup_{x=(x_1,x_2,\ldots, x_d)\in\R^d}\bigl(|f(x_1)|+|f'(x_1)|+|f''(x_1)|+|g_d(x)|+
\sum_{i=1}^{d}| \tfrac{\partial g_d}{\partial x_i}(x)|^2\bigr)<\infty$ in the second line 
of \cref{h01} assures that there exists a real number $\introkap \in \R$
such that for all $d \in \N$, $v \in \R$, $x = (x_1, x_2, \ldots, x_d) \in \R^d$
it holds that
$|f(v)| \leq \introkap$, $| f'(v) | \leq \introkap$, $|f''(v)| \leq \introkap$, 
$|g_d(x)|\leq \introkap$, and $\sum_{i=1}^{d}| \tfrac{\partial g_d}{\partial x_i}(x)|^2\le \introkap$.

The numbers $\FEY{d,n,M}\in\N_0$, $d, n, m \in \N$, in the first line below \eqref{eq:approx_intro} in \cref{h01} model the computational cost of the Monte Carlo-type approximation algorithm in \eqref{eq:mlp_intro}--\eqref{eq:approx_intro}.
More specifically, for every $d, n, M \in \N$ we have that 
$\FEY{d,n,M}$ specifies the sum of the number of realizations of one-dimensional random variables, of the number of 
function evaluations of $f \colon \R \to \R$, and of the number 
of function evaluations of $g_d \colon \R^d \to \R$ which are used to compute one realization of  
$(\Yappr^{d,n,M}_{kT/M^n})_{k\in \{0,1,\ldots, M^n\}} $ (cf.\ \eqref{d01b} for a precise definition).
Observe that \eqref{eq:approx_intro} in \cref{h01} ensures that for every $d, n, M \in \N$
we have that 
$(\Yappr^{d,n,M}_t)_{ t \in [0,T] }$ is the piecewise affine linear interpolation associated to 
$(\Yappr^{d,n,M}_{kT/M^n})_{k\in \{0,1,\ldots, M^n\}} $
 in sense that for all $k \in \{ 1, 2, ..., M^n \}$, $t \in [ \frac{(k-1)T}{M^n}, \frac{ k T }{ M^n } ]$ it holds that
$\Yappr^{d,n,M}_t = \frac{M^n}{T}\bigl[ ( \frac{kT}{M^n}-t ) \Yappr^{d,n,M}_{(k-1)T/M^n} + ( t-  \frac{(k-1)T}{M^n} ) \Yappr^{d,n,M}_{ k T / M^n }\bigr]$. 

\cref{h01} proves that the solution processes 
$Y^d\colon [0,T]\times\Omega\to \R
$, $d\in \N$, of the BSDEs in \eqref{eq:bsde_intro}
can be approximated by means 
of the Monte Carlo-type approximation algorithm 
in \eqref{eq:mlp_intro}--\eqref{eq:approx_intro} with a computational cost 
which grows at most polynomially in the dimension $d \in \N$ of the BSDE 
and up to an arbitrarily small polynomial order at most quadratically 
in the reciprocal $ \nicefrac{ 1 }{ \varepsilon } $ 
of the prescribed approximation accuracy $\varepsilon > 0$. 
The arbitrarily small polynomial order is described through 
the real number $\delta \in (0,\infty)$ in the first line of \cref{h01}.

In the following we also add some comments on shortcomings and possible generalizations of \cref{h01}. In particular, we observe that the driver $f \colon \R \to \R$ in the BSDEs in \eqref{eq:bsde_intro} does only depend on the solution processes $Y^d \colon [0,T]\times \Omega \to \R$, $d \in \N$, but not on the time variable $s \in [0,T]$, not on the driving Brownian motions 
$W^{ d, 0 } \colon[0,T]\times \Omega \to \R^d$, $d \in \N$, and also not on
the stochastic processes $Z^{ d, j } \colon [0,T]\times \Omega \to \R$, $j\in \{1,2,\ldots, d\}$, $d\in \N$.
However, in the more general result in \cref{t26} in \cref{sec:compltwoplus} below the drivers of the BSDEs under consideration 
do additionally also depend on the time variable $s \in [0,T]$ and 
on the driving Bronwnian motions $W^{ d, 0 } \colon[0,T]\times \Omega \to \R^d$, $d \in \N$. We refer to \eqref{t29b} in \cref{t26} below for details. The dependence 
of the drivers of the BSDEs under considerations on 
the stochastic processes $Z^{ d, j } \colon [0,T]\times \Omega \to \R$, $j\in \{1,2,\ldots, d\}$, $d\in \N$, is not covered within this article 
and the numerical approximation of 
the stochastic processes $Z^{ d, j } \colon [0,T]\times \Omega \to \R$, $j\in \{1,2,\ldots, d\}$, $d\in \N$, is also not covered within this article 
but the arguments revealed in this article together
with the arguments in the article \cite{hjk2019overcoming} allow also to overcome the curse of dimensionality in these more general cases of BSDEs. 

In \cref{h01} we also use a rather restrictive regularity hypothesis on the driver $f \colon \R \to \R$ and the functions $g_d \colon \R^d \to \R$, $d\in \N$, in the sense that there exists $\introkap \in \R$ such that for all $d \in \N$, $v \in \R$, $x = (x_1, x_2, \ldots, x_d) \in \R^d$
it holds that
$|f(v)| \leq \introkap$, $| f'(v) | \leq \introkap$, $|f''(v)| \leq \introkap$, 
$|g_d(x)|\leq \introkap$, and $\sum_{i=1}^{d}| \tfrac{\partial g_d}{\partial x_i}(x)|^2\le \introkap$. In the more general result in \cref{t26} in \cref{sec:compltwoplus} below this hypothesis is replaced by suitable more general Lipschitz-type assumptions. We refer to \eqref{t02b}--\eqref{m11b} in \cref{t26} below for details.

The remainder of this article is organized as follows. In \cref{sec:2} below we establish upper bounds for a generalized norm of the difference between a vector space valued process and appropriate multi-grid approximations for this process. A key aspect in the derivation of the Monte Carlo-type approximation algorithm in \eqref{eq:mlp_intro}--\eqref{eq:approx_intro} in \cref{h01} is, roughly speaking, to reformulate the solutions of the BSDEs in \eqref{eq:bsde_intro} as solutions of appropriate stochastic fixed-point equations (SFPEs) associated to the BSDEs in  \eqref{eq:bsde_intro} and in \cref{sec:3} below we establish existence, uniqueness, and H\"older continuity properties for solutions of precisely such SFPEs. In \cref{sec:compltwoplus} below we establish upper bounds for appropriate H\"older seminorms of the difference between the solutions of such SFPEs and suitable MLP approximations for such SFPEs. In \cref{sec:compcompl} below we combine the findings from Sections \ref{sec:2} and \ref{sec:compltwoplus} to provide a computational complexity analysis for the Monte Carlo-type approximation algorithm in \eqref{eq:mlp_intro}--\eqref{eq:approx_intro} and, thereby, we also prove \cref{h01} above.


\section{Error analysis for multi-grid approximations}\label{sec:2}

A central aspect in the derivation of the Monte Carlo-type approximation algorithm for BSDEs in \eqref{eq:mlp_intro}--\eqref{eq:approx_intro} in \cref{h01} in \cref{sec:intro} above is, roughly speaking, to approximate the exact solution of the BSDE under consideration by means of appropriate multi-grid approximations on coarser and coarser time grids and, then, to exploit suitable uniform temporal regularity properties for the employed multi-grid approximations. 

In \cref{n002} in this section we formulate this approach in an abstract setting and in \cref{n002} we also establish explicit upper bounds for a generalized error norm of the difference between a vector space valued process (which we think of as the solution process of the considered BSDE) and appropriate multi-grid approximations for this process. 

Our approach is based on the multilevel method in the articles Heinrich
 \cite{h98,Heinrich01}. In these references Heinrich proposed and formulated the multilevel method in the context of Monte Carlo approximations of certain parameter-dependent integrals (see also Heinrich \& Sindambiwe~\cite{heinrich1999monte}).

Our proof of \cref{n002} employs the essentially well-known error estimate for piecewise affine linear interpolation functions in \cref{n001} and the essentially well-known H\"older continuity result for piecewise affine linear interpolation functions in \cref{n002b}. \cref{n001} is, e.g., a slight extension of Cox et al.~\cite[Lemma 2.2]{cox2016convergence} and \cref{n002b} is, e.g., a slight extension of Cox et al.~\cite[Lemma 2.5]{cox2016convergence}.

\begin{lemma}[]\label{n001}
\renewcommand{\epsilon}{\varepsilon}
\newcommand{\restricted}{\upharpoonright}
\renewcommand{\intOp}{\mathscr{L}}
\newcommand{\barY}{\mathscr{Y}}
\newcommand{\norm}[1]{{\left\lVert #1\right\rVert}}
\newcommand{\norms}[1]{{\bigl\lVert #1 \bigr\rVert}}
\newcommand{\Xlem}{X}
Let $V$ be an $\R$-vector space,
let $\norm{\cdot}\colon V\to[0,\infty]$ satisfy for all
$v,w\in V$, $\mathscr{v},\mathscr{w}\in\R$ 
with $\norm{v}+\norm{w}<\infty$
that
$\norm{\mathscr{v} v+\mathscr{w} w}\leq|\mathscr{v}| \norm{v}+|\mathscr{w}|\norm{w}$,
let $T, \alpha \in(0,\infty)$, 
$m\in\N$,
$\tau_0, \tau_1,\ldots, \tau_m\in \R$
satisfy $0 = \tau_{ 0 } 
< \tau_{ 1 } < \ldots < \tau_{  m } = T$, and let
$x=(x_t)_{t\in [0,T]}\colon [0,T] \to V$ and 
$\Xlem=(\Xlem_t)_{t\in [0,T]} \colon [0,T]\to V$
satisfy for all $k \in \{ 1, 2, \ldots, m \}$, $t \in [ \tau_{ k-1 }, \tau_{ k } ] $
that
$
\Xlem_t= ({\tau_{k}-\tau_{k-1}})^{-1}
[(\tau_{k}-t)x_{\tau_{k-1}}+(t-\tau_{k-1})x_{\tau_k}]
$.
Then
\begin{equation}
\begin{split}
&\sup_{t\in [0,T]} \norms{ \Xlem_t-x_t}\le 2^{-\!\min\{3,\alpha\}} \left[\max_{ k \in \{ 1, 2, \ldots, m \}} \lvert \tau_{ k } - \tau_{ k-1 } \rvert ^\alpha\right] \left[
\sup_{r,s\in [0,T], \, r\neq s} \frac{\norm{x_r-x_s}}{|r-s|^\alpha}\right].
\label{n06}
\end{split}
\end{equation}
\end{lemma}
\begin{proof}[Proof of \cref{n001}]
\renewcommand{\epsilon}{\varepsilon}
\newcommand{\barY}{\mathscr{Y}}
\newcommand{\norm}[1]{{\left\lVert #1\right\rVert}}
\newcommand{\norms}[1]{{\bigl\lVert #1 \bigr\rVert}}
\newcommand{\vlem}{\mathfrak{y}}
\newcommand{\Xlem}{X}
Throughout this proof assume without loss of generality that for all $s,t\in [0,T]$ with $s\neq t$ it holds that $\norm{x_s-x_t}<\infty$.
Note that 
for all $k \in \{ 1, 2, \ldots, m \}$, $t \in ( \tau_{ k-1 }, \tau_{ k } ) $ it holds 
that
\begin{equation}
\begin{split}
\Xlem_t-x_t&=
\left[
\frac{(\tau_{k}-t)x_{\tau_{k-1}}+(t-\tau_{k-1})x_{\tau_{k}}}{{\tau_{k}-\tau_{k-1}}}
\right]
-x_t\\
&= 
\frac{(\tau_{k}-t)(x_{\tau_{k-1}}-x_t)+(t-\tau_{k-1})(x_{\tau_{k}}-x_t)}{{\tau_{k}-\tau_{k-1}}}
\\
&=
\left[\frac{(t-\tau_{k-1})(\tau_{k}-t)^\alpha}{\tau_{k}-\tau_{k-1}} \right] \left[ \frac{x_{\tau_{k}}-x_t}{(\tau_{k}-t)^\alpha} \right]
-
\left[\frac{(\tau_{k}-t)(t-\tau_{k-1})^\alpha}{\tau_{k}-\tau_{k-1}} \right] \left[ \frac{x_t-x_{\tau_{k-1}}}{(t-\tau_{k-1})^\alpha} \right].
\end{split}
\end{equation}
The assumption that for all
$v,w\in V$, $\mathscr{v},\mathscr{w}\in\R$ 
with $\norm{v}+\norm{w}<\infty$ it holds
that
$\norm{\mathscr{v} v+\mathscr{w} w}\leq|\mathscr{v}| \norm{v}+|\mathscr{w}|\norm{w}$ hence ensures that  
for all $k \in \{ 1, 2, \ldots, m \}$, $t \in [ \tau_{ k-1 }, \tau_{ k } ] $ it holds 
that
\begin{align}
\norms{\Xlem_t-x_t}
&\leq 
\left(
\left[\frac{(t-\tau_{k-1})(\tau_{k}-t)^\alpha}{\tau_{k}-\tau_{k-1}} \right] 
+
\left[\frac{(\tau_{k}-t)(t-\tau_{k-1})^\alpha}{\tau_{k}-\tau_{k-1}} \right]
\right)
 \left[
\sup_{r,s\in [0,T], \, r\neq s} \frac{\norm{x_r-x_s}}{|r-s|^\alpha}\right] \nonumber
\\
&=
\left(
\left[\frac{t-\tau_{k-1}}{\tau_{k}-\tau_{k-1}}\right]
\left[\frac{
\tau_{k}-t}{\tau_{k}-\tau_{k-1}} \right]^\alpha 
+
\left[\frac{\tau_{k}-t}{\tau_{k}-\tau_{k-1}}\right]
\left[\frac{t-\tau_{k-1}}{\tau_{k}-\tau_{k-1}} \right]^\alpha
\right)\nonumber
\\\label{n06c}
&\quad \cdot [\tau_{k}-\tau_{k-1}]^\alpha
 \left[
\sup_{r,s\in [0,T], \, r\neq s} \frac{\norm{x_r-x_s}}{|r-s|^\alpha}\right]
\\
&=
\left(
\left[\frac{t-\tau_{k-1}}{\tau_{k}-\tau_{k-1}}\right]
\left[1-\left(\frac{t-\tau_{k-1}}{\tau_{k}-\tau_{k-1}}\right)\right]^\alpha 
+
\left[1-\left(\frac{t-\tau_{k-1}}{\tau_{k}-\tau_{k-1}}\right)\right]
\left[\frac{t-\tau_{k-1}}{\tau_{k}-\tau_{k-1}} \right]^\alpha
\right)\nonumber
\\
&\quad \cdot
[\tau_{k}-\tau_{k-1}]^\alpha
 \left[
\sup_{r,s\in [0,T], \, r\neq s} \frac{\norm{x_r-x_s}}{|r-s|^\alpha}\right]\nonumber
\\
&\le \left[\max_{ l \in \{ 1, 2, \ldots, m \}} | \tau_{ l } - \tau_{ l-1 } |^\alpha \right]  \left[ \sup_{z\in [0,1]}(z(1-z)^\alpha+(1-z)z^\alpha) \right]\left[
\sup_{r,s\in [0,T], \, r\neq s} \frac{\norm{x_r-x_s}}{|r-s|^\alpha}\right].\nonumber
\end{align}
Next observe that the fact that for all $z\in [0,1]$ it holds that $z(1-z)\le 2^{-2}$ and Jensen's inequality imply that for all $z\in [0,1]$ it holds that
\begin{equation}\label{n06d} 
\begin{split}
&\1_{(0,1]}(\alpha)\bigl(z(1-z)^\alpha+(1-z)z^\alpha\bigr)\le \1_{(0,1]}(\alpha)\bigl(z(1-z)+(1-z)z\bigr)^\alpha\\
&=\1_{(0,1]}(\alpha)2^\alpha \bigl(z(1-z)\bigr)^\alpha
\le \1_{(0,1]}(\alpha)2^{-\alpha}.
\end{split}
\end{equation}
In addition, note the fact that for all $z\in [0,1]$ it holds that $z(1-z)\le 2^{-2}$ and Jensen's inequality imply that for all $z\in [0,1]$ it holds that
\begin{equation}\label{n06e} 
\begin{split}
&\1_{(1,2]}(\alpha)\bigl(z(1-z)^\alpha+(1-z)z^\alpha\bigr)= \1_{(1,2]}(\alpha)(2z(1-z))\left( \tfrac{(1-z)^{\alpha-1}}{2}+\tfrac{z^{\alpha-1}}{2}\right)\\
&\le \1_{(1,2]}(\alpha)\left(\tfrac{1}{2}\right)\left[ \tfrac{(1-z)}{2}+\tfrac{z}{2}\right]^{\alpha-1}=\1_{(1,2]}(\alpha)2^{-\alpha}.
\end{split}
\end{equation}
Next observe that the fact that for all $z\in [0,1]$ it holds that $z(1-z)\le 2^{-2}$,  
and the fact that for all $c\in [0,1]$ it holds that $[0,\nicefrac{1}{4}] \ni y \mapsto y(1-2y)^c \in \R$ is non-decreasing,
and Jensen's inequality imply that for all $z\in [0,1]$ it holds that
\begin{align}
&\1_{(2,\infty)}(\alpha)\bigl(z(1-z)^\alpha+(1-z)z^\alpha\bigr)\le 
\1_{(2,\infty)}(\alpha)\bigl(z(1-z)^{\min\{\alpha,3\}}+(1-z)z^{\min\{\alpha,3\}}\bigr)\\
\nonumber
&= 
\1_{(2,\infty)}(\alpha)(z(1-z))\bigl((1-z)[(1-z)^{\min\{\alpha,3\}-2}]+z[z^{\min\{\alpha,3\}-2}]\bigr)\\
\nonumber
&\le 
\1_{(2,\infty)}(\alpha)(z(1-z))\bigl((1-z)^2+z^2\bigr)^{\min\{\alpha,3\}-2}
=
\1_{(2,\infty)}(\alpha)(z(1-z))\bigl(1-2z(1-z)\bigr)^{\min\{\alpha,3\}-2}
\\
\nonumber
&\le 
\1_{(2,\infty)}(\alpha)\max_{y\in [0,\nicefrac{1}{4}] }\Bigl[y\bigl(1-2y\bigr)^{\min\{\alpha,3\}-2}\Bigr]
=\1_{(2,\infty)}(\alpha)\Bigl[\tfrac{1}{4}\bigl(1-\tfrac{1}{2}\bigr)^{\min\{\alpha,3\}-2}\Bigr]
=\1_{(2,\infty)}(\alpha)2^{-\!\min\{\alpha,3\}}.
\end{align}
Combining this with \eqref{n06d} and \eqref{n06e} demonstrates that 
$\sup_{z\in [0,1]}(z(1-z)^\alpha+(1-z)z^\alpha)\le 2^{-\!\min\{3,\alpha\}}$.
This and \eqref{n06c} show that 
\begin{equation}
\begin{split}
&\sup_{t\in [0,T]} \norms{ \Xlem_t-x_t}
\le 2^{-\!\min\{3,\alpha\}} \left[\max_{ k \in \{ 1, 2, \ldots, m \}} | \tau_{ k } - \tau_{ k-1 } |^\alpha \right] \left[
\sup_{r,s\in [0,T], \, r\neq s} \frac{\norm{x_r-x_s}}{|r-s|^\alpha}\right].
\label{n06h}
\end{split}
\end{equation}
The proof of \cref{n001} is thus complete.
\end{proof}

\begin{lemma}[]\label{n002b}
\renewcommand{\epsilon}{\varepsilon}
\newcommand{\restricted}{\upharpoonright}
\renewcommand{\intOp}{\mathscr{L}}
\newcommand{\barY}{\mathscr{Y}}
\newcommand{\norm}[1]{{\left\lVert #1\right\rVert}}
\newcommand{\norms}[1]{{\bigl\lVert #1 \bigr\rVert}}
\newcommand{\Xlem}{X}
Let $V$ be an $\R$-vector space,
let $\norm{\cdot}\colon V\to[0,\infty]$ satisfy for all
$v,w\in V$, $\mathscr{v},\mathscr{w}\in\R$ 
with $\norm{v}+\norm{w}<\infty$
that
$\norm{\mathscr{v} v+\mathscr{w} w}\leq|\mathscr{v}| \norm{v}+|\mathscr{w}|\norm{w}$,
let $T \in(0,\infty)$,
$\alpha\in (0,1]$,
$m\in\N$,
$\tau_0, \tau_1,\ldots, \tau_m\in \R$
satisfy $0 = \tau_{ 0 } 
< \tau_{ 1 } < \ldots < \tau_{  m } = T$, and let
$x=(x_t)_{t\in [0,T]}\colon [0,T] \to V$ and 
$\Xlem=(\Xlem_t)_{t\in [0,T]} \colon [0,T]\to V$
satisfy for all $k \in \{ 1, 2, \ldots, m \}$, $t \in [ \tau_{ k-1 }, \tau_{ k } ] $
that
$
\Xlem_t= ({\tau_{k}-\tau_{k-1}})^{-1}
[(\tau_{k}-t)x_{\tau_{k-1}}+(t-\tau_{k-1})x_{\tau_k}]
$.
Then
\begin{equation}
\begin{split}
\left[\sup_{s,t\in [0,T], \, s\neq t}\frac{\norm{X_s-X_t}}{|s-t|^\alpha}\right]  \le \left[ \sup_{s,t\in [0,T], \, s\neq t}\frac{\norm{x_s-x_t}}{|s-t|^\alpha}\right].
\label{n06b}
\end{split}
\end{equation}
\end{lemma}
\begin{proof}[Proof of \cref{n002b}]
\renewcommand{\epsilon}{\varepsilon}
\newcommand{\barY}{\mathscr{Y}}
\newcommand{\norm}[1]{{\left\lVert #1\right\rVert}}
\newcommand{\norms}[1]{{\bigl\lVert #1 \bigr\rVert}}
\newcommand{\vlem}{\mathfrak{y}}
\newcommand{\Xlem}{X}
Throughout this proof assume without loss of generality that for all $s,t\in [0,T]$ with $s\neq t$ it holds that $\norm{x_s-x_t}<\infty$
and
 let $n\colon [0,T]\to \N$ and $\rho \colon [0,T] \to [0,1]$ satisfy for all $t\in [0,T]$ that
\begin{equation}\label{eq:def_n_rho}
n(t)=\min\{k\in \{1,2,\ldots,m\}\colon \tau_k \ge t \} \qquad \text{and} \qquad \rho(t)=\frac{t-\tau_{n(t)-1}}{\tau_{n(t)}-\tau_{n(t)-1}}.
\end{equation}
Note that \eqref{eq:def_n_rho} ensures that for all $t\in [0,T]$ it holds that
\begin{equation}\label{eq:Xlemtwotwo}
X_t=(1-\rho(t))x_{\tau_{n(t)-1}}+\rho(t)x_{\tau_{n(t)}}=x_{\tau_{n(t)-1}}+\rho(t)(x_{\tau_{n(t)}}-x_{\tau_{n(t)-1}}).
\end{equation}
The fact that for all $v,w\in V$, $\mathscr{v},\mathscr{w}\in\R$ 
with $\norm{v}+\norm{w}<\infty$ it holds 
that
$\norm{\mathscr{v} v+\mathscr{w} w}\leq|\mathscr{v}| \norm{v}+|\mathscr{w}|\norm{w}$ 
hence ensures that for all $t_1,t_2\in [0,T]$ with $t_1<t_2$ and $n(t_1)=n(t_2)$ it holds that
\begin{align}\label{n0815}
&\nonumber \norm{X_{t_1}-X_{t_2}}=
\bigl \lVert (\rho(t_1)-\rho(t_2))(x_{\tau_{n(t_1)}}-x_{\tau_{n(t_1)-1}})\bigr \rVert\\
& =\bigl \lVert (\rho(t_1)-\rho(t_2))(x_{\tau_{n(t_1)}}-x_{\tau_{n(t_1)-1}})+0 (x_T-x_0)\bigr \rVert
\nonumber 
\le \lvert \rho(t_1)-\rho(t_2) \rvert \bigl \lVert x_{\tau_{n(t_1)}}-x_{\tau_{n(t_1)-1}}\bigr \rVert\\
\nonumber
&\le 
\left[
\lvert \rho(t_1)-\rho(t_2) \rvert 
\right]
\left[
\lvert \tau_{n(t_1)}-\tau_{n(t_1)-1} \rvert^\alpha 
\right]
\left[
\sup_{s,t \in [0,T],\, s\neq t}\tfrac{ \lVert x_s-x_t  \rVert}{|s-t|^\alpha}
\right]\\
&=
\lvert \rho(t_1)-\rho(t_2) \rvert 
^{1-\alpha}
\left[
\lvert \rho(t_1)-\rho(t_2) \rvert 
\lvert \tau_{n(t_1)}-\tau_{n(t_1)-1} \rvert
\right]^\alpha 
\left[
\sup_{s,t \in [0,T],\, s\neq t}\tfrac{ \lVert x_s-x_t  \rVert}{|s-t|^\alpha}
\right]\\
\nonumber
&\le 
\left[
\lvert \rho(t_1)-\rho(t_2) \rvert 
\lvert \tau_{n(t_1)}-\tau_{n(t_1)-1} \rvert
\right]^\alpha 
\left[
\sup_{s,t \in [0,T],\, s\neq t}\tfrac{ \lVert x_s-x_t  \rVert}{|s-t|^\alpha}
\right]
=
\lvert t_1-t_2 \rvert 
^\alpha 
\left[
\sup_{s,t \in [0,T],\, s\neq t}\tfrac{ \lVert x_s-x_t  \rVert}{|s-t|^\alpha}
\right].
\end{align}
Moreover, observe that \eqref{eq:Xlemtwotwo} and the fact that for all $v,w\in V$, $\mathscr{v},\mathscr{w}\in\R$ 
with $\norm{v}+\norm{w}<\infty$ it holds 
that
$\norm{\mathscr{v} v+\mathscr{w} w}\leq|\mathscr{v}| \norm{v}+|\mathscr{w}|\norm{w}$ ensure that for all $t_1,t_2\in [0,T]$ with $n(t_1)<n(t_2)$ it holds that
\begin{align}
&\norm{X_{t_1}-X_{t_2}}=
\bigl \lVert \bigl[(1-\rho(t_1))x_{\tau_{n(t_1)-1}}+\rho(t_1)x_{\tau_{n(t_1)}}\bigr]
-
\bigl[ 
(1-\rho(t_2))x_{\tau_{n(t_2)-1}}+\rho(t_2)x_{\tau_{n(t_2)}}
\bigr]
 \bigr \rVert 
\nonumber  
 \\
 &\le 
 (1-\rho(t_1))(1-\rho(t_2))\bigl \lVert
x_{\tau_{n(t_1)-1}}-x_{\tau_{n(t_2)-1}}
 \bigr \rVert
 +
 \rho(t_1)\rho(t_2)\bigl \lVert
x_{\tau_{n(t_1)}}-x_{\tau_{n(t_2)}}
 \bigr \rVert\\
 \nonumber
 &+
 (1-\rho(t_1))\rho(t_2)\bigl \lVert
x_{\tau_{n(t_1)-1}}-x_{\tau_{n(t_2)}}
 \bigr \rVert
 +
 \rho(t_1)(1-\rho(t_2))\bigl \lVert
x_{\tau_{n(t_1)}}-x_{\tau_{n(t_2)-1}}
 \bigr \rVert\\
 \nonumber
 &\le 
 \left[
\sup_{s,t \in [0,T],\, s\neq t}\tfrac{ \lVert x_s-x_t  \rVert}{|s-t|^\alpha}
\right]
 \bigl [
 (1-\rho(t_1))(1-\rho(t_2))\lvert
\tau_{n(t_1)-1}-\tau_{n(t_2)-1}
  \rvert^\alpha
 +
 \rho(t_1)\rho(t_2)\lvert
\tau_{n(t_1)}-\tau_{n(t_2)}
  \rvert^\alpha\\
  \nonumber 
 &+
 (1-\rho(t_1))\rho(t_2) \lvert
\tau_{n(t_1)-1}-\tau_{n(t_2)}
 \rvert^\alpha 
 +
 \rho(t_1)(1-\rho(t_2)) \lvert
 \tau_{n(t_1)}-\tau_{n(t_2)-1}
  \rvert^\alpha
 \bigr].
 \end{align}
The fact that the function $(-\infty,0]\ni z \mapsto |z|^\alpha \in \R$ is concave hence shows that for all $t_1,t_2\in [0,T]$ with $n(t_1)<n(t_2)$ it holds that
\begin{align}
&\norm{X_{t_1}-X_{t_2}} \nonumber \\
\nonumber
 &\le 
 \left[
\sup_{s,t \in [0,T],\, s\neq t}\tfrac{ \lVert x_s-x_t  \rVert}{|s-t|^\alpha}
\right]
 \bigl \lvert
 (1-\rho(t_1))(1-\rho(t_2))(
\tau_{n(t_1)-1}-\tau_{n(t_2)-1}
  )
 +
 \rho(t_1)\rho(t_2)(
\tau_{n(t_1)}-\tau_{n(t_2)}
  )\\
 &+
 (1-\rho(t_1))\rho(t_2) (
\tau_{n(t_1)-1}-\tau_{n(t_2)}
)
 +
 \rho(t_1)(1-\rho(t_2))(
 \tau_{n(t_1)}-\tau_{n(t_2)-1}
  )
 \bigr\rvert^\alpha \\
 \nonumber
 &=
 \left[
\sup_{s,t \in [0,T],\, s\neq t}\tfrac{ \lVert x_s-x_t  \rVert}{|s-t|^\alpha}
\right]
 \bigl \lvert  [\tau_{n(t_1)-1}+\rho(t_1)(\tau_{n(t_1)}-\tau_{n(t_1)-1})]
 -[
 \tau_{n(t_2)-1}+\rho(t_2)(\tau_{n(t_2)}-\tau_{n(t_2)-1})
 ]
  \bigr \rvert ^\alpha\\
 \nonumber
 &= \left[
\sup_{s,t \in [0,T],\, s\neq t}\tfrac{ \lVert x_s-x_t  \rVert}{|s-t|^\alpha}
\right] |t_1-t_2|^\alpha.
 \end{align}
 Combining this and \eqref{n0815} proves \eqref{n06b}. The proof of \cref{n002b} is thus complete.
\end{proof}

\begin{lemma}[]\label{n002}
\renewcommand{\epsilon}{\varepsilon}
\newcommand{\restricted}{\upharpoonright}
\renewcommand{\intOp}{\mathscr{L}}
\newcommand{\barY}{\mathscr{Y}}
\newcommand{\norm}[1]{{\left\lVert #1\right\rVert}}
\newcommand{\norms}[1]{{\bigl\lVert #1 \bigr\rVert}}
Let $V$ be an $\R$-vector space,
let $\norm{\cdot}\colon V\to[0,\infty]$ satisfy for all
$v,w\in V$, $\mathscr{v},\mathscr{w}\in\R$ 
with $\norm{v}+\norm{w}<\infty$
that
$\norm{\mathscr{v} v+\mathscr{w} w}\leq|\mathscr{v}| \norm{v}+|\mathscr{w}|\norm{w}$,
let $T \in(0,\infty)$, $\alpha\in (0,1]$, 
$n\in\N$,
$m_1, m_2, \ldots, m_n \in \N$, let $\tau_{ l, k } \in \R$, $k \in \{ 0, 1, \ldots, m_l \}$, $l \in \{ 1, 2, \ldots, n \}$, satisfy for all $l \in \{ 1, 2, \ldots, n \}$ that $0 = \tau_{ l, 0 } 
< \tau_{ l, 1 } < \ldots < \tau_{ l, m_l } = T$ and 
$
( \cup_{ i = 0 }^{ m_l } \{ \tau_{ l, i } \} ) \subseteq ( \cup_{i=0}^{ m_{l+1} } \{\tau_{ l+1, i }\} )
$,
let $\intOp_l \colon V^{ [0,T] } \to V^{ [0,T] }$, $l \in \N$, satisfy
for all $l \in \{ 1, 2, \ldots , n \}$, $k \in \{ 1, 2, \ldots, m_l \}$, $t \in [ \tau_{ l,k-1 }, \tau_{ l,k } ] $, $y = ( y_t )_{ t \in [0,T] } \colon [0,T] \to V$ that 
$(\intOp_l(y))(t)=({\tau_{l,k}-\tau_{l,k-1}})^{-1}
[(\tau_{l,k}-t)y_{\tau_{l,k-1}}+(t-\tau_{l,k-1})y_{\tau_{l,k}}]$,
and let $Y^\ell=(Y^\ell_t)_{t\in[0,T]}\colon [0,T]\to V$,
 $\ell\in \N_0$,  and
$\barY =(\barY_t)_{t\in[0,T]}\colon [0,T]\to V$ satisfy
\begin{equation}\label{n05}
\barY
= \intOp_1(Y^n)
+
\sum_{\ell=1}^{n-1}\left[ \intOp_{l+1}(Y^{n-l})
-\intOp_{l}(Y^{n-l})
\right]
.
\end{equation}
Then
\begin{align}
\sup_{t\in[0,T]}\norm{\barY_t-Y^0_t}
&\leq  \max_{k\in \{0,1,\ldots m_1\}}\norms{Y^n_{\tau_{1,k}}-Y^0_{\tau_{1,k}}}
+ \left[\max_{ k \in \{ 1, 2, \ldots, m_{n} \}} \tfrac{| \tau_{ n,k } - \tau_{ n,k-1 } |^\alpha}{2^{\alpha}}\right]\left[\sup_{t,s\in[0,T],\, t\neq s}\tfrac{\norm{Y_t^0-Y_s^0}}{|t-s|^\alpha}\right]\nonumber \\
&\quad
+\sum_{l=1}^{n-1}\left[\max_{ k \in \{ 1, 2, \ldots, m_{l} \}} \tfrac{| \tau_{ l,k } - \tau_{ l,k-1 } |^\alpha}{2^{\alpha}}\right] \left[
\sup_{t,s\in [0,T], \, t\neq s} \tfrac{\norm{(Y^{n-l}_t-Y^0_t)-(Y^{n-l}_s-Y^0_s)}}{|t-s|^\alpha}\right].
\end{align}
\end{lemma}
\begin{proof}[Proof of \cref{n002}]
\renewcommand{\epsilon}{\varepsilon}
\renewcommand{\intOp}{\mathscr{L}}
\newcommand{\barY}{\mathscr{Y}}
\newcommand{\norm}[1]{{\left\lVert #1\right\rVert}}
\newcommand{\norms}[1]{{\bigl\lVert #1 \bigr\rVert}}
\newcommand{\vlem}{\mathfrak{y}}
Throughout this proof
let $\varepsilon_l\in \R$, $l\in \{1,2,\ldots, n\}$, satisfy for all $l\in \{1,2,\ldots, n\}$ that 
$\varepsilon_l=\max_{ k \in \{ 1, 2, \ldots, m_{l} \}} | \tau_{ l,k } - \tau_{ l,k-1 } |$.
Observe that 
for all $l\in \{0,1,\ldots,n\}$, $y = ( y_t )_{ t \in [0,T] } \colon [0,T] \to V$
it holds that
\begin{equation}\label{n060}
\sup_{t\in[0,T]}
\norm{(\intOp_l(y))(t)}
\leq
\max_{k\in \{0,1,\ldots, m_l\}} \norm{y_{\tau_{l,k}}}.
\end{equation}
Next note that \eqref{n05} and the fact that 
$
\intOp_n (Y^0)= \intOp_{1} (Y^0)+
\sum_{l=1}^{n-1}
\left[ \intOp_{l+1} (Y^0)
- \intOp_l(Y^0)
\right]
$
demonstrate that
\begin{align}
&\barY - \intOp_n (Y^0)\nonumber\\
&=
\left[
\intOp_1(Y^n)
+
\sum_{\ell=1}^{n-1}\left[ \intOp_{l+1}(Y^{n-l})
-\intOp_{l}(Y^{n-l})
\right]
\right]
-
\left[
 \intOp_{1} (Y^0)+
\sum_{l=1}^{n-1}
\left[ \intOp_{l+1} (Y^0)
- \intOp_l(Y^0)
\right]
\right]\nonumber\\
&=
\intOp_{1}(Y^n-Y^0)
+
\sum_{l=1}^{n-1}
\left[
\intOp_{l+1}(Y^{n-l}-Y^0)
-
\intOp_l(Y^{n-l}-Y^0)
\right].
\end{align}
This, the fact that for all $v, w \in V$ it holds that 
$\norm{v+w} \leq \norm{v} + \norm{v}$, and \eqref{n060} ensure that
\begin{align}\label{n007}
&\sup_{t\in[0,T]}\norm{ \barY_t - (\intOp_n(Y^0))(t)}\\
\nonumber
&\leq \sup_{t\in[0,T]}\norm{
(\intOp_1(Y^n-Y^0))(t)}
+\sum_{l=1}^{n-1}\left[\sup_{t\in[0,T]}
\norm{
\bigl(
\intOp_{l+1}(Y^{n-l}-Y^0)\bigr)(t)
-
\bigl(\intOp_l(Y^{n-l}-Y^0)\bigr)(t)
}\right]\\ 
&\le \max_{k\in \{0,1,\ldots m_1\}}\norms{Y^n_{\tau_{1,k}}-Y^0_{\tau_{1,k}}}+\sum_{l=1}^{n-1}\left[\sup_{t\in[0,T]}
\norm{
\bigl(\intOp_{l+1}(Y^{n-l}-Y^0)\bigr)(t)
-
\bigl(\intOp_l(Y^{n-l}-Y^{0})\bigr)(t)
}\right].\nonumber
\end{align}
Moreover, note that for all $l\in \{1,2,\ldots,n-1\}$, $y = ( y_t )_{ t \in [0,T] } \colon [0,T] \to V$, $i\in \{0,1,\ldots, m_{l+1}\}$
it holds that $(\intOp_{l+1}(y))(\tau_{l+1,i})=y_{\tau_{l+1,i}}$. The assumption that for all $l\in \{1,2,\ldots,n-1\}$ it holds that 
$
( \cup_{ i = 0 }^{ m_l } \{ \tau_{ l, i } \} ) \subseteq ( \cup_{i=0}^{ m_{l+1} } \{\tau_{ l+1, i }\} )
$
therefore implies that
for all $l\in \{1,2,\ldots,n-1\}$, $y = ( y_t )_{ t \in [0,T] } \colon [0,T] \to V$, $i\in \{0,1,\ldots, m_{l}\}$
it holds that $(\intOp_{l+1}(y))(\tau_{l,i})=y_{\tau_{l,i}}$.
This proves that
for all $l\in \{1,2,\ldots,n-1\}$, $y = ( y_t )_{ t \in [0,T] } \colon [0,T] \to V$ it holds that
$\intOp_{l}(y)=\intOp_{l}(\intOp_{l+1}(y))$.
This and \eqref{n007} ensure that
\begin{equation}
\begin{split}
&\sup_{t\in[0,T]}\norm{ \barY_t - (\intOp_n(Y^0))(t)}
\leq \max_{k\in \{0,1,\ldots m_1\}}\norms{Y^n_{\tau_{1,k}}-Y^0_{\tau_{1,k}}}\\
&
+\sum_{l=1}^{n-1}\left[\sup_{t\in[0,T]}
\norm{
\bigl(\intOp_{l+1}(Y^{n-l}-Y^0)\bigr)(t)
-
\bigl(\intOp_l(\intOp_{l+1}(Y^{n-l}-Y^{0}))\bigr)(t)
}\right].
\end{split}
\end{equation}
\cref{n001} 
hence proves that
\begin{equation}
\begin{split}
&\sup_{t\in[0,T]}\norms{ \barY_t - (\intOp_n(Y^0))(t)}
\leq \max_{k\in \{0,1,\ldots m_1\}}\norms{Y^n_{\tau_{1,k}}-Y^0_{\tau_{1,k}}}\\
&+\sum_{l=1}^{n-1}\frac{ \lvert \varepsilon_l\rvert^\alpha}{2^{\alpha}} \left[
\sup_{r,s\in [0,T], \, r\neq s} \frac{\norm{\bigl(\intOp_{l+1}(Y^{n-l}-Y^0)\bigr)(r)-\bigl(\intOp_{l+1}(Y^{n-l}-Y^0)\bigr)(s)}}{|r-s|^\alpha}\right].
\end{split}
\end{equation}
\cref{n002b}
hence ensures that 
\begin{equation}\label{n008}
\begin{split}
&\sup_{t\in[0,T]}\norm{ \barY_t - (\intOp_n(Y^0))(t)}
\leq\max_{k\in \{0,1,\ldots m_1\}}\norms{Y^n_{\tau_{1,k}}-Y^0_{\tau_{1,k}}}\\
&+\sum_{l=1}^{n-1}\frac{ \lvert \varepsilon_l\rvert^\alpha}{2^{\alpha}} \left[
\sup_{r,s\in [0,T], \, r\neq s} \frac{\norm{(Y^{n-l}_s-Y^0_s)-(Y^{n-l}_r-Y^0_r)}}{|r-s|^\alpha}\right].
\end{split}
\end{equation}
Moreover, observe that \cref{n001} ensures that 
\begin{align}
\sup_{t\in [0,T]}\norm{(\intOp_n (Y^0))(t)-Y^0_t}
\leq 
\frac{|\epsilon_n|^\alpha}{2^{\alpha}}\left[\sup_{t,s\in[0,T],\, t\neq s}\frac{\norm{Y_t^0-Y_s^0}}{|t-s|^\alpha}\right].
\end{align}
The fact that for all $v, w \in V$ it holds that 
$\norm{v+w} \leq \norm{v} + \norm{v}$ and \eqref{n008} hence demonstrate that
\begin{align}\begin{split}
\sup_{t\in[0,T]}\norm{\barY_t-Y^0_t}&\leq 
\sup_{t\in[0,T]}\Biggl[
\norm{\barY_t-(\intOp_n (Y^0))(t)}+
\norm{(\intOp_n (Y^0)(t)-Y^0_t}\Biggr]\\
&
\leq \max_{k\in \{0,1,\ldots m_1\}}\norms{Y^n_{\tau_{1,k}}-Y^0_{\tau_{1,k}}}
+ \frac{|\epsilon_n|^\alpha}{2^{\alpha}}\left[\sup_{t,s\in[0,T],\, t\neq s}\frac{\norm{Y_t^0-Y_s^0}}{|t-s|^\alpha}\right]\\
&\quad
+\sum_{l=1}^{n-1}\frac{ \lvert \varepsilon_l\rvert^\alpha}{2^{\alpha}} \left[
\sup_{r,s\in [0,T], \, r\neq s} \frac{\norm{(Y^{n-l}_s-Y^0_s)-(Y^{n-l}_r-Y^0_r)}}{|r-s|^\alpha}\right].
\end{split}\end{align}
The proof of \cref{n002} is thus complete.
\end{proof}

\section{Existence, uniqueness, and H\"older continuity properties for solutions of stochastic fixed-point equations}\label{sec:3}
\newcommand{\lyaV}{\varphi}
\newcommand{\cstC}{\left\|\|Z\|\right\|_{\lpspace{\frac{1}{1-\beta}}}}

An important aspect in the derivation of the Monte Carlo-type approximation algorithm for BSDEs in \eqref{eq:mlp_intro}--\eqref{eq:approx_intro} in \cref{h01} in \cref{sec:intro} above is, loosely speaking, to reformulate the solutions of the BSDEs in \eqref{eq:bsde_intro} as solutions of appropriate SFPEs associated to the BSDEs in \eqref{eq:bsde_intro} and in this section we establish in \cref{k16} below existence, uniqueness, and H\"older continuity properties for solutions of such SFPEs.

In particular, under suitable assumptions, item \eqref{k03} in \cref{k16} proves that the SFPE in \eqref{k03b} below has a unique solution within the set of functions which grow at most like a certain Lyapunov-type function (see the function $V\colon [0,T]\times \R^d \to [1,\infty)$ above \eqref{k03b} in \cref{k16} for details), item \eqref{k15} in \cref{k16} establishes a suitable explicit a priori growth bound for the unique solution of the SFPE in \eqref{k03b}, and item \eqref{k20b} in \cref{k16} proves that the unique solution of the SFPE in \eqref{k03b} is $ \nicefrac{1}{2} $-H\"{o}lder-continuous in the time variable $ t \in [0,T] $ and locally $ 1 $-H\"{o}lder continous (locally Lipschitz continuous) in the space variable $ x \in \R^d $. 

Further existence, uniqueness, and regularity results for SFPEs can, e.g., be found in \cite[Section 4]{hjk2019overcoming}, \cite[Section 2 and Section 3]{beck2021nonlinear}, \cite[Section 2 and Section 3]{beck2019existenceb}, and \cite[Section 2]{HJKN20}.


\begin{lemma}\label{k16}
Let $d\in \N$, 
 $L\in [0,\infty)$, $T\in (0,\infty)$, $p_1,p_2,p_3\in[1,\infty]$
satisfy $\frac{1}{p_1}+\frac{1}{p_2}+\frac{1}{p_3}\leq 1$,
let $\lVert \cdot \rVert \colon \R^d\to[0,\infty)$ be a norm, 
let
$f\colon [0,T]\times\R^d\times\R \to \R$, $g\colon \R^d\to \R$, 
$\phi\colon [0,T]\times\R^d\to   [1,\infty) $, ${V}\colon [0,T]\times\R^d\to   [1,\infty) $, and $\psi\colon [0,T]\times\R^d\to   [1,\infty) $ be measurable,
let 
$(\Omega,\mathcal{F},\P)$ be a probability space,
for every random variable 
$\mathfrak{X}\colon\Omega\to\R$
 let 
 $\lp{\mathfrak{X}}_{p} \in [0,\infty]$, $p\in[1,\infty]$,
satisfy for all $p\in[1,\infty)$ that $\lp{\mathfrak{X}}_{p}= (\E[|\mathfrak{X}|^p])^{1/p}$ and $\lp{\mathfrak{X}}_{\infty}=
\inf (\{r\in [0,\infty)\colon \P(|\mathfrak{X}|>r )=0 \}\cup\{\infty\})
$,
for every
$ s\in[0,T]$, $x\in\R^d$ let
$X^x_{s,(\cdot)}=(X_{s,t}^x(\omega))_{(t,\omega)\in[s,T]\times\Omega}\colon  [s,T] \times\Omega\to\R^d$
be measurable,
assume for all measurable  $h\colon [0,T]\times\R^d\times\R^d\to [0,\infty)$ that
$
\{(\mathfrak{s},\mathfrak{t})\in [0,T]^2\colon\mathfrak{s}\leq \mathfrak{t}\}
\times\R^d \times \R^d \ni (\mathfrak{s},\mathfrak{t},x,y) \mapsto \E\bigl[h\bigl(\mathfrak{t},X^{x}_{\mathfrak{s},\mathfrak{t}},X^{y}_{\mathfrak{s},\mathfrak{t}}\bigr)\bigr] \in [0,\infty]
$
is measurable,
 and
assume for all $s\in [0,T]$, $t\in[s,T]$, $r\in[t,T]$,
$x,y\in \R^d$,
$v,w\in \R$ and all measurable  $h\colon [0,T]\times\R^d\times\R^d\to [0,\infty)$ that
\begin{gather}\label{k07}
\lp{\phi(t,X_{s,t}^x)}_{p_3}\leq \phi(s,x), \qquad \max\bigr\{|g(x)|\1_{\{T\}}(s),|Tf(s,x,0)|, \lp{V(t,X_{s,t}^x)}_{p_1}\bigl\}\leq V(s,x)
\\
\label{k05}
|g(x)-g(y)|\leq \tfrac{1}{2\sqrt{T}} (V(T,x)+V(T,y))\|x-y\|,\\
\label{k05b}
|f(t,x,v)-f(t,y,w)|\leq L |v-w|+
\tfrac{1}{2T^{3/2}} (V(t,x)+V(t,y))\|x-y\|,
 \\
\lp{\|
X_{s,t}^x-x\|
}_{p_2}\leq \psi(s,x)|s-t|^{\nicefrac{1}{2}},\qquad 
\lp{\|X_{s,t}^x-X_{s,t}^y\|}_{p_2}\leq
\tfrac{1}{2} (\phi(s,x)+\phi(s,y))\|x-y\|,
\label{k04}
\\
  \label{b01k}
  \text{and}\qquad 
\E\! \left[ \E\!\left[ h\bigl(r,X^{a}_{t,r},X^{b}_{t,r}\bigr) \right] \bigr|_{(a,b)=(X_{s,t}^{x},X_{s,t}^{y})} \right]
   =\E\!\left[h\bigl(r,X^{x}_{s,r},X^{y}_{s,r}\bigr)\right] ,
\end{gather}
Then 
\begin{enumerate}[(i)]
\item \label{k03}
there exists a unique measurable  $u\colon [0,T]\times\R^d\to\R$ which satisfies for all $t\in[0,T]$, $x\in \R^d$ that
$\E\bigl[|
g(X_{t,T}^x)|\bigr]+
\int_{t}^{T} \E\bigl[|f(r,X_{t,r}^x, u(r,X_{t,r}^x))|
\bigr]\,d r+\sup_{r\in[0,T],\xi\in\R^d}\bigl(\frac{|u(r,\xi)|}{V(r,\xi)}\bigr)<\infty$ and
\begin{align}\label{k03b}
u(t,x)=
\E\bigl[
g(X_{t,T}^x)\bigr]+
\int_{t}^{T} \E\bigl[f(r,X_{t,r}^x, u(r,X_{t,r}^x))
\bigr]\,dr ,
\end{align}
\item \label{k15} it holds 
 for all 
$t\in[0,T]$,
$x\in\R^d$ 
that $|u(t,x)|\leq 2e^{L(T-t)}V(t,x)$,
and
\item  \label{k20b}it holds 
 for all 
 $s\in[0,T]$, 
$t\in[s,T]$,
$x,y\in\R^d$ 
that
$
|u(s,x)-u(t,y)|\leq T^{ - 1 / 2 }  e^{2 L T}(V(s,x)+V(t,y))
(\phi(s,x)+\phi(t,y))
\bigl[\psi(s,x)|s-t|^{\nicefrac{1}{2}}+\|x-y\|\bigr]
$.
\end{enumerate}
\end{lemma}

\begin{proof}[Proof of \cref{k16}]
Observe that \cite[Proposition 2.2]{HJKN20} (applied with 
$ \mathcal{O}\defeq\R^d$
 in the notation of 
\cite[Proposition 2.2]{HJKN20})
 and \eqref{k07} prove items \eqref{k03} and \eqref{k15}.
Next note that 
 \eqref{k03b}, the triangle inequality, and
\eqref{b01k}
 show that for all
 $s\in[0,T]$, 
$t\in[s,T]$,
$x,y\in\R^d$ it holds that 
{\allowdisplaybreaks\begin{align}
&
\E\bigl[
|u(t,X^{x}_{s,t})-u(t,X^{y}_{s,t})|\bigr]
=\E\bigl[
|u(t,a)-u(t,b)||_{{(a,b)=(X_{s,t}^x,X_{s,t}^y)}}\bigr]
\nonumber \\
&=\E\!\left[\biggl|
\E\bigl[
g(X_{t,T}^{a})
-g(X_{t,T}^{b})
\bigr]+
\int_{t}^{T} \E\bigl[f(r,X_{t,r}^{a}, u(r,X_{t,r}^{a}))
-f(r,X_{t,r}^{b}, u(r,X_{t,r}^{b}))
\bigr]dr\biggr|\biggr|_{(a,b)=(X_{s,t}^x,X_{s,t}^y)}\right]\nonumber 
\\
&
\leq \E\!\left[
\E\!\left[
|
g(X_{t,T}^{a})
-g(X_{t,T}^{b})
|\right]\bigr|_{(a,b)=(X_{s,t}^x,X_{s,t}^y)}\right]\nonumber \\&\quad+
\int_{t}^{T} \E\!\left[\E\!\left[| f(r,X_{t,r}^{a}, u(r,X_{t,r}^{a}))
-f(r,X_{t,r}^{b}, u(r,X_{t,r}^{b}))|
\right]\bigr|_{(a,b)=(X_{s,t}^x,X_{s,t}^y)}\right]dr
\nonumber \\
&=\E\Bigl[ \bigl|g(X_{s,T}^x)-g(X_{s,T}^y)\bigr| \Bigr] +
\int_{t}^{T}
\E\Bigl[
\left|
f(r,X_{s,r}^{x}, u(r,X_{s,r}^{x}))
-f(r,X_{s,r}^{y}, u(r,X_{s,r}^{y}))\right|\Bigr]dr.
\end{align}
H\"older's inequality, \eqref{k05}, \eqref{k05b},
 the fact that
$\frac{1}{p_1}+\frac{1}{p_2}\leq 1$, \eqref{k04}, and \eqref{k07}
hence  demonstrate that for all
 $s\in[0,T]$, 
$t\in[s,T]$,
$x,y\in\R^d$ it holds that 
\begin{align}
&
\E\bigl[
|u(t,X^{x}_{s,t})-u(t,X^{y}_{s,t})|\bigr]
\nonumber
\\
&\leq \tfrac{1}{2\sqrt{T}}\,
\E\!\left[
(
 V(T,X_{s,T}^x)+V(T,X_{s,T}^y))
\|X_{s,T}^x-
X_{s,T}^y\|\right]
+\int_{t}^{T}L \,\E\!\left[|
u(r,X_{s,r}^x)-u(r,X_{s,r}^y)|\right] dr\nonumber \\
&\quad +\tfrac{ 1 }{ 2 T \sqrt{T} }\int_t^T
\E\!\left[
(V(r,X_{s,r}^x)+V(r,X_{s,r}^y))
\|X_{s,r}^x-X_{s,r}^y\|\right]dr\nonumber \\
&\leq \sup_{r\in [t,T]}\left[\tfrac{1}{\sqrt{T}}
\lp{
 V(r,X_{s,r}^x)+V(r,X_{s,r}^y)}_{p_1} 
\lp{\|X_{s,r}^x-
X_{s,r}^y\|}_{p_2}\right]\nonumber \\
&\quad 
+\int_{t}^{T}L \,\E\!\left[|
u(r,X_{s,r}^x)-u(r,X_{s,r}^y)|\right] dr\nonumber \\
&\leq  \left[
\tfrac{V(s,x)+V(s,y)}{\sqrt{T}}\right]\!
\left[
\tfrac{\phi(s,x)+\phi(s,y) }{2}
\right]\!
\|x-y\| +\int_{t}^{T}
L \, \E\!\left[|
u(r,X_{s,r}^x)-u(r,X_{s,r}^y)|\right] dr
.
\end{align}}%
This, item \eqref{k15}, \eqref{k07}, and Gronwall's lemma (see, e.g., \cite[Lemma~3.2]{hutzenthaler2019overcoming}) show that for all
 $s\in[0,T]$, 
$t\in[s,T]$,
$x,y\in\R^d$ it holds that
\begin{align}
\E\!\left[|u(t,X^{x}_{s,t})-u(t,X^{y}_{s,t})|\right]\leq \tfrac{1}{2\sqrt{T}}(V(s,x)+V(s,y) )
(\phi(s,x)+\phi(s,y) )
\|x-y\|e^{L (T-t)}.\label{k13b}
\end{align}
Moreover, observe that \eqref{k04} ensures that for all 
$t\in[0,T]$,
$x\in\R^d$ it holds that
$\P( \lVert X^x_{ t, t } - x\rVert=0 ) = 1$. Hence, we obtain that for all
$t\in[0,T]$,
$x\in\R^d$ it holds that
$\P( X^x_{ t, t } = x ) = 1$.
Combining this with \eqref{k13b} establishes that
for all 
$t\in[0,T]$,
$x,y\in\R^d$ it holds that
\begin{align}
\label{k06}
|u(t,x)-u(t,y)|\leq \tfrac{1}{2\sqrt{T}}(V(t,x)+V(t,y) )
(\phi(t,x)+\phi(t,y) )
\|x-y\|e^{L (T-t)}.
\end{align}
Next note that \eqref{k03b}, 
Fubini's theorem, and \eqref{b01k}  show that for all
 $s\in[0,T]$, 
$t\in[s,T]$,
$x\in\R^d$ it holds that
\begin{align}
u(s,x)-\E\! \left[u(t,X_{s,t}^x)\right]&=
\E\bigl[
g(X_{s,T}^x)\bigr]+
\int_{s}^{T} \E\bigl[f(r,X_{s,r}^x, u(r,X_{s,r}^x))
\bigr]\,dr\nonumber  \\
&\quad 
-
 \E\!\left[
\left[
\E\bigl[
g(X_{t,T}^{\tilde{x}})\bigr]+
\int_{t}^{T} \E\bigl[f(r,X_{t,r}^{a}, u(r,X_{t,r}^{\tilde{x}}))
\bigr]\,dr\right]\biggr|_{{a}=X_{s,t}^x}\right]\nonumber \\
&=\int_{s}^{t}
\E\bigl[f(r,X_{s,r}^x, u(r,X_{s,r}^x))
\bigr]dr.
\label{k01b}\end{align}
This, the triangle inequality, \eqref{k05b}, \eqref{k07}, and item \eqref{k15} demonstrate that  for all
 $s\in[0,T]$, 
$t\in[s,T]$,
$x\in\R^d$ it holds that
\begin{align}\begin{split}
&\left|
u(s,x)-
\E\! \left[u(t,X_{s,t}^x)\right]\right|
\leq |t-s| \left[\sup_{r\in[s,t]}\left(
\E\!\left[|
f(r,X_{s,r}^x, 0)|\right]+
 L
\E\!\left[|
u(r,X_{s,r}^x)|\right] \right)\right]\\
&
 \leq |t-s|\left[\sup_{r\in[s,t]}\left(
( \tfrac{1}{T}+2 L e^{LT})
\E\!
\left[ V(r,X_{s,r}^x)\right]\right)\right]\leq 
\Bigl[\tfrac{1+2T  L e^{LT}}{T}\Bigr] |t-s|V(s,x)
\\
&
\leq \tfrac{1}{\sqrt{T}}
(2+4LTe^{LT})\tfrac{1}{2}(V(s,x)+V(t,y))|t-s|^{\nicefrac{1}{2}}
.
\end{split}\label{k13c}\end{align}
Next observe that \eqref{k06}, H\"older's inequality, the fact that
$\frac{1}{p_1}+\frac{1}{p_2}+\frac{1}{p_3}\leq 1$,
the triangle inequality, \eqref{k07}, and \eqref{k04} prove that for all 
 $s\in[0,T]$, 
$t\in[s,T]$,
$x,y\in\R^d$ it holds that
\begin{align}
&
\left|\E\! \left[u(t,X_{s,t}^x)\right]-u(t,y)\right| 
\leq \E\! \left[|u(t,X_{s,t}^x)-u(t,y)| \right]\nonumber 
\\
&\leq \E\!\left[
 \tfrac{2 }{\sqrt{T}}\tfrac{1}{2}(V(t,X_{s,t}^x)+V(t,y))
\tfrac{1}{2}(\phi(t,X_{s,t}^x)+\phi(t,y))
\|X_{s,t}^x-y\|\right]e^{L(T-t)}\nonumber 
\\
&\leq 
\tfrac{2e^{L T}}{\sqrt{T}}
\tfrac{1}{2}\left(\lp{V(t,X_{s,t}^x)}_{p_1} +V(t,y) \right)\tfrac{1}{2}
\left(\lp{\phi(t,X_{s,t}^x)}_{p_3}+\phi(t,y) \right)
\lp{\|X_{s,t}^x-y\|}_{p_2}\nonumber  \\
&\leq 
\tfrac{1}{\sqrt{T}}2e^{L T}\tfrac{1}{2}(V(s,x)+V(t,y))
\tfrac{1}{2}(\phi(s,x)+\phi(t,y))
\left[\psi(s,x)|s-t|^{\nicefrac{1}{2}}+\|x-y\|\right].\label{k25}
\end{align}
This, the triangle inequality,  \eqref{k13c}, the fact that $\phi\ge 1$, the fact that $\psi\geq 1$,
and the fact that
$ 2+4LTe^{LT}+2e^{LT}\leq 4e^{LT}(1+LT) \leq 4e^{2LT}$
 show that for all 
 $s\in[0,T]$, 
$t\in[s,T]$,
$x,y\in\R^d$ it holds that
\begin{align}\begin{split}
&
|u(s,x)-u(t,y)|\leq \left|
u(s,x)-
\E\! \left[u(t,X_{s,t}^x)\right]\right|+\left|\E\! \left[u(t,X_{s,t}^x)\right]-u(t,y)\right|\\
&\leq\tfrac{1}{\sqrt{T}}4 e^{2 L T}\tfrac{1}{2}(V(s,x)+V(t,y))
\tfrac{1}{2}(\phi(s,x)+\phi(t,y))
\left[\psi(s,x)|s-t|^{\nicefrac{1}{2}}+\|x-y\|\right].
\end{split}\end{align}
This proves item \eqref{k20b}. The proof of \cref{k16} is thus complete.
\end{proof}


\section{Error analysis in H\"older seminorms for full-history recursive multilevel Picard (MLP) approximations}\label{sec:compltwoplus}

\newcommand{\bbM}{R}%
\renewcommand{\intOp}{\mathscr{L}}
\renewcommand{\thetaBar}{\theta}
\renewcommand{\FEU}[1]{\mathcal{C}_{#1}}
\renewcommand{\FEY}[1]{{\mathfrak{C}}_{#1}}
\newcommand{\FEYT}[1]{\widehat{\mathcal{C}}_{#1}}
\renewcommand{\F}{\mathbb{F}}
\newcommand{\luckyCst}{\ensuremath{49}}
\newcommand{\luckyCstT}{\ensuremath{50}}
\newcommand{\threenorm}[2]{{\left\vert\kern-0.25ex\left\vert\kern-0.25ex\left\vert #1     \right\vert\kern-0.25ex\right\vert\kern-0.25ex\right\vert}_{#2}}
\newcommand{\funcEps}[2]{\mathcal{E}^{(#1)}_{#2,M}}
\newcommand{\funcEpsT}[1]{\mathcal{E}_{#1,M}}
\newcommand{\capG}{\mathbf{g}}
\newcommand{\capA}{{\mathbf{a}}_M}
\newcommand{\capB}{\mathbf{b}_M}
\newcommand{\capV}{\mathbf{v}_M}
\newcommand{\barV}{\mathbb{V}}
\newcommand{\funcLambda}[2]{\Lambda(#1,#2)}
\newcommand{\cstMD}{ c }

In \cref{t26} in \cref{sec:compcompl} below we supply a computational complexity analysis for the Monte Carlo-type approximation algorithm for BSDEs in \eqref{eq:mlp_intro}--\eqref{eq:approx_intro} in \cref{h01} in \cref{sec:intro} above. Our proof of \cref{t26} exploits the multi-grid approximation result in \cref{n002} in \cref{sec:2} above as well as the error analysis for appropriate MLP approximations in \cref{m02} in this section. Specifically, in \cref{m02} below we establish upper bounds for appropriate H\"older seminorms of the difference between solutions of SFPEs and suitable MLP approximations for such SFPEs.

\begin{setting}\label{m01}
Let  $T \in (0,\infty)$, $L,\rho\in [0,\infty)$,  $\beta\in(0,\nicefrac{1}{12}]$, 
$d\in \N$,
$f\in C( [0,T]\times \R^{d}\times\R,\R)$, ${g}\in C( \R^d,\R)$, 
let
$V\colon \R^d\to[1,\infty)$ be
measurable,
let $\lVert \cdot \rVert \colon \R^d\to[0,\infty)$ be a norm,
let 
$  \Theta = \bigcup_{ n \in \N }\! \Z^n$,
let
$(\Omega, \mathcal{F}, \P, (\F_t)_{t\in[0,T]})$ 
be a filtered probability space,
let $\unif^\theta\colon \Omega\to[0,1]$, $\theta\in \Theta$, 
be i.i.d.\ random variables, assume for all $t\in (0,1)$ that 
$\P(\unif^0\le t)=t$,
let $\stdNormal^\theta\colon  \Omega\to\R^d$, $\theta\in\Theta$,
be i.i.d.\ standard normal random vectors,
let $W=(W^1,W^2,\ldots,W^d)\colon [0,T]\times\Omega\to\R^d$ be a standard 
$(\F_t)_{t\in[0,T]}$-Brownian motion,
assume that 
$
(\unif^\theta)_{\theta\in\Theta}$, $
(\stdNormal^\theta)_{ \theta\in\Theta}$, and
$W$
 are independent,
let $X_{s,t}^\theta \colon \R^d \times \Omega \to \R^d$, $s,t\in[0,T]$, $\theta\in\Theta$, satisfy for all $s,t\in[0,T]$, $x\in\R^d$, $\theta\in\Theta$ that 
$
X_{s,t}^\theta(x)=x+\lvert t-s\rvert^{1/2}\stdNormal^\theta
$,
let
${F}\colon\R^{[0,T]\times\R^d}\to\R^{[0,T]\times\R^d}$ satisfy for all $t\in [0,T]$, $x\in \R^d$, $v\in \R^{[0,T]\times\R^d}$ that $(F(v))(t,x)= f(t,x,v(t,x))$, 
let
$ 
  {{U}}_{ n,M}^{\theta } \colon [0, T] \times \R^d \times \Omega \to \R
$, $n,M\in\Z$, $\theta\in\Theta$, satisfy
for all $M \in \N$, $n\in \N_0$, $\theta\in\Theta $, 
$ t \in [0,T]$, $x\in\R^d $
 that 
\begin{align}
&  {{U}}_{n,M}^{\theta}(t,x)
=
  \frac{\1_{\N}(n)}{M^n}
 \sum_{i=1}^{M^n} 
      {g} \bigl({X}^{(\theta,0,-i)}_{t,T}(x)\bigr)
 \label{t27} \\
& +
  \sum_{\ell=0}^{n-1}\left[ \frac{(T-t)}{M^{n-\ell}}
    \sum_{i=1}^{M^{n-\ell}}
      \bigl({F}\bigl({{U}}_{\ell,M}^{(\theta,\ell,i)}\bigr)-\1_{\N}(\ell){F}\bigl( {{U}}_{\ell-1,M}^{(\theta,-\ell,i)}\bigr)\bigr)
      \bigl(t+(T-t)\unif^{(\theta,\ell,i)},{X}_{t,t+(T-t)\unif^{(\theta,\ell,i)}}^{(\theta,\ell,i)}(x)\bigr)
    \right],\nonumber
\end{align}
and
assume for all $s,t\in[0,T]$, $x,y\in\R^d$,
 $v_1,v_2,w_1,w_2\in\R$  that
{\allowdisplaybreaks
\begin{gather}
%
\begin{split}
&
\max\{|f(s,x,v_1)-f(t,y,v_2)|,T^{-1}
|g(x)-g(y)|\}\\
&\leq T^{-\nicefrac{3}{2}} |{V}(x)+{V}(y)|^{\beta}\left(|s-t|^{\nicefrac{1}{2}}+\|x-y\|\right)+L |v_1-v_2|,\label{t02b}
\end{split}
\\[5pt]
\label{m11b}\begin{split}
&\left|\left[f(s,x,{v_1})-f(s,x,{w_1})\right] 
-
\left[f(t,y,{v_2})-f(t,y,{w_2})\right]\right|\leq L \left|({v_1}-{w_1})-({v_2}-{w_2})\right|\\
&\quad
+
T^{-\nicefrac{3}{2}}|{V}(x)+{V}(y)|^{\beta}\bigl[
\left(\max\!\left\{\E\!\left[\|\stdNormal^0\|^4\right],1\right\}\right)^{\!\nicefrac{1}{4}}
|s-t|^{\nicefrac{1}{2}}+\|x-y\|\bigr]|{v_1}-{w_1}|
\\&\quad +T^{-1}|{V}(x)+{V}(y)|^{\beta}\left(|{v_1}-{w_1}|+|{v_2}-{w_2}|
\right)\left|{w_1}-{w_2}\right|,
\end{split}
\end{gather}}%
and
$\max\{|Tf(t,x,0)|^{\nicefrac{1}{\beta}},|g(x)|^{\nicefrac{1}{\beta}}, \E[V(X_{s,t}^0(x))]\}\leq e^{\rho\lvert t-s\rvert }V(x)$.
\end{setting}

\begin{proposition}\label{m02}
Assume \cref{m01}, let 
$\cstMD\in \R $ satisfy 
$
\cstMD= (\max\{\E [\|\stdNormal^0\|^4],1\})^{\nicefrac{1}{4}}
$, and
let $\barV\colon [0,T]\times\R^d\to \R $ satisfy 
for all $t\in[0,T]$,  $x\in\R^d$ that $\barV(t,x)= e^{\rho (T-t)}V(x)$.
Then
\begin{enumerate}[(i)]
\item   \label{t28c}
 there exists a unique 
measurable 
$u\colon  [0,T]\times\R^d\to\R$ 
which satisfies
for all   $t\in[0,T]$, $x\in\R^d$  that
$
\E\bigl[|{g}(X^0_{t,T}(x))|\bigr]+\int_t^T\E\bigl[| ({F}(u))(s,X^0_{t,s}(x))| \bigr]\,ds +
\sup_{r\in[0,T],y\in\R^d}
\bigl(
\frac{|u(r,y)|}{|V(y)|^\beta}\bigr)<\infty$ and
\begin{align}
\begin{split}
u(t,x)=\E\!\left[{g}(X^0_{t,T}(x))\right]+\int_t^T
\E\!
\left[ \left({F}(u)\right)\!\left(s,X^0_{t,s}(x)\right)  \right]ds,
\end{split}
\end{align} 
\item \label{t33}
it holds for all $\theta\in\Theta$, $n\in\N_0$, $M\in\N$ that $U_{n,M}^\theta$ is measurable,
\item\label{k21b} it holds
 for all
$M\in\N$, $N\in\N_0$ that 
\begin{equation}
\sup_{s\in [0,T]} \sup_{x\in\R^d}\left[ \frac{\left(\E\!\left[|{U}_{N,M}^{0}(s,x)-u(s,x)|^2\right]\right)^{\nicefrac{1}{2}}}{(\barV(s,x))^{\beta}}\right]\le 
e^{M/2}M^{-N/2}\left(\luckyCstT e^{2LT}\right)^{N+1},
\end{equation}
and 
\item\label{k21bb}
it holds for all
$M\in\N$, $N\in\N_0$ that 
\begin{multline}
\sup_{\substack{s,t\in [0,T],\\ s\neq t}}
\sup_{\substack{x,y\in\R^d,\\x\neq y}}
\left[\frac{T^{\nicefrac{1}{2}}\left(\E\!\left[|[{U}_{N,M}^{0}(s,x)-u(s,x)]-[{U}_{N,M}^{0}(t,y)-u(t,y)]|^2\right]\right)^{\nicefrac{1}{2}}}{
\bigl[
\cstMD
|s-t|^{\nicefrac{1}{2}}+\|x-y\|\bigr]\bigl(\barV(s,x)+\barV(t,y)\bigr)^{\nicefrac{1}{4}}}\right]\\
\le 
e^{M/2}M^{-N/2}\left(\luckyCstT e^{2LT}\right)^{N+1}.
\end{multline}
\end{enumerate}
\end{proposition}
\begin{proof} [Proof of \cref{m02}]
Throughout this proof 
let 
$\Lambda\colon [0,1]\times[0,T]\to[0,T] $ 
 satisfy for all $t\in[0,T]$, $\lambda\in [0,1]$ that
$
\funcLambda{\lambda}{t}= t+ \lambda(T-t)
$,
for every $q\in [1,\infty)$ and every random variable 
$\mathfrak{X}\colon\Omega\to\R$
let $\lp{\mathfrak{X}}_{q} \in [0,\infty]$
satisfy that $\lp{X}_{q}= (\E[|\mathfrak{X}|^q])^{1/q}$,
and 
for every $r\in [0,T]$ and
 every random field
$H\colon [0,T]\times \R^d\times \Omega\to\R $ let 
$\threenorm{H}{k,r}\in [0,\infty]$, $k\in \{0,1,2\}$, satisfy
\begin{align}\label{m22}
\begin{split}
\threenorm{H}{0,r}&=\max _{j\in\{1,2\}}\threenorm{H}{j,r},\qquad 
\funcN{H}{r}= \sup_{x\in\R^d,s\in [r,T]}\left[ \frac{\left(\E\!\left[|H(s,x)|^2\right]\right)^{\nicefrac{1}{2}}}{(\barV(s,x))^{\beta}}\right],
\\
\text{and} \qquad
\funcH{H}{r}&=\sup_{\substack{s,t\in [r,T],x,y\in\R^d\\(s,x)\neq(t,y)}}\left[\frac{T^{\nicefrac{1}{2}}\left(\E\!\left[|H(s,x)-H(t,y)|^2\right]\right)^{\nicefrac{1}{2}}}{
\bigl[
\cstMD
|s-t|^{\nicefrac{1}{2}}+\|x-y\|\bigr]\bigl(\barV(s,x)+\barV(t,y)\bigr)^{\nicefrac{1}{4}}}\right].
\end{split}\end{align}
Observe that \eqref{m22} ensures that for all $j\in\{0,1,2\}$, $r\in[0,T]$, $\lambda,\mu\in\R$ and all random fields 
$H_k\colon [0,T]\times\R^d\times\Omega\to\R$, $k\in\{1,2\}$,
it holds that
\begin{align}
\threenorm{\lambda H_1+\mu H_2}{j, r}\leq 
|\lambda|
\threenorm{ H_1 }{j, r}+
|\mu| \threenorm{ H_2 }{j, r}
.
\label{m22b}
\end{align}
Moreover, note that \eqref{m22} assures that for all $j\in \{0,1,2\}$ and all random fields $H\colon [0,T]\times \R^d\times \Omega\to\R $ it holds that $[0,T]\ni r\mapsto \threenorm{H}{j,r}\in [0,\infty]$ is non-increasing. This shows that for all $j\in \{0,1,2\}$ and all random fields $H\colon [0,T]\times \R^d\times \Omega\to\R $ it holds that $[0,T]\ni r\mapsto \threenorm{H}{j,r}\in [0,\infty]$ is measurable. 
Next observe that Jensen's inequality and
the fact that for all $t\in[0,T]$, $x\in\R^d$ it holds that $\barV(t,x)= e^{\rho (T-t)}V(x)$
show that  for all 
$\gamma\in [0,1]$,
$s\in[0,T]$, $t\in[s,T]$, $x\in\R^d$ it holds  that
\begin{align}\begin{split}
&
\E\!\left[({{\barV}}(t,X_{s,t}^0(x)))^\gamma \right] 
\leq 
\left(\E\!\left[{{\barV}}(t,X_{s,t}^0(x)) \right]\right)^\gamma\leq 
\left(
e^{\rho(T-t)}
e^{\rho(t-s)}V(x)\right)^{\gamma}\leq 
 ({{\barV}}(s,x))^\gamma.\label{t10}
\end{split}\end{align}
 Combining this,
the fact that
$0< \beta\leq 3/4$, the fact that $\cstMD\geq 1$, 
\eqref{t02b},   and
\cref{k16} (applied with 
$p_1\defeq 1/\beta$,
$p_2\defeq 4$, $p_3\defeq\infty$,
$\phi\defeq (\R^d\ni x\mapsto 1\in [1,\infty))$,
$V\defeq 2\barV^\beta$,
$ \psi \defeq 
(\R^d\ni x\mapsto \cstMD\in [1,\infty))
$,
$ (X_{s,t}^x(\omega))_{(s,t,x,\omega)\in \{(\mathfrak s, \mathfrak t)\in [0,T]^2\colon \mathfrak s\le \mathfrak t\}\times \R^d \times \Omega}
\defeq (X _{s,t}^0(x,\omega))_{(s,t,x,\omega)\in \{(\mathfrak s,\mathfrak t)\in [0,T]^2\colon \mathfrak s\le \mathfrak t\}\times \R^d \times \Omega}$
in the notation of \cref{k16}) 
 implies that \begin{enumerate}[(a)]
\item 
there exists a unique 
measurable 
$u\colon  [0,T]\times\R^d\to\R$ 
which satisfies
for all   $s\in[0,T]$, $x\in\R^d$  that
$
\E\bigl[|{g}(X^0_{s,T}(x))|\bigr]+\int_s^T\E\bigl[| ({F}(u))(t,X^0_{s,t}(x))| \bigr]dt +
\sup_{r\in[0,T],y\in\R^d}
\bigl(
\frac{|u(r,y)|}{|\barV(r,y)|^\beta}
\bigr)
 <\infty$ and
\begin{align}
\begin{split}
u(s,x)=\E\!\left[{g}(X^0_{s,T}(x))\right]+\int_s^T
\E\! \left[\left({F}(u)\right)\!\left(t,X^0_{s,t}(x)\right) \right]ds ,\label{t27b}
\end{split}
\end{align}
   \item it holds for all  $t\in[0,T]$, $x\in\R^d$  that
\begin{align}\label{m23b}
|u(t,x)|\leq 2 e^{LT}2(\barV(t,x))^\beta=
4e^{LT}(\barV(t,x))^{\beta},
\end{align}
and
\item it holds for all  $s,t\in[0,T]$, $x,y\in\R^d$ that
\begin{align}\label{m23}
\begin{split}
|u(s,x)-u(t,y)|&\leq \tfrac{1}{\sqrt{T}}4 e^{2 L T}\tfrac{1}{2}\left(2(\barV(s,x))^\beta+2(\barV(t,y))^\beta\right)
\left[\cstMD|s-t|^{\nicefrac{1}{2}}+\|x-y\|\right]\\
&\leq \tfrac{1}{\sqrt{T}} 8 e^{2 L T}(\barV(s,x)+\barV(t,y))^\beta
\left[\cstMD|s-t|^{\nicefrac{1}{2}}+\|x-y\|\right].
\end{split}
\end{align}

\end{enumerate}
This establishes item \eqref{t28c}. 
Next observe that, e.g.,
\cite[Lemmas 3.2--3.4]{HJKN20}, the fact that
$\stdNormal^\theta$, $\theta\in\Theta$, are independent,
and item \eqref{t28c}
show that 
\begin{enumerate}[(A)]
\item\label{t34} it holds 
for all $n \in \N_0$, $M \in \N$, $\theta\in\Theta $ that
$
  {U}_{ n,M}^{\theta }
$  and 
$
F(  {U}_{ n,M}^{\theta })
$ 
are 
measurable,
\item 
it holds
for all $n, m \in \N_0$, $M \in \N$, $i,j,k,\ell \in \Z$, $\theta \in \Theta$ with $(i,j) \neq (k,l)$ 
that
$
U^{(\theta,i,j)}_{n,M}
$,
$  
  U^{(\theta,k,\ell)}_{m,M}
$,
$ \unif^{(\theta,i,j)}$, and $X^{(\theta,i,j)}$
are independent,
\item it holds 
for all 
$M,n\in\N$,
$t\in[0,T]$, $x\in\R^d$ 
that $U_{n,M}^\theta(t,x)$, $ \theta\in\Theta$, are identically distributed,
\item it holds for all 
$t\in[0,T]$, $x\in\R^d$  that ${g} ({X}^{\theta}_{t,T}(x))$, $ \theta\in\Theta$, are i.i.d., 
\item it holds for all $M,n\in\N$, $t\in[0,T]$, $x\in\R^d$ that
 $\E[|U^0_{n,M}(t,x)|+|{g} ({X}^{0}_{t,T}(x))|]<\infty$,
\item it holds
for all  $\ell\in \N_0$, 
$M\in\N$
 that  
$\bigl((T-t)  ({F}({{U}}_{\ell,M}^{(0,\ell,i)})-\1_{\N}(\ell){F}( {{U}}_{\ell-1,M}^{(0,-\ell,i)}))
      (t+(T-t)\unif^{(0,\ell,i)},{X}_{t,t+(T-t)\unif^{(0,\ell,i)}}^{(0,\ell,i)}(x))\bigr)_{(t,x)\in[0,T]\times\R^d}$,
$ i\in\N$, are i.i.d.,
\item it holds
for all  $\ell\in \N_0$, 
$M\in\N$, $t\in[0,T]$, $x\in\R^d$ that
\begin{align}
\E\Bigl[ \bigl| ({F}({{U}}_{\ell,M}^{(0,\ell,1)})-\1_{\N}(\ell){F}( {{U}}_{\ell-1,M}^{(0,-\ell,1)}))
      (t+(T-t)\unif^{(0,\ell,1)},{X}_{t,t+(T-t)\unif^{(0,\ell,1)}}^{(0,\ell,1)}(x))\bigr|\Bigr]<\infty,
\end{align} 
 \item it holds
for all 
$M,n\in\N$,
$t\in[0,T]$, $x\in\R^d$ 
 that
\begin{align}\label{t31a}
\E\Bigl[\bigl|(T-t)(F(U_{n-1.M}^{0}-F(u))(t+\unif^{0}(T-t),X_{t,t+\unif^{0}(T-t)}^{0}(x))\bigr|\Bigr]<\infty
\end{align}
\item it holds
for all 
$M,n\in\N$,
$t\in[0,T]$, $x\in\R^d$ 
 that
\begin{align}\begin{split}
&\E\!\left[U^0_{n,M}(t,x)\right]-u(t,x)\\
&
=\E\bigl[(T-t)(F(U_{n-1.M}^{0})-F(u))(t+\unif^{0}(T-t),X_{t,t+\unif^{0}(T-t)}^{0}(x))\bigr],
\end{split}\label{t31}\end{align}
and
\item\label{t35} it holds for all $M\in\N$, $n\in\N_0$
that $
U^{0}_{n,M}$, $X^0$, and $\unif^0$ are independent. 
\end{enumerate}
This establishes \eqref{t33}.
Next note that \eqref{m22}, 
\eqref{m23b},
\eqref{m23}, and the fact that for all $M\in \N$ it holds that $U_{0,M}^0=0$
show that for all $r\in[0,T]$, $M\in\N$ it holds that
\begin{align}
\threenorm{U_{0,M}^0-u}{0,r}=
\threenorm{u}{0,r}=\max_{j\in\{1,2\}}
\threenorm{u}{j,r}\leq \max\! \left\{4e^{LT}, 8e^{2LT}\right\}= 8 e^{2LT}.
\label{k10}
\end{align}%
Furthermore, observe that \eqref{t02b}, H\"older's inequality, the triangle inequality,
\eqref{t10}, 
the fact that 
$0\le 4\beta\le 1$, and the fact that 
$V\geq 1$
imply that for all $r\in[0,T]$, $s,t\in[r,T]$, $x,y\in\R^d$ it holds that
\begin{align}
&\lp{g(X^0_{s,T}(x))-g(X^0_{t,T}(y))}_{2}
\leq \tfrac{1}{\sqrt{T}} \blp{\left(\barV(T,X^0_{s,T}(x))+\barV(T,X^0_{t,T}(y))\right)^{\beta}\left\|X^0_{s,T}(x)-X^0_{t,T}(y)\right\|}_{2}\nonumber \\
&
\leq \tfrac{1}{\sqrt{T}}
\left(\E\!\left[\left|\barV(T,X^0_{s,T}(x))+\barV(T,X^0_{t,T}(y))\right|^{4\beta}\right]\right)^{\nicefrac{1}{4}}
\blp{\left\|X^0_{s,T}(x)-X^0_{t,T}(y)\right\|}_{4}\nonumber 
\\
&\leq \tfrac{1}{\sqrt{T}} (\barV(s,x)+\barV(t,y))^{\nicefrac{1}{4}}\left[\cstMD|t-s|^{\nicefrac{1}{2}}+\|x-y\|
\right].
\end{align}%
This and \eqref{m22} assure that for all $r\in[0,T]$ it holds that
\begin{align}
\funcH{[0,T]\times\R^d\times\Omega\ni
(s,x,\omega)\mapsto
g(X^0_{s,T}(x,\omega))\in\R}{r}\leq 1.\label{b01}
\end{align}
Next note that \eqref{t02b} and \eqref{m22} show that for all $r\in[0,T]$, $n\in\N_0$, $M\in \N$ it holds that
\begin{align}\label{m14}
\threenorm{
F(U_{n,M}^{0})-F(u)}{1,r}\leq L \threenorm{U_{n,M}^{0}-u}{1,r}.
\end{align}
In addition, observe that
\eqref{m11b}, the triangle inequality,
\eqref{m22}, \eqref{m23},
the fact that $V\geq1$, and the fact that $0\leq 3\beta\leq 1/4$
imply that  for all $r\in[0,T]$, $s,t\in[r,T]$, $x,y\in\R^d$, $n\in\N_0$, $M\in \N$
 it holds that
{%
\allowdisplaybreaks%
\begin{align}
&
\blp{[F(U_{n,M}^{0})-F(u)](s,x)-[F(U_{n,M}^{0})-F(u)](t,y)}_{\lpspace{2}}
\nonumber \\
&
=\blp{\left[f(s,x,{U_{n,M}^{0}(s,x)})-f(s,x,{u(s,x)})\right] 
-
\left[f(t,y,{U_{n,M}^{0}(t,y)})-f(t,y,{u(t,y)})\right]}_{\lpspace{2}}\nonumber \\
&
\leq L \blp{({U_{n,M}^{0}(s,x)}-{u(s,x)})-({U_{n,M}^{0}(t,y)}-{u(t,y)})}_{\lpspace{2}}\nonumber \\&\quad+\tfrac{1}{T\sqrt{T}}({\barV}(s,x)+{\barV}(t,y))^{\beta} \bigl[\cstMD|t-s|^{\nicefrac{1}{2}}+\|x-y\|\bigr]\lp{{U_{n,M}^{0}(s,x)}-{u(s,x)}}_{\lpspace{2}}\nonumber \\
&\quad+\tfrac{1}{T}({\barV}(s,x)+{\barV}(t,y))^{\beta} \bigl[\lp{{U_{n,M}^{0}(s,x)}-{u(s,x)}}_{\lpspace{2}}+\lp{{U_{n,M}^{0}(t,y)}-{u(t,y)}}_{\lpspace{2}}
\bigr]\left|{u(s,x)}-{u(t,y)}\right|
\nonumber \\
&\leq L\funcH{U_{n,M}^{0}-u}{r} \tfrac{1}{\sqrt{T}}\bigl[\cstMD|t-s|^{\nicefrac{1}{2}}+\|x-y\|\bigr]
(\barV(s,x)+\barV(t,y))^{\nicefrac{1}{4}}\nonumber \\
&\quad+\tfrac{1}{T\sqrt{T}}({\barV}(s,x)+{\barV}(t,y))^{\beta}\bigl[\cstMD|t-s|^{\nicefrac{1}{2}}+\|x-y\|\bigr]\funcN{U_{n,M}^{0}-u}{r}
(\barV(s,x))^{\beta}\nonumber \\
&\quad+\tfrac{1}{T}({\barV}(s,x)+{\barV}(t,y))^{\beta}
\funcN{U_{n,M}^{0}-u}{r}\left[(\barV(s,x))^{\beta}+(\barV(t,y))^{\beta}\right]
\nonumber \\&\quad \cdot\tfrac{8 e^{2LT}}{\sqrt{T}} \bigl[\cstMD|t-s|^{\nicefrac{1}{2}}+\|x-y\|\bigr](\barV(s,x)+\barV(t,y))^{\beta}
\nonumber \\
&\leq \left[\tfrac{1}{\sqrt{T}}L\funcH{U_{n,M}^{0}-u}{r}+\tfrac{1}{T\sqrt{T}}(1+16 e^{2LT})\funcN{U_{n,M}^{0}-u}{r}\right]\nonumber \\&\quad\cdot\left[\cstMD|t-s|^{\nicefrac{1}{2}}+\|x-y\|\right](\barV(s,x)+\barV(t,y))^{\nicefrac{1}{4}}.
\end{align}}%
Combining this and \eqref{m22} shows for all $r\in[0,T]$, $n\in\N_0$, $M\in \N$ that
\begin{align}\label{m13}\threenorm{
F(U_{n,M}^{0})-F(u)}{2,r}\leq L \threenorm{U_{n,M}^{0}-u}{2,r}+ \tfrac{1}{T}(1+16 e^{2LT})\threenorm{U_{n,M}^{0}-u}{1,r}.
\end{align}
Next observe that the fact that
$\forall\,t\in[0,T], \lambda\in [0,1]\colon \funcLambda{\lambda}{t}= t+ \lambda(T-t)$, the fact that 
$\forall\,a,b\in[0,\infty)\colon\left|a^{\nicefrac{1}{2}}-b^{\nicefrac{1}{2}}\right|\leq |a-b|^{\nicefrac{1}{2}}$, and the fact that
$\forall\,\lambda\in [0,1]\colon|1-\lambda|^{\nicefrac{1}{2}}+|\lambda|^{\nicefrac{1}{2}}\leq \sqrt{2}$  demonstrate that for all 
$\lambda\in[0,1]$, $r\in [0,T]$, $s,t\in [r,T]$ it holds that
\begin{gather}\label{m19}
\renewcommand{\sqrt}[1]{\left|#1\right|^{\nicefrac{1}{2}}}
\begin{split}
&\sqrt{\funcLambda{\lambda}{s}-\funcLambda{\lambda}{t}}+\left|\sqrt{\funcLambda{\lambda}{s}-s}-\sqrt{\funcLambda{\lambda}{t}-t}\right|
\\&= \sqrt{(1-\lambda)(s-t)}
+\left|\sqrt{\lambda(T-t)}-\sqrt{\lambda(T-s)}\right|
\\
&\leq \sqrt{(1-\lambda)(s-t)}+
\sqrt{\lambda(s-t)}
\leq  \sqrt{2(s-t)}.
\end{split}
\end{gather}
This,
\eqref{m22}, H\"older's inequality, the triangle inequality, 
and
\eqref{t10}
show that for all 
$\lambda\in[0,1]$, $r\in [0,T]$, $s,t\in [r,T]$, $x,y\in\R^d$, $n\in \N_0$, $M\in\N$,
$H\in \{\lambda_1F(U_{n,M}^0)+\lambda_2F(u)+\lambda_3F(0)\colon \lambda_1,\lambda_2,\lambda_3\in \R\}$ it holds that
{\allowdisplaybreaks\renewcommand{\sqrtT}[1]{\left|#1\right|^{\nicefrac{1}{2}}}
\begin{align}
&
\sqrt{T}\bblp{
\lp{{H}\bigl(\funcLambda{\lambda}{s}, a\bigr) -{H}\bigl(\funcLambda{\lambda}{t}, b\bigr)}_{\lpspace{2}}\bigr|_{\substack{(a,b)=(X^0_{s,\funcLambda{\lambda}{s}}(x),X^0_{t,\funcLambda{\lambda}{t}}(y))}}}_{\lpspace{2}}\nonumber  \\
&\leq \threenorm{H}{2,\funcLambda{\lambda}{r}}\biggl(\cstMD \sqrtT{\funcLambda{\lambda}{s}-\funcLambda{\lambda}{t}}+
\blp{
\bigl\|X^0_{s,\funcLambda{\lambda}{s}}(x)-X^0_{t,\funcLambda{\lambda}{t}}(y)\bigr\|
}_{\lpspace{4}}\biggr)\nonumber 
\\&\quad\cdot
\lp{
\bigl(\barV(\funcLambda{\lambda}{s},X^0_{s,\funcLambda{\lambda}{s}}(x))+\barV(\funcLambda{\lambda}{t},X^0_{t,\funcLambda{\lambda}{t}}(y))\bigr)^{\nicefrac{1}{4}}
}_{\lpspace{4}}
\nonumber \\
&\leq \threenorm{H}{2,\funcLambda{\lambda}{r}}\left[\cstMD \sqrtT{\funcLambda{\lambda}{s}-\funcLambda{\lambda}{t}}
+
\cstMD
 \left|\sqrtT{\funcLambda{\lambda}{s}-s}-\sqrtT{\funcLambda{\lambda}{t}-t}\right|+\|x-y\|\right]\nonumber \\&\quad \cdot
\Bigl(
\E\!\left[\barV(\funcLambda{\lambda}{s},X^0_{s,\funcLambda{\lambda}{s}}(x))+\barV(\funcLambda{\lambda}{t},X^0_{t,\funcLambda{\lambda}{t}}(y))\right]\Bigr)^{\!\nicefrac{1}{4}}\nonumber \\
&\leq\threenorm{H}{2,\funcLambda{\lambda}{r}}\left(\cstMD  \sqrtT{2(s-t)}+\|x-y\|\right)\bigl(\barV(s,x)+\barV(t,y)\bigr)^{\nicefrac{1}{4}}.
\end{align}}%
This, e.g., the disintegration-type result in \cite[Lemma~2.2]{HJKNW2018}, and  the independence property in item \eqref{t35} imply that for all 
$r\in [0,T]$, $s,t\in [r,T]$, 
$x,y\in\R^d$, $n\in \N_0$, $M\in\N$,
$H\in\{\lambda_1F(U_{n,M}^0)+\lambda_2F(u)+\lambda_3F(0)\colon \lambda_1,\lambda_2,\lambda_3\in \R\}$ it holds that
{\newcommand{\sqrtN}[1]{\left|#1\right|^{\nicefrac{1}{2}}}
\begin{align}
\label{m07}
&\sqrt{T}\lp{{H}\bigl(\funcLambda{\unif^0}{s}, X^0_{s,\funcLambda{\unif^0}{s}}(x)\bigr) -{H}\bigl(\funcLambda{\unif^0}{t}, X^0_{t,\funcLambda{\unif^0}{t}}(y)\bigr)}_{\lpspace{2}}\nonumber \\
&=\lp{\left.\sqrt{T}\lp{\lp{{H}\bigl(\funcLambda{\lambda}{s}, a\bigr) -{H}\bigl(\funcLambda{\lambda}{t}, b\bigr)}_{\lpspace{2}}\bigr|_{\substack{(a,b)=(X^0_{s,\funcLambda{\lambda}{s}}(x) ,X^0_{t,\funcLambda{\lambda}{t}}(y))}}}_{\lpspace{2}}\right|_{\lambda=\unif^0}}_{\lpspace{2}}\nonumber \\
&\leq 
\lp{
\funcH{H}{\funcLambda{\unif^0}{r}}}_{\lpspace{2}}
\left( \cstMD \sqrtN{2(s-t)}+\|x-y\|\right)
(\barV(s,x)+\barV(t,y))^{\nicefrac{1}{4}}.
\end{align}}%
Next note that, e.g., the disintegration-type result in \cite[Lemma~2.2]{HJKNW2018}, 
 the independence property in item \eqref{t35}, and
\eqref{m22}
 imply that for all 
$r\in [0,T]$, $t\in [r,T]$,  $y\in\R^d$,  $n\in \N_0$, $M\in\N$, $H\in \{\lambda_1F(U_{n,M}^0)+\lambda_2F(u)+\lambda_3F(0)\colon \lambda_1,\lambda_2,\lambda_3\in \R\}$ it holds that
{\renewcommand{\sqrt}[1]{\left(#1\right)^{\nicefrac{1}{2}}}
\allowdisplaybreaks
\begin{align}
&\blp{
 {H}\bigl(\funcLambda{\unif^0}{t}, X_{t,\funcLambda{\unif^0}{t}}^0(y)\bigr)}_{\lpspace{2}}
=
\bblp{
\blp{
\lp{
 {H}\bigl(\funcLambda{\lambda}{t}, b\bigr)}_{\lpspace{2}}\bigr|_{b=X_{t,\funcLambda{\lambda}{t}}^0(y)}
}_{\lpspace{2}}\Bigr|_{\lambda=\unif^0}}_{\lpspace{2}}\nonumber \\
&\leq 
\bbblp{
\bblp{\left[\threenorm{H}{1,\funcLambda{\lambda}{r}}
(\barV(\funcLambda{\lambda}{t},b))^{\beta}\right]
\Bigr|_{b=X_{t,\funcLambda{\lambda}{t}}^0(y)}
}_{\lpspace{2}}\biggr|_{\lambda=\unif^0}}_{\lpspace{2}}\nonumber \\
&=
\bblp{
\blp{\|\threenorm{H}{1,\funcLambda{\lambda}{r}}
\bigl(\barV(\funcLambda{\lambda}{t},X^0_{t,\funcLambda{\lambda}{t}}(y))\bigr)^{\beta}
}_{\lpspace{2}}\Bigr|_{\lambda=\unif^0}}_{\lpspace{2}}\leq 
\blp{\threenorm{H}{1,\funcLambda{\unif^0}{r}}}_{\lpspace{2}}(\barV(t,y))^{\beta}.
\end{align}}%
This, the triangle inequality, 
\eqref{m07}, the fact that $c\ge 1$ the fact that $V\geq 1$,  and the fact that 
$0\le \beta\le 1/4$ 
  imply that for all 
$\lambda\in[0,1]$, $r\in [0,T]$, $s,t\in [r,T]$, $x,y\in\R^d$, $n\in \N_0$, $M\in\N$, $H\in \{\lambda_1F(U_{n,M}^0)+\lambda_2F(u)+\lambda_3F(0)\colon \lambda_1,\lambda_2,\lambda_3\in \R\}$ it holds that
{\newcommand{\sqrtN}[1]{\left|#1\right|^{\nicefrac{1}{2}}}
\allowdisplaybreaks
\begin{align}
&\sqrt{T}\blp{(T-s){H}\bigl(\funcLambda{\unif^0}{s},X_{s,\funcLambda{\unif^0}{s}}^0(x)\bigr)
-
(T-t){H}\bigl(\funcLambda{\unif^0}{t},X_{t,\funcLambda{\unif^0}{t}}^0(y)\bigr)}_{\lpspace{2}}  \nonumber \\
&
\leq \sqrt{T}
(T-s)\blp{{H}\bigl(\funcLambda{\unif^0}{s}, X^0_{s,\funcLambda{\unif^0}{s}}(x)\bigr) -{H}\bigl(\funcLambda{\unif^0}{t}, X^0_{t,\funcLambda{\unif^0}{t}}(y)\bigr)}_{\lpspace{2}
}  \nonumber 
\\
&\quad
+\sqrt{T}
|s-t| \blp{{H}\bigl(\funcLambda{\unif^0}{t}, X_{t,\funcLambda{\unif^0}{t}}^0(y)\bigr)}_{\lpspace{2}} \nonumber 
\\
&\leq(T-s)
\blp{\threenorm{H}{2,\funcLambda{\unif^0}{r}}}_{\lpspace{2}}
\left( \cstMD \sqrtN{2(s-t)}+\|x-y\|\right)
(\barV(s,x)+\barV(t,y))^{\nicefrac{1}{4}}  \nonumber 
\\&\quad
+\sqrt{T}|s-t|\lp{\threenorm{H}{1,\funcLambda{\unif^0}{r}}}_{\lpspace{2}}(\barV(t,y))^{\beta} \nonumber 
\\
&\leq \sqrtN{T(T-r)}\biggl[
\blp{\threenorm{H}{1,\funcLambda{\unif^0}{r}}}_{\lpspace{2}}+\sqrt{2}
\blp{\threenorm{H}{2,\funcLambda{\unif^0}{r}}}_{\lpspace{2}}\biggr]
\nonumber   \\&\quad\cdot\left[ \cstMD \sqrtN{s-t}+\|x-y\|\right]
(\barV(s,x)+\barV(t,y))^{\nicefrac{1}{4}}   
.\label{m20}
\end{align}}%
Next observe that
\eqref{m22}, 
\eqref{t02b},
the fact that
$\forall\,x\in\R^d,t\in[0,T]\colon |Tf(t,x,0)|\leq  ({V}(x))^{\beta}$,
the fact that $V\geq 1$, and the fact that $ 0\le \beta\le 1/4$ imply that for all $r\in[0,T]$ it holds that
$\max_{j\in\{1,2\}}\threenorm{TF(0)}{j,r} \leq 1$.
This,  \eqref{m22}, \eqref{m20}, and the fact that $\P(0\leq \unif^0\leq 1)=1 $  show that for all $r\in[0,T]$ it holds that 
\begin{align}
&
\funcH{
[0,T]\times\R^d\times\Omega\ni(s,x,\omega)\mapsto\left[(T-s){(F(0))}\bigl(s+\unif^0(T-s),X_{s,s+\unif^0(T-s)}^0(x)\bigr)\right](\omega)\in\R
}{r}\nonumber 
\\
&\leq 
|(T-r)T|^{\nicefrac{1}{2}}  \max_{\zeta\in [r,T]}\left[
\threenorm{F(0)}{1,\zeta}+\sqrt{2}
\threenorm{F(0)}{2,\zeta}\right]
\leq 3.\label{b03}
\end{align}
Moreover, note that the integral transformation theorem and
the fact that $\unif^0$ is continuous uniformly distributed on $[0,1]$ imply that for all
$r\in [0,T]$ and all measurable  $h\colon [0,T]\to \R$ it holds that
\begin{align}\label{b04}\textstyle
|T-r|^{\nicefrac{1}{2}}\lp{h(\funcLambda{\unif^0}{r})}_2
= \left[\int_{0}^{1}(T-r)
|h(r+(T-r)\lambda)|^2\,d\lambda\right]^{\nicefrac{1}{2}}= 
\left[\int_{r}^{T}
|h(\zeta)|^2\,d\zeta\right]^{\nicefrac{1}{2}}.
\end{align}
This, 
\eqref{m20},
\eqref{m22},   \eqref{m14}, \eqref{m13},
and the fact that
$
TL+\sqrt{2}( 16 e^{2LT}+1)+\sqrt{2}TL\leq 
16\sqrt{2}
e^{2LT}+\sqrt{2}(1+2TL) 
\leq 17\sqrt{2}e^{2LT }\leq 
24.5 e^{2LT }
$
show for all $r\in[0,T]$, $n\in \N_0$, $M\in\N$ that
{ \allowdisplaybreaks\begin{align}
&\threenorm{\begin{aligned}
&[0,T]\times\R^d\times\Omega\ni(s,x,\omega)\\
&\mapsto (T-s)\left[(F(U_{n,M}^0)-F(u))\bigl(s+\unif^0(T-s),X_{s,s+\unif^0(T-s)}^0(x)\bigr)\right](\omega)\in\R
\end{aligned}
}{2,r}\nonumber \\
&
\leq |T(T-r)|^{\nicefrac{1}{2}}\biggl[
\lp{\threenorm{F(U_{n,M}^0)-F(u)}{1,\funcLambda{\unif^0}{r}}}_{\lpspace{2}}+\sqrt{2}
\lp{\threenorm{F(U_{n,M}^0)-F(u)}{2,\funcLambda{\unif^0}{r}}}_{\lpspace{2}}\biggr]\nonumber \\
&
\leq |T(T-r)|^{\nicefrac{1}{2}}\biggl[\left(
L+\sqrt{2}\tfrac{16 e^{2LT}+1}{T}\right)
\lp{\threenorm{U_{n,M}^0-u}{1,\funcLambda{\unif^0}{r}}}_{\lpspace{2}}+
\sqrt{2}L
\lp{\threenorm{U_{n,M}^0-u}{2,\funcLambda{\unif^0}{r}}}_{\lpspace{2}}
\biggr]\nonumber \\
&
\leq \tfrac{1}{\sqrt{T}}T\left[
\left(
L+\sqrt{2}\tfrac{16 e^{2LT}+1}{T}\right)+
\sqrt{2}L\right]
\max_{j\in\{1,2\}}\left[
|T-r|^{\nicefrac{1}{2}}
\lp{\threenorm{U_{n,M}^0-u}{j,\funcLambda{\unif^0}{r}}}_{\lpspace{2}}\right]
\nonumber \\
&\leq \tfrac{24.5}{\sqrt{T}}e^{2LT}
\left[ \int_{r}^{T}\threenorm{U_{n,M}^0-u}{0,\zeta}^2\,d\zeta\right]^{\nicefrac{1}{2}}.
\label{b02}
\end{align}}%
This, \eqref{t27}, \eqref{m22b}, \eqref{m22},  
items \eqref{t34}--\eqref{t35},
 Bienaym\'e's identity,
\eqref{b01}, and \eqref{b03} imply  for all $n,M\in\N$, $r\in[0,T]$  that
{\allowdisplaybreaks\begin{align}
&\threenorm{U^0_{n,M}-\E[U^0_{n,M}]}{2,r}\nonumber 
\leq  \threenorm{[0,T]\times\R^d\times\Omega\ni (t,x,\omega)\mapsto 
\frac{1}{M^n}\sum_{i=1}^{M^n}
\bigl[ g(X^{(0,0,-i)}_{t,T}(x))\bigr](\omega)\in\R}{2,r}\nonumber \\
&\quad +
\sum_{\ell=0}^{n-1}\left[ 
\threenorm{
\begin{aligned}
&[0,T]\times\R^d\times\Omega\ni (t,x,\omega)\mapsto\frac{T-t}{M^{n-\ell}}\sum_{i=1}^{M^{n-\ell}}\bigl[
      \bigl({F}\bigl({{U}}_{\ell,M}^{(0,\ell,i)}\bigr)-\1_{\N}(\ell){F}\bigl( {{U}}_{\ell-1,M}^{(0,-\ell,i)}\bigr)\bigr)\\&
      \bigl(t+(T-t)\unif^{(0,\ell,i)},{X}_{t,t+(T-t)\unif^{(0,\ell,i)}}^{(0,\ell,i)}(x)\bigr)\bigr](\omega)\in\R
\end{aligned}
}{2,r}\right]\nonumber 
\\
&= \frac{1}{\sqrt{M^n}}
     \threenorm{[0,T]\times\R^d\times\Omega\ni (t,x,\omega)\mapsto 
\bigl[ g(X^{0}_{t,T}(x))\bigr](\omega)\in\R}{2,r}\nonumber\\
&
\quad + \sum_{\ell=0}^{n-1}\left[ \frac{1}{\sqrt{M^{n-\ell}}}
\threenorm{
\begin{aligned}
&[0,T]\times\R^d\times\Omega\ni (t,x,\omega)\mapsto \\&(T-t)\bigl[
      \bigl({F}\bigl({{U}}_{\ell,M}^{(0,\ell,1)}\bigr)-\1_{\N}(\ell){F}\bigl( {{U}}_{\ell-1,M}^{(0,-\ell,1)}\bigr)\bigr)\\&
      \bigl(t+(T-t)\unif^{(0,\ell,1)},{X}_{t,t+(T-t)\unif^{(0,\ell,1)}}^{(0,\ell,1)}(x)\bigr)\bigr](\omega)\in\R
\end{aligned}
}{2,r}\right]\nonumber 
\\
&
\leq \frac{1}{\sqrt{M^n}}
     \threenorm{[0,T]\times\R^d\times\Omega\ni (t,x,\omega)\mapsto  \bigl[g(X^{0}_{t,T}(x))\bigr](\omega)\in\R}{2,r}\nonumber\\
&\quad + \frac{1}{\sqrt{M^n}}
\threenorm{\begin{aligned}
&
[0,T]\times\R^d\times\Omega\ni (t,x,\omega)\mapsto\\
&(T-t)\bigl[(F(0))(t+\unif^{0}(T-t),X_{t,t+\unif^{0}(T-t)}^{0}(x))\bigr](\omega)\in\R
\end{aligned}
}{2,r}
\nonumber\\
&
\quad + \sum_{\ell=1}^{n-1}\left[ \frac{1}{\sqrt{M^{n-\ell}}}
\threenorm{
\begin{aligned}
&[0,T]\times\R^d\times\Omega\ni (t,x,\omega)\mapsto (T-t)\bigl[
      \bigl({F}\bigl({{U}}_{\ell,M}^{0}\bigr)-{F}(u)\bigr)\\&
      \bigl(t+(T-t)\unif^{0},{X}_{t,t+(T-t)\unif^{0}}^{0}(x)\bigr)\bigr](\omega)\in\R
\end{aligned}
}{2,r}\right]\nonumber\\
&
\quad + \sum_{\ell=1}^{n-1}\left[ \frac{1}{\sqrt{M^{n-\ell}}}
\threenorm{
\begin{aligned}
&[0,T]\times\R^d\times\Omega\ni (t,x,\omega)\mapsto (T-t)\bigl[
      \bigl({F}\bigl({{U}}_{\ell-1,M}^{0}\bigr)-{F}(u)\bigr)\\&
      \bigl(t+(T-t)\unif^{0},{X}_{t,t+(T-t)\unif^{0}}^{0}(x)\bigr)\bigr](\omega)\in\R
\end{aligned}
}{2,r}\right]\nonumber \\
&\leq 
\frac{4}{\sqrt{M^n}}
+\sum_{\ell=0}^{n-1}\left[\frac{(2-\1_{\{n-1\}}(\ell))24.5 T^{-\nicefrac{1}{2}} e^{2LT} }{\sqrt{M^{n-\ell-1}}}
\left[\int_r^T
\threenorm{U^0_{\ell,M}-u}{0,\zeta}^2\,d\zeta
\right]^{\nicefrac{1}{2}}\right]
.\label{t32}
\end{align}}%
Next observe that  \eqref{t31},  
\eqref{m22},
Jensen's inequality, and
\eqref{b02} assure that  for all $n,M\in\N$, $r\in[0,T]$ it holds that
{\allowdisplaybreaks
\begin{align}
&\threenorm{\E[U^0_{n,M}]-u}{2,r}
=
\threenorm{\begin{aligned}
&[0,T]\times\R^d\times\Omega\ni (t,x,\omega)\mapsto\\
&\E\bigl[(T-t)(F(U_{n-1.M}^{0})-F(u))(t+\unif^{0}(T-t),X_{t,t+\unif^{0}(T-t)}^{0}(x))\bigr]\in\R 
\end{aligned}
}{2,r}\nonumber 
\\
&\leq  \threenorm{\begin{aligned}
&[0,T]\times\R^d\times\Omega\ni (t,x,\omega)\mapsto\\&(T-t)\bigl[(F(U_{n-1.M}^{0}-F(u))(t+\unif^{0}(T-t),X_{t,t+\unif^{0}(T-t)}^{0}(x))\bigr](\omega)\in\R
\end{aligned}}{2,r}\nonumber \\
&\leq 24.5 T^{-\nicefrac{1}{2}} e^{2LT} 
\left[\int_r^T
\threenorm{U^0_{\ell,M}-u}{0,\zeta}^2\,d\zeta
\right]^{\nicefrac{1}{2}}.
\end{align}}%
This, \eqref{t32},  and the triangle inequality  
show  for all $n,M\in\N$, $r\in[0,T]$  that
\begin{align}\begin{split}
\threenorm{U^0_{n,M}-u}{2,r}
&\leq \frac{4}{\sqrt{M^n}}
+\sum_{\ell=0}^{n-1}\left[\frac{\luckyCst T^{-\nicefrac{1}{2}} e^{2LT} }{\sqrt{M^{n-\ell-1}}}
\left[\int_r^T
\threenorm{U^0_{\ell,M}-u}{0,\zeta}^2\,d\zeta
\right]^{\nicefrac{1}{2}}\right].
\end{split}\label{k09}\end{align}
Moreover, note that \cite[Lemma~3.5]{HJKN20} (applied
for every $s\in[0,T]$
 with 
$\rho\defeq 2\beta\rho$,
$\varphi\defeq V^{2\beta}$,
$Y\defeq X$,
$t\defeq s$
 in the notation of \cite[Lemma~3.5]{HJKN20}),
\eqref{t02b},
the fact that $\forall\,M\in \N\colon U_{0,M}^0=0$,
\eqref{t27},
\eqref{t27b},
\eqref{m23b}, and
\eqref{t10}
 prove that for all $s\in[0,T]$, $M,n\in\N$ it holds that
\begin{align}\begin{split}
&\sup_{x\in\R^d}\left[
\frac{e^{\beta\rho s}\lp{{U}_{n,M}^{0}(s,x)-u(s,x)}_{\lpspace{2}}}{(V(x))^\beta}\right]\\
&
\leq 
\frac{2e^{\beta \rho T}}{\sqrt{M^n}}+
\sum_{\ell=0}^{n-1}\left[
\frac{2L(T-s)^{\nicefrac{1}{2}}}{\sqrt{M^{n-\ell-1}}}
\left(
\int_{s}^{T}\sup_{\eta\in[\zeta,T]}
\sup_{x\in\R^d}\left[\frac{e^{2\beta \rho \eta}
\lp{
U^{0}_{\ell,M}(\eta,x) -u(\eta,x)}^2_{\lpspace{2}}}{(V(x))^{2\beta}}
\right]d\zeta\right)^{\!\nicefrac{1}{2}}\right].
\end{split}\end{align}
Combining this,
the fact that $\forall\,t\in[0,T],  x\in\R^d\colon\barV(t,x)= e^{\rho (T-t)}V(x)$, and \eqref{m22}  ensures that for all $r\in[0,T]$,  $M,n\in\N$ it holds that
{\allowdisplaybreaks
\begin{align}
&\threenorm{{U}_{n,M}^{0}-u}{1,r}=
\sup_{s\in[r,T]}\sup_{x\in\R^d}
\frac{\lp{{U}_{n,M}^{0}(s,x)-u(s,x)}_{\lpspace{2}}}{(\barV(s,x))^\beta}=\sup_{s\in[r,T]}
 \sup_{x\in\R^d}
\frac{\lp{{U}_{n,M}^{0}(s,x)-u(s,x)}_{\lpspace{2}}}{e^{\beta\rho (T-s)}(V(x))^\beta}
\nonumber \\
&\leq 
\frac{2}{\sqrt{M^n}}+\sup_{s\in [r,T]}
\sum_{\ell=0}^{n-1}\left[
\frac{T^{-\nicefrac{1}{2}}e^{2LT}}{\sqrt{M^{n-\ell-1}}}
\left(
\int_{s}^{T}
\sup_{\eta\in [\zeta,T]}
\sup_{x\in\R^d}\left[\frac{
\lp{
U^{0}_{\ell,M}(\eta,x) -u(\eta,x)}_{\lpspace{2}}}{e^{\beta\rho (T-\eta)}(V(x))^\beta}
\right]^2d\zeta\right)^{\!\nicefrac{1}{2}}\right]
\nonumber \\
&= 
\frac{2}{\sqrt{M^n}}+
\sum_{\ell=0}^{n-1}\left[
\frac{T^{-\nicefrac{1}{2}}e^{2LT}}{\sqrt{M^{n-\ell-1}}}
\left(
\int_{r}^{T}
\sup_{\eta\in [\zeta,T]}
\sup_{x\in\R^d}\left[\frac{
\lp{
U^{0}_{\ell,M}(\eta,x) -u(\eta,x)}_{\lpspace{2}}}{
(\barV(\eta,x))^\beta}
\right]^2d\zeta\right)^{\!\nicefrac{1}{2}}\right]
\nonumber \\
&=
\frac{2}{\sqrt{M^n}}+
\sum_{\ell=0}^{n-1}\left[
\frac{T^{-\nicefrac{1}{2}}e^{2LT}}{\sqrt{M^{n-\ell-1}}}
\left(
\int_{r}^{T}
\threenorm{U^0_{\ell,M}-u}{1,\zeta}^2d\zeta\right)^{\!\nicefrac{1}{2}}\right].
\end{align}}%
This, \eqref{m22}, and \eqref{k09} demonstrate that for all  $r\in[0,T]$,  $M,n\in\N$ it holds that
\begin{align}\begin{split}
&\threenorm{{U}_{n,M}^{0}-u}{r}
\leq \frac{4}{\sqrt{M^n}}
+\sum_{\ell=0}^{n-1}\left[\frac{\luckyCst T^{-\nicefrac{1}{2}}  e^{2LT} }{\sqrt{M^{n-\ell-1}}}
\left[\int_r^T
\threenorm{U^0_{\ell,M}-u}{0,\zeta}^2\,d\zeta
\right]^{\nicefrac{1}{2}}\right].
\end{split}\label{k22}\end{align}
 Combining
\cite[Lemma~3.10]{HJKN20}
(applied for every $M,N\in\N$, $r\in[0,T]$ with 
 $a\defeq 4$, $b\defeq \luckyCst T^{-\nicefrac{1}{2}} e^{2LT}$, $c\defeq 1/\sqrt{M}$, $\alpha\defeq 0$, $\beta\defeq T$,
$(f_n)_{n\in[0,N]\cap\N_0 }\defeq ([0,T]\ni s\mapsto\threenorm{{U}_{n,M}^{0}-u}{s}\in[0,\infty])_{n\in[0,N]\cap\N_0}
$ in the notation of \cite[Lemma~3.10]{HJKN20}), 
the fact that $\forall\,k\in \N_0\colon M^{k}/k!\leq e^M$,
and \eqref{k10} hence assures that for all $r\in[0,T]$,
$M,N\in\N$ it holds that 
\begin{align}
&\threenorm{{U}_{N,M}^{0}-u}{0,r}\nonumber 
\\
&\leq \left[4+\luckyCst T^{-\nicefrac{1}{2}} e^{2LT}T^{\nicefrac{1}{2}}\sup_{s\in[r,T]}\threenorm{{U}_{0,M}^{0}-u}{0,s}\right]\left[
\sup_{k\in[0,N]\cap\Z}
\frac{M^{-(N-k)/2}}{\sqrt{k!}}\right]
\left(1+\luckyCst T^{-\nicefrac{1}{2}} e^{2LT}T^{\nicefrac{1}{2}}\right)^{N-1}
\nonumber \\
&\leq \left(4+\luckyCst e^{2LT}8e^{2LT}\right)
e^{M/2}M^{-N/2}\left(1+\luckyCst e^{2LT}\right)^{N-1}
\leq 
e^{M/2}M^{-N/2}\left(\luckyCstT e^{2LT}\right)^{N+1}
.
\end{align}
The fact that $\forall\,M\in \N\colon U_{0,M}^0=0$
and \eqref{k10} therefore show that 
 for all $r\in[0,T]$,
$M\in\N$, $N\in\N_0$ it holds that 
\begin{align}\begin{split}
&\threenorm{{U}_{N,M}^{s,0}-u}{r}
\leq 
e^{M/2}M^{-N/2}\left(\luckyCstT e^{2LT}\right)^{N+1}
.
\end{split}\end{align}
This establishes item \eqref{k21b} and item \eqref{k21bb}. The proof of \cref{m02} is thus complete.
\end{proof}

\section{Computational complexity analysis for MLP approximations for backward stochastic differential equations (BSDEs)}
\label{sec:compcompl}

In this section we combine the findings from Sections \ref{sec:2} and \ref{sec:compltwoplus} to supply in \cref{t26} and \cref{cor:sec4} computational complexity analyses for the Monte Carlo-type approximation algorithm for BSDEs in \eqref{eq:mlp_intro}--\eqref{eq:approx_intro} in \cref{h01} in \cref{sec:intro} above. 
\cref{cor:sec4} specializes \cref{t26} to the specific situation where the driver of the BSDEs is twice continuously differentiable with bounded derivatives and does neither depend on the time variable $t \in [0,T]$ nor on the space variable $x \in \R^d$ but only on the solution processes $Y^d : [0,T]\times \Omega \to \R$, $d \in \N$, of the BSDEs under consideration. 

Our proof of \cref{cor:sec4} uses beside \cref{t26} also the elementary Lipschitz-type estimate for twice continuously differentiable functions in \cref{f01} below. For completeness we also include in this section a detailed proof for \cref{f01}. Our proof of \cref{t26}, in turn, employs \cref{n002} from \cref{sec:2} and \cref{m02} from \cref{sec:compltwoplus}. 
\cref{h01} in the introduction is a direct consequence of \cref{cor:sec4}.

\begin{theorem}\label{t26}Assume \cref{m01},
let $\alpha\in\N$,
 $(\thetaBar_n)_{n\in\N_0}\subseteq\Theta$,
let $\lfloor \cdot \rfloor_M \colon \R \to \R$, $ M \in \N $, and 
$\lceil \cdot \rceil_M \colon  \R \to \R$, $ M \in \N $, 
satisfy for all $M \in \N$, $t \in [0,T]$ that
$\lfloor t \rfloor_M = \max( ([0,t]\backslash \{T\}) \cap \{ 0, \frac{ T }{ M }, \frac{ 2T }{ M }, \ldots \} )$
and 
$\lceil t \rceil_M = \min(((t,\infty) \cup \{T\})\cap  \{ 0, \frac{ T }{ M }, \frac{ 2T }{ M }, \ldots \} )$,
let $\Yappr^{n,M}
\colon [0,T]\times\Omega\to\R $,
 $n,M\in\N$, 
satisfy  for all $n,M\in\N$, $t\in[0,T]$  that
\begin{align}\label{d04b} 
\Yappr^{n,M}_{t}
&= \sum_{\ell=0}^{n-1}\biggl[
\left[ \tfrac{ \lceil t \rceil_{M^{l+1}} - t }{ ( T / M^{ l + 1 } ) } \right]
U^{\thetaBar_\ell}_{n-\ell,M}(\lfloor t \rfloor_{M^{l+1}}, W_{\lfloor t \rfloor_{M^{l+1}}})+
\left[\tfrac{ t-\lfloor t \rfloor_{M^{l+1}} }{ ( T / M^{ l + 1 } ) }\right]
U^{\thetaBar_\ell}_{n-\ell,M}(\lceil t \rceil_{M^{l+1}}, W_{\lceil t \rceil_{M^{l+1}}})
\nonumber
\\
&\quad-\1_{\N}(\ell)\Bigl(
\left[ \tfrac{ \lceil t \rceil_{M^{l}} - t }{ ( T / M^{ l  } ) } \right]U^{\thetaBar_\ell}_{n-\ell,M}(\lfloor t \rfloor_{M^{l}}, W_{\lfloor t \rfloor_{M^{l}}})+
\left[ \tfrac{ t-\lceil t \rceil_{M^{l}} }{ ( T / M^{ l  } ) } \right]
U^{\thetaBar_\ell}_{n-\ell,M}(\lceil t \rceil_{M^{l}}, W_{\lceil t \rceil_{M^{l}}})
\Bigr)\biggr],
\end{align}
and let
$\FEU{n,M}\in\N_0$, $n,M\in\Z$, and $\FEY{n,M}\in\N_0$, $n,M\in\Z$, 
satisfy 
for all $n,M\in \N_0$ that 
\begin{equation}
\label{d01}
 \FEU{n,M}\leq \alpha M^n \1_{\N}(n)+ \sum_{\ell=0}^{n-1}\left[M^{n-\ell}\left(1+\alpha+\FEU{\ell,M}+\FEU{\ell-1,M}\1_{\N}(\ell)\right)\right]
 \end{equation}
and $
\FEY{n,M}\leq \alpha(M^n+1)+\sum_{\ell=0}^{n-1} \left[(M^{\ell+1}+1)\FEU{n-\ell,M}\right]$.
Then 
\begin{enumerate}[(i)]
\item\label{t29} there exists an
$(\F_{t})_{t\in[0,T]}$-predictable stochastic process 
$\mathbf{Y}=(Y,Z)=(Y,Z^1,Z^2,\ldots,Z^d) \colon \allowbreak [0,T]\times\Omega\to \R \times\R^d
$
with
$
\int_0^T  \E \!\left[|Y_t|+\|Z_t\|^2\right]dt<\infty
$
which satisfies
that
for all $t\in[0,T]$ it holds $\P$-a.s.\ that
\begin{align}
Y_t=g(W_T)+\int_t^T f(s,W_s,Y_s)\,ds-\sum_{j=1}^{d}\int_t^T Z_s^j\, dW_s^j,
\label{t29b}
\end{align}
\item\label{t28} it holds for all $M,n\in \N$,  $t\in[0,T]$ that
$\Yappr^{n,M}_t$ is 
measurable,
\item \label{t28x} it holds for all $M,n\in \N$, $t\in[0,T]$ that
\begin{equation}
\bigl(
\E\bigl[
|
\Yappr^{n,M}_t-Y_t|^2\bigr]\bigr)^{\nicefrac{1}{2}}
\leq 8 n
e^{M/2+4nLT+\rho T/2}M^{-n/2}
\luckyCstT^{2n}
\bigl\lvert V(0) \max\! \left\{\E\!\left[\|\stdNormal^0\|^4\right],1\right\}\bigr\rvert^{\!\nicefrac{1}{4}},
\end{equation}
and
\item\label{t28b} there exists
$\sfN\colon (0,\infty) \to \N$ such that
for all $\epsilon,\delta\in (0,1]$ it holds
  that
$\sup_{t\in [0,T]}
(\E[
|Y_t-
\Yappr^{\sfN(\epsilon),\sfN(\epsilon)}_t|^2])^{1/2}
\leq\epsilon
$ and
\begin{align}
&\FEY{\sfN(\epsilon),\sfN(\epsilon)}\\
&\leq \alpha \left(\sup_{n\in\N}\left[\tfrac{10^{n+3}n^3 \left[8ne^{n/2+4nLT}\luckyCstT^{2n}\right]^{2+\delta} }{n^{\delta n/2}}\right]\right)
\bigl[e^{2\rho T} V(0) \max\{\E\!\left[\|\stdNormal^0\|^4\right],1\}\bigr]^{\frac{2+\delta}{4}}\epsilon^{-(2+\delta)}<\infty.\nonumber
\end{align}
\end{enumerate}
\end{theorem}
\begin{proof}[Proof of \cref{t26}]\sloppy
Throughout this proof
let 
$\cstMD\in[1,\infty) $ satisfy 
$
\cstMD= (\max\{\E [\|\stdNormal^0\|^4],1\})^{\nicefrac{1}{4}}
$, let $\barV\colon [0,T]\times\R^d\to [1,\infty) $ satisfy 
for all $t\in[0,T]$,  $x\in\R^d$ that $\barV(t,x)= e^{\rho (T-t)}V(x)$,
for every $q\in [1,\infty)$ and every random variable 
$\mathfrak{X}\colon\Omega\to\R$
let $\lp{\mathfrak{X}}_{q} \in [0,\infty]$
satisfy that $\lp{\mathfrak{X}}_{q}= (\E[|\mathfrak{X}|^q])^{1/q}$,
for every $r\in [0,T]$ and
for
 every random field
$H\colon [0,T]\times \R^d\times \Omega\to\R $ let $\threenorm{H}{j,r},\in[0,\infty] $, $j\in \{0,1,2\}$, satisfy that
\begin{align}\label{m22c}
\begin{split}
\threenorm{H}{0,r}&=\max _{j\in\{1,2\}}\threenorm{H}{j,r},\quad 
\funcN{H}{r}= \sup_{x\in\R^d,s\in [r,T]}\left[ \frac{\left(\E\!\left[|H(s,x)|^2\right]\right)^{\nicefrac{1}{2}}}{(\barV(s,x))^{\beta}}\right],\quad\text{and}\quad
\\
\funcH{H}{r}&=\sup_{\substack{s,t\in [r,T],x,y\in\R^d\\(s,x)\neq(t,y)}}\left[\frac{T^{\nicefrac{1}{2}}\left(\E\!\left[|H(s,x)-H(t,y)|^2\right]\right)^{\nicefrac{1}{2}}}{
\bigl[
\cstMD
|s-t|^{\nicefrac{1}{2}}+\|x-y\|\bigr]\bigl(\barV(s,x)+\barV(t,y)\bigr)^{\nicefrac{1}{4}}}\right]
\end{split}\end{align}
 and let $\bbM\colon\Omega\to\R$ and $\mathscr{Y}\colon[0,T]\times\Omega\to\R$ satisfy  for all $s\in[0,T]$  that
\begin{align}
\mathscr{Y}_s=u(s,W_s)\qquad\text{and}\qquad
\bbM={g}(W_{T}) +
\int_0^T 
f(t,W_{t},\mathscr{Y}_t)\,dt.\label{c04}
\end{align}
Note that it is well-known that \eqref{t02b} and \eqref{m11b} imply item \eqref{t29} (cf., e.g., \cite[Theorem~4.3.1]{Zha17}).
Next observe that \cref{m02}
proves that
\begin{enumerate}[(a)]
\item   
 there exists a unique 
measurable 
$u\colon  [0,T]\times\R^d\to\R$ 
which satisfies
for all   $t\in[0,T]$, $x\in\R^d$  that
$
\E\bigl[|{g}(X^0_{t,T}(x))|\bigr]+\int_t^T\E\bigl[| ({F}(u))(s,X^0_{t,s}(x))| \bigr]\,ds +
\sup_{r\in[0,T],y\in\R^d}
\bigl(\frac{|u(r,y)|}{|V(y)|^\beta}\bigr)<\infty$ and
\begin{align}\label{t28d}
\begin{split}
u(t,x)=\E\!\left[{g}(X^0_{t,T}(x))\right]+\int_t^T
\E\!
\left[ \left({F}(u)\right)\!\left(s,X^0_{t,s}(x)\right)  \right]ds,
\end{split}
\end{align} 
\item it holds that $U_{n,M}^\theta$, $\theta\in\Theta$, $n\in\N_0$, $M\in\N$, are measurable,
and
\item it holds
 for all $r\in[0,T]$,
$M\in\N$, $N\in\N_0$ that 
\begin{align}\label{k21c}\begin{split}
&\threenorm{{U}_{N,M}^{0}-u}{0,r}
\leq 
e^{M/2}M^{-N/2}\left(\luckyCstT e^{2LT}\right)^{N+1}
.
\end{split}\end{align}
\end{enumerate}
This and the fact that $W$ is measurable
establish item \eqref{t28}.
Moreover, observe that 
\eqref{c04},
the triangle inequality, and
\eqref{t02b}
 prove that
\begin{align}
|\bbM|\leq |{g}(W_{T})|+\int_0^T |f(t,W_t,0)|+L |u(t,W_t)|\,dt.
\end{align}
This, the fact that 
$\sup_{x\in\R^d,t\in[0,T]} \bigl(\frac{|Tf(t,x,0)|+|g(x)|+|u(t,x)|}{ |{V}(x)|^{\beta}}\bigr)<\infty$,
the fact that $0\le \beta\le 1/2$,
the fact that $\forall\,x\in\R^d,s\in[0,T], t\in[s,T]\colon\E[V(x+W_{t-s})]\leq e^{\rho(t-s)}V(x)$, and Jensen's inequality
 imply that $\E[|\bbM|^2]<\infty$. This, the fact that 
 for all $A\in \mathcal{B}(\R)$ it holds that $\bbM^{-1}(A)\in \F_T$
and the martingale representation theorem (see, e.g., \cite[Theorem~4.3.4]{Oek03}) imply that there exists an $(\F_s)_{s\in[0,T]}$-progressively measurable stochastic process $\mathscr{Z}=(\mathscr{Z}^1,\mathscr{Z}^2,\ldots, \mathscr{Z}^d)\colon [0,T]\times\Omega\to\R^d$ which satisfies
that
  for all $s\in[0,T]$ it holds  $\P$-a.s.  that
$
\E[\bbM|\F_s]= \E[\bbM]+\sum_{j=1}^{d}\int_0^s\mathscr{Z}_r^j\,dW^j_r.
$
The fact that 
 for all $A\in \mathcal{B}(\R)$ it holds that $\bbM^{-1}(A)\in \F_T$
hence shows that for all $s\in[0,T]$ it holds $\P$-a.s.\ that
\begin{align}
\bbM-
\E[\bbM|\F_s]= 
\E[\bbM|\F_T]-
\E[\bbM|\F_s]= 
\sum_{j=1}^{d}\int_s^T\mathscr{Z}_r^j\,dW^j_r.\label{c05}
\end{align}
Next note that \eqref{t28d}
and the fact that $\forall\,t\in [0,T], B\in\mathcal{B}(\R^d)\colon 
\P( W_t\in B)=\P( \stdNormal^0\sqrt{t}\in B)$ show that for all $s\in [0,T]$, $z\in\R^d$ it holds that
\begin{align}\begin{split}
u(s,z)
=
\E\!\left[{g}(z+W_{T-s})+\int_s^T 
f(t,z+W_{t-s},u(t,z+W_{t-s}))\, dt\right] .
\end{split}\end{align}
The Markov property of Brownian motions,
\eqref{c04}, the fact that for all $s\in[0,T]$, $z\in\R^d$, $A\in \mathcal{B}(\R)$ it holds that
$\{ \omega \in \Omega \colon \int_0^s
f(t,z+W_{t}(\omega),u(t,z+W_{t}(\omega)))\,dt \in A \} \in \F_s$,
and \eqref{c05} hence
 show that for all $s\in[0,T]$, $x\in\R^d$ it holds $\P$-a.s.\ that
{\allowdisplaybreaks\begin{align}
&\mathscr{Y}_s=u(s,W_s)
= \left.\E\!\left[{g}(z+W_{T-s})+\int_s^T 
f(t,z+W_{t-s},u(t,z+W_{t-s}))\, dt\right]\right|_{z=W_s}
\nonumber
\\
&
=\E\!\left[{g}(W_{T}) +
\int_s^T 
f(t,W_{t},u(t,W_{t}))\,dt \middle| \F_{s}\right] \nonumber\\
&=
{g}(W_{T}) +
\int_s^T 
f(t,W_{t},\mathscr{Y}_t)\,dt\nonumber\\
&\quad-\left(
{g}(W_{T}) +
\int_s^T 
f(t,W_{t},\mathscr{Y}_t)\,dt
-
\E\!\left[{g}(W_{T}) +
\int_s^T 
f(t,W_{t},\mathscr{Y}_t)\,dt \middle| \F_{s}\right] \right)\nonumber\\
&= {g}(W_{T}) +
\int_s^T 
f(t,W_{t},\mathscr{Y}_t)\,dt - \left(\bbM-\E [\bbM|\F_s]\right)\nonumber\\
&= {g}(W_{T}) +
\int_s^T 
f(t,W_{t},\mathscr{Y}_t)\,dt - 
\sum_{j=1}^{d}\int_s^T\mathscr{Z}_r^j\,dW^j_r.
\end{align}}%
Combining item \eqref{t29} and the fact that $[0,T]\times \Omega \ni (t,\omega)\mapsto (\mathscr{Y}_t(\omega),\mathscr{Z}_t(\omega))\in \R\times \R^d$ is $(\F_t)_{t\in[0,T]}$-progressively measurable hence implies that for all
$t\in[0,T]$ 
it holds $\P$-a.s.\ that $(Y_t,Z_t)=(\mathscr{Y}_t,\mathscr{Z}_t)$. This, \eqref{c04} and the fact that
$Y$ and $\mathscr{Y}$  have continuous sample paths (see \eqref{c04} and item \eqref{t29})
 prove that $\P$-a.s. it holds for all $t\in[0,T]$ that
\begin{align}
Y_t=u(t,W_t).\label{t30b}
\end{align}
Next note that the assumption that 
$
(\unif^\theta)_{ \theta\in\Theta}$, $
(\stdNormal^\theta)_{ \theta\in\Theta}$, and
$W$
 are independent, 
the fact that $\forall\,M\in \N,\theta\in\Theta\colon U_{0,M}^\theta=0$,
and \eqref{t27} imply that 
$
(U^\theta_{n,M})_{  n\in\N_0, M\in\N, \theta\in\Theta}
$ and $W$ are independent. Combining
\eqref{m22c}, 
the fact that
$\forall\,t\in[0,T],  x\in\R^d\colon\barV(t,x)= e^{\rho (T-t)}V(x)$,
H\"older's inequality, 
the fact that $0\le \beta \le 1/4$, the fact that $V\geq 1$,
the triangle inequality,
the fact that for all $n,M\in\N$ it holds that $U^\theta_{n,M}$,  $\theta\in \Theta$, are identically distributed,
 \eqref{k21c}, and, e.g., the 
disintegration-type result in \cite[Lemma~2.2]{HJKNW2018} hence
 implies that for all $s\in[0,T]$, $t\in[s,T]$, $x\in\R^d$, $M\in\N$,
$n\in\N_0$, $\theta\in\Theta$ it holds that  
\begin{align}\begin{split}
&
\blp{
(U_{n,M}^\theta-u)(t,x+W_t)}_{\lpspace{2}}
=\blp{\lp{
(U_{n,M}^\theta-u)(t,y)}_{\lpspace{2}}\bigr|_{y=x+W_t}}_{\lpspace{2}}
\\
&
\leq 
\blp{\threenorm{U_{n,M}^\theta-u}{1,t}(\barV(t,X^0_{0,t}(x)))^{\beta}}_{\lpspace{2}}
= \blp{\threenorm{U_{n,M}^\theta-u}{1,t}e^{\beta\rho(T-t)}(V(X^0_{0,t}(x)))^{\beta}}_{\lpspace{2}}\\
&
\leq 
\threenorm{U_{n,M}^0-u}{0,t} e^{\beta\rho T}(V(x))^{\beta}
\leq \threenorm{U_{n,M}^0-u}{0,t} e^{\rho T/4}(V(x))^{\nicefrac{1}{4}}
\\
&
\leq e^{M/2}M^{-n/2}\left[\luckyCstT e^{2LT}\right]^{n+1}(V(x))^{\nicefrac{1}{4}} e^{\rho T/4}
\label{k13}
\end{split}\end{align}
and
{\renewcommand{\sqrt}[1]{\left|#1\right|^{\nicefrac{1}{2}}}\allowdisplaybreaks
\begin{align}
&\blp{(U_{n,M}^\theta-u)(s,x+W_s)-(U_{n,M}^\theta-u)(t,x+W_t)}_{\lpspace{2}}\nonumber \\
&
=\blp{\lp{(U_{n,M}^\theta-u)(s,a)-(U_{n,M}^\theta-u)(t,b)}_{\lpspace{2}}\bigr|_{(a,b)=(x+W_s,x+W_t)}}_{\lpspace{2}}\nonumber \\
&\leq 
\bblp{\threenorm{U_{n,M}^\theta-u}{2,s}T^{-\nicefrac{1}{2}}
\left[\cstMD\sqrt{s-t}+\|a-b\|
\right](\barV(s,a)+\barV(t,b))^{\nicefrac{1}{4}}\Bigr|_{\substack{(a,b)=(x+W_s, x+W_t)}}
}_2\nonumber \\
&\leq 
\bblp{\threenorm{U_{n,M}^\theta-u}{2,s}T^{-\nicefrac{1}{2}}e^{\rho T/4}
\left[\cstMD\sqrt{s-t}+\|a-b\|
\right](V(a)+V(b))^{\nicefrac{1}{4}}\Bigr|_{\substack{(a,b)=(x+W_s,x+W_t)}}
}_2\nonumber \\
&\leq
\threenorm{U_{n,M}^\theta-u}{0,s}T^{-\nicefrac{1}{2}} e^{\rho T/4}
 \left[\cstMD\sqrt{s-t}+\lp{\|W_s-W_t\|}_{\lpspace{4}}
\right] \bigl(\E\!\left[V(x+W_s)+V(x+W_t)\right]\bigr)^{\nicefrac{1}{4}}\nonumber
\\
&\leq 4e^{M/2}M^{-n/2}\left(\luckyCstT e^{2LT}\right)^{n+1}T^{-\nicefrac{1}{2}} e^{\rho T/2}\cstMD\sqrt{s-t}(V(x))^{\nicefrac{1}{4}}.
\label{k11}
\end{align}}%
The fact that $\forall\, M\in \N\colon U_{0,M}^0=0$ therefore assures
 for all $s\in[0,T]$, $t\in[s,T]$, $x\in\R^d$ that 
\begin{align}\renewcommand{\sqrt}[1]{\left|#1\right|^{\nicefrac{1}{2}}}
\begin{split}
&\lp{u(t,x+W_t)-u(s,x+W_s)}_{\lpspace{2}}\leq  4e^{M/2}\left(\luckyCstT e^{2LT}\right)T^{-\nicefrac{1}{2}} e^{\rho T/2}\cstMD\sqrt{t-s}(V(x))^{\nicefrac{1}{4}}.\label{k11b}
\end{split}\end{align}
Combining
\eqref{t30b}, 
\cref{n002} (applied for every $n,M\in\N$ with 
$V\defeq 
\{\mathfrak{Z}\colon \Omega\to\R \colon\mathfrak{Z} \text{ is }
\text{measurable}
\}
$,
$\lVert \cdot \rVert \defeq \lp{\cdot}_2$,
$\alpha\defeq 1/2$,
$(m_l)_{l\in \{1,2,\ldots,n\}}\defeq (M^l)_{l\in \{1,2,\ldots,n\}}$,
$
(\tau_{l,k})_{k\in  \{0,1,\ldots,m_l\}, l\in \{1,2,\ldots,n\}}
\defeq 
(\frac{kT}{M^l})_{k\in  \{0,1,\ldots,M^l\}, l\in \{1,2,\ldots,n\}}
$,
$(Y^0_t)_{t\in[0,T]}\defeq (u(t,W_t))_{t\in[0,T]}$,
$((Y^\ell_t)_{t\in[0,T]}
)_{\ell\in[1,n]\cap\N}
\defeq
((U^{\thetaBar_\ell}_{\ell,M}(t,W_t))_{t\in[0,T]})_{\ell\in[1,n]\cap\N}$,
$\Yappr \defeq \Yappr^{n,M}$ in the notation  of \cref{n002}),
\eqref{k13},  and
\eqref{k11}
hence demonstrates that for all $n,M\in\N$ it holds that
{\allowdisplaybreaks\begin{align}
&
\sup_{t\in[0,T]}\lp{\Yappr^{n,M}_t-Y_t}_2= 
\sup_{t\in[0,T]}\lp{\Yappr^{n,M}_t-u(t,W_t)}_2
\leq \sup_{t\in[0,T]}\lp{U^{\thetaBar_n}_{n,M}(t,W_t)-u(t,W_t)}_2\nonumber \\
&\quad
 + 2^{-1/2}T^{1/2}M^{-n/2}\left[\sup_{t,s\in[0,T],t\neq s}\frac{\lp{u(t,W_t)-u(s,W_s)}_2}{|t-s|^{\nicefrac{1}{2}}}\right]\nonumber \\
&\quad +
\sum_{\ell=1}^{n-1}\left[2^{-1/2}T^{1/2}
M^{-l/2}\left[\sup_{t,s\in[0,T],t\neq s}
\frac{\lp{
(U^{\thetaBar_{n-\ell}}_{n-\ell,M}-u)(t,W_t)
-
(U^{\thetaBar_{n-\ell}}_{n-\ell,M}-u)(s,W_s)
}_2}{|t-s|^{\nicefrac{1}{2}}}\right]\right]\nonumber \\
\nonumber \\
&\leq e^{M/2}M^{-n/2}\left[\luckyCstT e^{2LT}\right]^{n+1}(V(0))^{\nicefrac{1}{4}} e^{\rho T/2}
+2^{-1/2}T^{1/2}M^{-n/2}
4e^{M/2}\left(\luckyCstT e^{2LT}\right)T^{-\nicefrac{1}{2}} e^{\rho T/2}\cstMD(V(0))^{\nicefrac{1}{4}}\nonumber \\
&\quad +\sum_{\ell=1}^{n-1}\left[
2^{-1/2}T^{1/2}M^{-\ell/2}4e^{M/2}M^{-(n-\ell)/2}\left(\luckyCstT e^{2LT}\right)^{n-\ell+1}T^{-\nicefrac{1}{2}} e^{\rho T/2}\cstMD(V(0))^{\nicefrac{1}{4}}\right].
\end{align}}
The fact that $\cstMD \ge 1$ therefore proves that for all $n,M\in\N$ it holds that
\begin{equation}
\begin{split}
&\sup_{t\in[0,T]}\lp{\Yappr^{n,M}_t-Y_t}_2
\leq 
e^{M/2}M^{-n/2}(V(0))^{\nicefrac{1}{4}} e^{\rho T/2}
\\
&\quad
\cdot
\left[
 \left(\luckyCstT e^{2LT}\right)^{n+1}
 +
 2^{-1/2}4\cstMD\left(\luckyCstT e^{2LT}\right)
 +
 2^{-1/2}4\cstMD\sum_{\ell=1}^{n-1}
\left(\luckyCstT e^{2LT}\right)^{n-\ell+1} 
\right]\\
&\leq 
\cstMD e^{M/2}M^{-n/2}(V(0))^{\nicefrac{1}{4}} e^{\rho T/2}
\left[
 \left(\luckyCstT e^{2LT}\right)^{n+1}
 +
 2^{-1/2}4\left(\luckyCstT e^{2LT}\right)
 +
 2^{-1/2}4n
\left(\luckyCstT e^{2LT}\right)^{n} 
\right]
\\
&=
\cstMD e^{M/2}M^{-n/2}(V(0))^{\nicefrac{1}{4}} e^{\rho T/2}\luckyCstT^{2n}
e^{4nLT}\\
&\quad \cdot
\left[
 \left(\luckyCstT e^{2LT}\right)^{-n+1}
 +
 2^{-1/2}4\left(\luckyCstT e^{2LT}\right)^{1-2n}
 +
 2^{-1/2}4n
\left(\luckyCstT e^{2LT}\right)^{-n} 
\right]
\\
&\leq 
\cstMD e^{M/2}M^{-n/2}(V(0))^{\nicefrac{1}{4}} e^{\rho T/2}\luckyCstT^{2n}
e^{4nLT}
\left[
1
 +
 2^{-1/2}4\left(\luckyCstT e^{2LT}\right)^{-1}
 +
 2^{-1/2}4n
\left(\luckyCstT e^{2LT}\right)^{-1} 
\right]
\\
&\leq 
n\cstMD e^{M/2}M^{-n/2}(V(0))^{\nicefrac{1}{4}} e^{\rho T/2}\luckyCstT^{2n}
e^{4nLT}
\left[
1
 +
 \tfrac{4}{\luckyCstT\sqrt{2}}
 +
 \tfrac{4}{\luckyCstT\sqrt{2}}
\right]
\\
&\leq  2 n
e^{M/2}M^{-n/2}
\luckyCstT^{2n} e^{4nLT}
e^{\rho T/2}\cstMD(V(0))^{\nicefrac{1}{4}}.
\label{t30}\end{split}
\end{equation}%
This 
 establishes item  \eqref{t28x}. 
Next note that \cite[Lemma~3.6]{HJKNW2018} (applied with $d \defeq \alpha$, $(\mathrm{RV}_{n,M})_{n,M\in\Z}\defeq (\FEU{n,M})_{n,M\in\Z}$ in the notation of \cite[Lemma~3.6]{HJKNW2018}) and
 \eqref{d01}
 show that for all $n,M\in\N$ it holds that 
$
\FEU{n,M}\leq \alpha(5M)^n
$
and
{\allowdisplaybreaks
\begin{align}
&\alpha^{-1}\FEY{n,M}\leq M^n+1+\sum_{\ell=0}^{n-1} \left[(M^{\ell+1}+1)\alpha^{-1}\FEU{n-\ell,M}\right]\leq 
M^n+1+\sum_{\ell=0}^{n-1} \left[(M^{\ell+1}+1)(5M)^{n-\ell}\right]\nonumber \\
&
\leq M^n+1+n(5M)^{n+1}+\left[\sum_{\ell=0}^{n-1} (5M)^{n-\ell}\right]= M^n+1+n(5M)^{n+1}+\frac{(5M)^{n+1}-5M}{5M-1}\nonumber \\
&\leq (n+2)(5M)^{n+1}
.\label{e04}
\end{align}}%
Hence, we obtain that for all $n\in\N$ it holds that
$
\FEY{n+1,n+1}\leq \alpha (n+3)(5n+5)^{n+2}\leq \alpha (10n)^{n+3}
$.
This and \eqref{t30} demonstrate that for all 
$t\in[0,T]$,
$\delta\in (0,\infty)$, $n\in\N$ it holds that
\begin{equation}
\begin{split}
&\FEY{n+1,n+1}
\lp{
\Yappr^{n,n}_t-Y_t}_{\lpspace{2}}^{2+\delta} \leq \left[
\tfrac{
\alpha (10n)^{n+3}\left[8ne^{n/2}\luckyCstT^{2n} e^{4nLT}\right]^{2+\delta}}{n^{(2+\delta)n/2}} \right]\left[e^{\rho T/2}\cstMD(V(0))^{\nicefrac{1}{4}}\right]^{2+\delta} \\
&\leq\left[\tfrac{ 10^{n+3}n^3 \left[8ne^{n/2}\luckyCstT^{2n} e^{4nLT}\right]^{2+\delta} }{n^{\delta n/2}}\right]
\alpha
\left[e^{\rho T/2}\cstMD(V(0))^{\nicefrac{1}{4}}\right]^{2+\delta}<\infty.
\label{e01}
\end{split}
\end{equation}
Next observe that
 \eqref{t30} and the fact that
$\limsup_{n\to\infty}\left[n
e^{n/2}n^{-n/2}
\luckyCstT^{2n} e^{4nLT}\right]=0$
 prove that
 \begin{equation}\label{eq:tk1008}
 \limsup_{n\to\infty} \sup_{t\in[0,T]}\lp{
\Yappr^{n,n}_t-Y_t}_{\lpspace{2}}=0.
 \end{equation}
In the next step let $\sfN\colon (0,\infty) \to [0,\infty]$ satisfy for all $\epsilon\in(0,\infty)$ that
\begin{align}\label{e02}
\sfN(\varepsilon)= \inf\! \left(\left\{n\in\N\colon\textstyle\sup_{t\in[0,T]}
\E\bigl[
\left|
\Yappr^{n,n}_t-Y_t\right|^2\bigr]<\epsilon^2\right\}\cup\{\infty\}\right).
\end{align}
Note that \eqref{eq:tk1008} and \eqref{e02} imply that for all $\epsilon\in(0,\infty)$ it holds that 
$\sfN(\epsilon)\in \N$ and
$
\sup_{t\in[0,T]}\lp{
\Yappr^{\sfN(\epsilon),\sfN(\epsilon)}_t-Y_t}_{\lpspace{2}}<
\epsilon\leq 
\1_{\{1\}}(\sfN(\epsilon))\epsilon+
\1_{(1,\infty)}(\sfN(\epsilon))
\sup_{t\in[0,T]}\lp{
\Yappr^{\sfN(\epsilon)-1,\sfN(\epsilon)-1}_t-Y_t}_{\lpspace{2}}.
$
Combining \eqref{e04} and \eqref{e01}  hence ensures that for all $\delta,\epsilon\in(0,1]$ it holds that
\begin{align}\begin{split}\label{e03}
&\FEY{\sfN(\epsilon),\sfN(\epsilon)}\epsilon^{2+\delta}\leq \1_{\{1\}}(\sfN(\epsilon))
\FEY{\sfN(\epsilon),\sfN(\epsilon)}\epsilon^{2+\delta}+\1_{(1,\infty)}(\sfN(\epsilon))
\left[
\FEY{\sfN(\epsilon),\sfN(\epsilon)}\sup_{t\in[0,T]}
\lp{
\Yappr^{\sfN(\epsilon)-1,\sfN(\epsilon)-1}_t-Y_t}_{\lpspace{2}}^{2+\delta}\right]\\&\leq\left(\sup_{n\in\N}\left[\tfrac{10^{n+3}n^3 \left[8ne^{n/2}\luckyCstT^{2n} e^{4nLT}\right]^{2+\delta} }{n^{\delta n/2}}\right]\right)
\alpha
\left[e^{\rho T/2}\cstMD(V(0))^{\nicefrac{1}{4}}\right]^{2+\delta}<\infty.
\end{split}\end{align}
This, \eqref{e02}, 
the fact that for all $\epsilon\in(0,\infty)$ it holds that $\sfN(\epsilon)<\infty$, and  the fact that
$
\cstMD= (\max\{\E [\|\stdNormal^0\|^4],1\})^{\nicefrac{1}{4}}
$
 establish item \eqref{t28b}. The proof of \cref{t26} is thus complete.
\end{proof}

\begin{lemma}\label{f01}Let 
$f\in C^2(\R,\R)$.
Then
it holds for all $v_1,v_2,w_1,w_2\in\R $
 that
 \begin{multline}\label{eq:f01}
 |(f(v_1)-f(w_1)) - (f(v_2) -f(w_2))|\leq \left(\textstyle\sup_{x\in\R}|f'(x)|\right)|(v_1-w_1) -(v_2-w_2)|\\
 +\tfrac{1}{2}\left(\textstyle\sup_{x\in\R}|f''(x)|\right)\bigl[|v_1-w_1|+|v_2-w_2|\bigr]\min\{ | v_1 - v_2 |, | w_1 - w_2 | \}.
 \end{multline}
\end{lemma}
\begin{proof}[Proof of \cref{f01}]  
Observe that
the fundamental theorem of calculus and the triangle inequality
 show that for all
$v_1,v_2,w_1,w_2\in\R $ it holds that
\begin{align}\label{c06c}\begin{split}
&\left|(f(v_1)-f(w_1)) - (f(v_2) -f(w_2))\right|=
\left|(f(v_1)-f(v_2)) - (f(w_1) -f(w_2))\right|
   \\
&
= \left| \int_{0}^{1} f'(\lambda v_1+(1-\lambda)v_2 )(v_1-v_2) -
 f'(\lambda w_1+(1-\lambda)w_2 )(w_1-w_2)\,d\lambda\right|   \\
&= \biggl|\int_{0}^{1} f'(\lambda v_1+(1-\lambda)v_2 )\bigl[(v_1-v_2)-(w_1-w_2) \bigr]d\lambda  \\
&\quad +\int_{0}^{1}\bigl[
f'(\lambda v_1+(1-\lambda)v_2 )-
f'(\lambda w_1+(1-\lambda)w_2 )\bigr](w_1-w_2)\,d\lambda \biggr|   \\
&\leq \left(\textstyle\sup_{x\in\R}|f'(x)|\right)|(v_1-w_1) -(v_2-w_2)|  \\
&\quad +\left(\textstyle\sup_{x\in\R}|f''(x)|\right)\left[\int_{0}^{1}
\bigl(\lambda|v_1-w_1|+(1-\lambda)|v_2-w_2|\bigr)\,d\lambda\right]|w_1-w_2|.\end{split} 
\end{align}%
This and the fact that $\int_{0}^{1}\lambda\,d\lambda=
\int_{0}^{1}(1-\lambda)\,d\lambda=\nicefrac{1}{2}
$ establish \eqref{eq:f01}.
The proof of \cref{f01} is thus complete.
\end{proof}

\begin{corollary}\label{cor:sec4}
Let  $T, \delta \in (0,\infty)$, 
 $  \Theta = \bigcup_{ n \in \N }\! \Z^n$,
$f\in C^2( \R,\R)$, 
let 
$g_d\in C^1( \R^d,\R)$, $d\in\N$, satisfy $\sup_{d\in\N}\sup_{x=(x_1,x_2,\ldots, x_d)\in\R^d}\bigl(|f(x_1)|+|f'(x_1)|+|f''(x_1)|+|g_d(x)|+
\sum_{i=1}^{d}| \tfrac{\partial g_d}{\partial x_i}(x)|^2\bigr)<\infty$,
let
$(\Omega, \mathcal{F}, \P, (\F_t)_{t\in[0,T]})$ 
be a filtered probability space,
let $\unif^\theta\colon \Omega\to[0,1]$, $\theta\in \Theta$, be i.i.d.\ random variables,
assume for all $t\in (0,1)$ that $\P(\unif^{0}\le t)=t$,
let $\stdNormal^{d,\theta}\colon  \Omega\to\R^d$, $\theta\in\Theta$, 
$d\in\N$,
be i.i.d.\ standard normal vectors,
let $W^d=(W^{d,1},W^{d,2},\ldots,W^{d,d})\colon [0,T]\times\Omega\to\R^d$, $d\in\N$, be  standard 
$(\F_t)_{t\in[0,T]}$-Brownian motions,
assume that 
$
(\unif^\theta)_{\theta\in\Theta}$, $
(\stdNormal^{d,\theta})_{(d,\theta)\in \N\times \Theta}$, and
$(W^d)_{d\in\N}$
 are independent,
let
$ 
  {{U}}_{ n,M}^{d,\theta } \colon [0, T] \times \R^d \times \Omega \to \R
$, 
$d,M,n\in\N_0$, $\theta\in\Theta$, satisfy
for all 
$d,M \in \N$, $n\in \N_0$, $\theta\in\Theta $, 
$ t \in [0,T]$, $x\in\R^d $
 that 
{\small
\begin{equation}
\begin{split}
&  {{U}}_{n,M}^{d,\theta}(t,x)
=(T-t)f(0)\1_{\N}(n)+
  \frac{\1_{\N}(n)}{M^n}
 \sum_{i=1}^{M^n} 
      {g}_d\bigl(x+[T-t]^{1/2}{\stdNormal}^{d,(\theta,0,-i)}\bigr)
 \\
& +
  \sum_{\ell=1}^{n-1}\Biggl[ \frac{(T-t)}{M^{n-\ell}}
    \sum_{i=1}^{M^{n-\ell}}
      \bigl(f\circ {{U}}_{\ell,M}^{d,(\theta,\ell,i)}-f\circ {{U}}_{\ell-1,M}^{d,(\theta,-\ell,i)}\bigr)
      \bigl(t+(T-t)\unif^{(\theta,\ell,i)},x+[(T-t)\unif^{(\theta,\ell,i)}]^{1/2}{\stdNormal}^{d,(\theta,\ell,i)}\bigr)
    \Biggr],
    \end{split}
\end{equation}}%
let $\lfloor \cdot \rfloor_M \colon \R \to \R$, $ M \in \N $, and 
$\lceil \cdot \rceil_M \colon  \R \to \R$, $ M \in \N $, 
satisfy for all $M \in \N$, $t \in [0,T]$ that
$\lfloor t \rfloor_M = \max( ([0,t]\backslash \{T\}) \cap \{ 0, \frac{ T }{ M }, \frac{ 2T }{ M }, \ldots \} )$
and 
$\lceil t \rceil_M = \min(((t,\infty) \cup \{T\})\cap  \{ 0, \frac{ T }{ M }, \frac{ 2T }{ M }, \ldots \} )$,
let $\Yappr^{d,n,M}
\colon [0,T]\times\Omega\to\R $,
 $d,n,M\in\N$, 
satisfy  for all $d,n,M\in\N$, $t\in[0,T]$  that
\begin{align}
\Yappr^{d,n,M}_{t}
&= \sum_{\ell=0}^{n-1}\biggl[
\left[ \tfrac{ \lceil t \rceil_{M^{l+1}} - t }{ ( T / M^{ l + 1 } ) } \right]
U^{d,\ell}_{n-\ell,M}(\lfloor t \rfloor_{M^{l+1}}, W_{\lfloor t \rfloor_{M^{l+1}}}^{d})+
\left[\tfrac{ t-\lfloor t \rfloor_{M^{l+1}} }{ ( T / M^{ l + 1 } ) }\right]
U^{d,\ell}_{n-\ell,M}(\lceil t \rceil_{M^{l+1}}, W_{\lceil t \rceil_{M^{l+1}}}^{d})
\nonumber
\\
&\quad-\1_{\N}(\ell)\Bigl(
\left[ \tfrac{ \lceil t \rceil_{M^{l}} - t }{ ( T / M^{ l  } ) } \right]U^{d,\ell}_{n-\ell,M}(\lfloor t \rfloor_{M^{l}}, W^{d}_{\lfloor t \rfloor_{M^{l}}})+
\left[ \tfrac{ t-\lceil t \rceil_{M^{l}} }{ ( T / M^{ l  } ) } \right]
U^{d,\ell}_{n-\ell,M}(\lceil t \rceil_{M^{l}}, W^{d}_{\lceil t \rceil_{M^{l}}})
\Bigr)\biggr],
\end{align}
let $\mathbf{Y}^d=(Y^d,Z^{d,1},Z^{2,d},\ldots,Z^{d,d})\colon [0,T]\times\Omega\to \R^{d+1}
$, $d\in \N$, 
be $(\F_{t})_{t\in[0,T]}$-predictable stochastic processes,
assume for all $d\in \N$ that
$
\int_0^T  \E \bigl[|Y_s^d|+\textstyle\sum_{j=1}^{d}|Z_s^{d,j}|^2\bigr]ds<\infty
$,
assume that for all $d\in \N$, $t\in[0,T]$ it holds $\P$-a.s. that
\begin{equation}
Y^d_t=g_d(W^d_T)+\int_t^T f(Y^d_s)\,ds-\sum_{j=1}^{d}\int_t^T Z_s^{d,j}\, dW_s^{d,j},
\end{equation}
and let
$\FEU{d,n,M} \in \N_0$, $d, n,M\in\Z$, and $\FEY{n,M}\in \N_0$, $d,n,M\in\Z$,
satisfy 
for all $d\in \N$, $n,M\in \N_0$ that 
\begin{equation}
\label{d01b}
 \FEU{d,n,M}\leq  (d+1) M^n \1_{\N}(n)+ \sum_{\ell=0}^{n-1}\left[M^{n-\ell}\left(2+d+\FEU{d,\ell,M}+\FEU{d,\ell-1,M}\1_{\N}(\ell)\right)\right],
 \end{equation}
 and
$
\FEY{d,n,M}\leq (d+1)(M^n+1)+\sum_{\ell=0}^{n-1} \left[(M^{\ell+1}+1)\FEU{d,n-\ell,M}\right]$.
Then
there exist 
$c\in \R$ and
$\sfN\colon \N \times (0,1]\to \N$
such that
 for all $d\in\N$, $\varepsilon\in  (0,1]$ it holds that
$\sup_{t\in [0,T]}
(\E[
|
\Yappr^{d,\sfN(d,\epsilon),\sfN(d,\epsilon)}_t-Y_t^d|^2])^{1/2}
\leq\epsilon
$ and
$
\FEY{d,\sfN(d,\epsilon),\sfN(d,\epsilon)}
\leq c
d^{c}\epsilon^{-(2+\delta)}.
$
\end{corollary}
\begin{proof}[Proof of \cref{cor:sec4}]
Note that \cref{f01} ensures that for 
 $v_1,v_2,w_1,w_2\in\R $ it holds 
 that
 \begin{multline}
 |(f(v_1)-f(w_1)) - (f(v_2) -f(w_2))|\leq \left(\textstyle\sup_{x\in\R}|f'(x)|\right)|(v_1-w_1) -(v_2-w_2)|\\
 +\tfrac{1}{2}\left(\textstyle\sup_{x\in\R}|f''(x)|\right)\bigl[|v_1-w_1|+|v_2-w_2|\bigr]\min\{ | v_1 - v_2 |, | w_1 - w_2 | \}.
 \end{multline}
 This and \cref{t26} prove that there exist 
$c\in \R$ and
$\sfN\colon \N \times (0,1]\to \N$
such that
 for all $d\in\N$, $\varepsilon\in  (0,1]$ it holds that
$\sup_{t\in [0,T]}
(\E[
|
\Yappr^{d,\sfN(d,\epsilon),\sfN(d,\epsilon)}_t-Y_t^d|^2])^{1/2}
\leq\epsilon
$ and
$
\FEY{d,\sfN(d,\epsilon),\sfN(d,\epsilon)}
\leq c
d^{c}\epsilon^{-(2+\delta)}.
$
The proof of \cref{cor:sec4} is thus complete.
\end{proof}

\subsubsection*{Acknowledgements}
We thank Benno Kuckuck for several helpful discussions.
This work has been funded by the Deutsche Forschungsgemeinschaft (DFG, German Research Foundation) under Germany’s Excellence Strategy EXC 2044-390685587, Mathematics M\"unster:  Dynamics-Geometry-Structure and through the research grant HU1889/6-2.

{
\bibliographystyle{acm}
\bibliography{bibfile}

\begin{thebibliography}{100}

\bibitem{abbas2020conditional}
{\sc Abbas-Turki, L., Diallo, B., and Pag{\`e}s, G.}
\newblock Conditional {M}onte {C}arlo learning for diffusions {I}: main
  methodology and application to backward stochastic differential equations.

\bibitem{abbas2020conditional2}
{\sc Abbas-Turki, L., Diallo, B., and Pag{\`e}s, G.}
\newblock Conditional {M}onte {C}arlo learning for diffusions {II}: extended
  methodology and application to risk measures and early stopping problems.

\bibitem{agarwal2020branching}
{\sc Agarwal, A., and Claisse, J.}
\newblock Branching diffusion representation of semi-linear elliptic {PDE}s and
  estimation using {M}onte {C}arlo method.
\newblock {\em Stochastic Processes and their Applications\/} (2020).

\bibitem{bally2003error}
{\sc Bally, V., and Pages, G.}
\newblock Error analysis of the optimal quantization algorithm for obstacle
  problems.
\newblock {\em Stochastic processes and their applications 106}, 1 (2003),
  1--40.

\bibitem{BallyPages2003}
{\sc Bally, V., and Pag\`es, G.}
\newblock A quantization algorithm for solving multi-dimensional discrete-time
  optimal stopping problems.
\newblock {\em Bernoulli 9}, 6 (2003), 1003--1049.

\bibitem{beck2019existenceb}
{\sc Beck, C., Gonon, L., Hutzenthaler, M., and Jentzen, A.}
\newblock On existence and uniqueness properties for solutions of stochastic
  fixed point equations.
\newblock {\em Discrete Contin. Dyn. Syst. Ser. B 26}, 9 (2021), 4927--4962.

\bibitem{beck2020overcoming}
{\sc Beck, C., Gonon, L., and Jentzen, A.}
\newblock Overcoming the curse of dimensionality in the numerical approximation
  of high-dimensional semilinear elliptic partial differential equations.
\newblock {\em arXiv:2003.00596\/} (2020).

\bibitem{beck2019overcoming}
{\sc Beck, C., Hornung, F., Hutzenthaler, M., Jentzen, A., and Kruse, T.}
\newblock Overcoming the curse of dimensionality in the numerical approximation
  of {A}llen-{C}ahn partial differential equations via truncated full-history
  recursive multilevel {P}icard approximations.
\newblock {\em Journal of Numerical Mathematics 28}, 4 (2020), 197--222.

\bibitem{beck2021nonlinear}
{\sc Beck, C., Hutzenthaler, M., and Jentzen, A.}
\newblock On nonlinear {F}eynman--{K}ac formulas for viscosity solutions of
  semilinear parabolic partial differential equations.
\newblock {\em Stochastics and Dynamics\/} (2021).

\bibitem{beck2020overview}
{\sc Beck, C., Hutzenthaler, M., Jentzen, A., and Kuckuck, B.}
\newblock An overview on deep learning-based approximation methods for partial
  differential equations.
\newblock {\em arXiv:2012.12348\/} (2020).

\bibitem{beck2020nonlinear}
{\sc Beck, C., Jentzen, A., and Kruse, T.}
\newblock Nonlinear {M}onte {C}arlo methods with polynomial runtime for
  high-dimensional iterated nested expectations.
\newblock {\em arXiv:2009.13989\/} (2020).

\bibitem{becker2020numerical}
{\sc Becker, S., Braunwarth, R., Hutzenthaler, M., Jentzen, A., and von
  Wurstemberger, P.}
\newblock Numerical simulations for full history recursive multilevel {P}icard
  approximations for systems of high-dimensional partial differential
  equations.
\newblock {\em Communications in Computational Physics 28}, 5 (2020),
  2109--2138.

\bibitem{Bellman}
{\sc Bellman, R.}
\newblock {\em Dynamic programming}.
\newblock Princeton Landmarks in Mathematics. Princeton University Press,
  Princeton, NJ, 2010.
\newblock Reprint of the 1957 edition, With a new introduction by Stuart
  Dreyfus.

\bibitem{BenderDenk2007}
{\sc Bender, C., and Denk, R.}
\newblock A forward scheme for backward {SDE}s.
\newblock {\em Stochastic Process. Appl. 117}, 12 (2007), 1793--1812.

\bibitem{bender2017iterative}
{\sc Bender, C., G{\"a}rtner, C., and Schweizer, N.}
\newblock Iterative improvement of lower and upper bounds for backward {SDE}s.
\newblock {\em SIAM Journal on Scientific Computing 39}, 2 (2017), B442--B466.

\bibitem{bender2018pathwise}
{\sc Bender, C., G{\"a}rtner, C., and Schweizer, N.}
\newblock Pathwise dynamic programming.
\newblock {\em Mathematics of Operations Research 43}, 3 (2018), 965--995.

\bibitem{Bender2015Primal}
{\sc Bender, C., Schweizer, N., and Zhuo, J.}
\newblock A primal-–dual algorithm for {BSDEs}.
\newblock {\em Math. Finance 27}, 3 (2017), 866--901.

\bibitem{bender2012least}
{\sc Bender, C., and Steiner, J.}
\newblock Least-squares monte carlo for backward sdes.
\newblock In {\em Numerical methods in finance}. Springer, 2012, pp.~257--289.

\bibitem{bender2008time}
{\sc Bender, C., and Zhang, J.}
\newblock Time discretization and {M}arkovian iteration for coupled {FBSDE}s.
\newblock {\em The Annals of Applied Probability 18}, 1 (2008), 143--177.

\bibitem{BouchardTouzi2004}
{\sc Bouchard, B., and Touzi, N.}
\newblock Discrete-time approximation and {M}onte-{C}arlo simulation of
  backward stochastic differential equations.
\newblock {\em Stochastic Process. Appl. 111}, 2 (2004), 175--206.

\bibitem{briand2001donsker}
{\sc Briand, P., Delyon, B., and M{\'e}min, J.}
\newblock Donsker-type theorem for bsdes.
\newblock {\em Electronic Communications in Probability 6\/} (2001), 1--14.

\bibitem{briand2019donsker}
{\sc Briand, P., Geiss, C., Geiss, S., and Labart, C.}
\newblock Donsker-type theorem for {BSDE}s: rate of convergence.
\newblock {\em arXiv:1908.01188\/} (2019).

\bibitem{BriandLabart2014}
{\sc Briand, P., and Labart, C.}
\newblock Simulation of {BSDEs by Wiener} chaos expansion.
\newblock {\em Ann. Appl. Probab. 24}, 3 (2014), 1129--1171.

\bibitem{ChangLiuXiong2016}
{\sc Chang, D., Liu, H., and Xiong, J.}
\newblock A branching particle system approximation for a class of {FBSDE}s.
\newblock {\em Probab. Uncertain. Quant. Risk 1\/} (2016), Paper No. 9, 34.

\bibitem{Chassagneux2014}
{\sc Chassagneux, J.-F.}
\newblock Linear multistep schemes for {BSDE}s.
\newblock {\em SIAM J. Numer. Anal. 52}, 6 (2014), 2815--2836.

\bibitem{ChassagneuxCrisan2014}
{\sc Chassagneux, J.-F., and Crisan, D.}
\newblock Runge-{K}utta schemes for backward stochastic differential equations.
\newblock {\em Ann. Appl. Probab. 24}, 2 (2014), 679--720.

\bibitem{chassagneux2020cubature}
{\sc Chassagneux, J.-F., and Garcia~Trillos, C.}
\newblock Cubature method to solve {BSDE}s: {E}rror expansion and complexity
  control.
\newblock {\em Mathematics of Computation 89}, 324 (2020), 1895--1932.

\bibitem{ChassagneuxRichou2015}
{\sc Chassagneux, J.-F., and Richou, A.}
\newblock Numerical stability analysis of the {E}uler scheme for {BSDE}s.
\newblock {\em SIAM J. Numer. Anal. 53}, 2 (2015), 1172--1193.

\bibitem{ChassagneuxRichou2016}
{\sc Chassagneux, J.-F., and Richou, A.}
\newblock Numerical simulation of quadratic {BSDE}s.
\newblock {\em Ann. Appl. Probab. 26}, 1 (2016), 262--304.

\bibitem{chen2019deep}
{\sc Chen, Y., and Wan, J.~W.}
\newblock Deep neural network framework based on backward stochastic
  differential equations for pricing and hedging {A}merican options in high
  dimensions.
\newblock {\em arXiv:1909.11532\/} (2019).

\bibitem{chevance1997numerical}
{\sc Chevance, D.}
\newblock Numerical methods for backward stochastic differential equations.
\newblock {\em Numerical methods in finance 232\/} (1997).

\bibitem{cox2016convergence}
{\sc Cox, S., Hutzenthaler, M., Jentzen, A., van Neerven, J., and Welti, T.}
\newblock Convergence in {H}{\"o}lder norms with applications to {M}onte
  {C}arlo methods in infinite dimensions.
\newblock {\em IMA Journal of Numerical Analysis 41}, 1 (2021), 493--548.

\bibitem{crepey2013financial}
{\sc Cr{\'e}pey, S.}
\newblock Financial modeling.
\newblock {\em Springer Finance, DOI 10\/} (2013), 978--3.

\bibitem{CrisanManolarakis2010}
{\sc Crisan, D., and Manolarakis, K.}
\newblock Probabilistic methods for semilinear partial differential equations.
  {A}pplications to finance.
\newblock {\em M2AN Math. Model. Numer. Anal. 44}, 5 (2010), 1107--1133.

\bibitem{CrisanManolarakis2012}
{\sc Crisan, D., and Manolarakis, K.}
\newblock Solving backward stochastic differential equations using the cubature
  method: application to nonlinear pricing.
\newblock {\em SIAM J. Financial Math. 3}, 1 (2012), 534--571.

\bibitem{CrisanManolarakis2014}
{\sc Crisan, D., and Manolarakis, K.}
\newblock Second order discretization of backward {SDE}s and simulation with
  the cubature method.
\newblock {\em Ann. Appl. Probab. 24}, 2 (2014), 652--678.

\bibitem{CrisanManolarakisTouzi2010}
{\sc Crisan, D., Manolarakis, K., and Touzi, N.}
\newblock On the {M}onte {C}arlo simulation of {BSDE}s: an improvement on the
  {M}alliavin weights.
\newblock {\em Stochastic Process. Appl. 120}, 7 (2010), 1133--1158.

\bibitem{cvitanic2005steepest}
{\sc Cvitanic, J., and Zhang, J.}
\newblock The steepest descent method for forward-backward {SDE}s.
\newblock {\em Electronic Journal of Probability 10\/} (2005), 1468--1495.

\bibitem{de2015cubature}
{\sc de~Raynal, P.~C., and Trillos, C.~G.}
\newblock A cubature based algorithm to solve decoupled {M}c{K}ean--{V}lasov
  forward--backward stochastic differential equations.
\newblock {\em Stochastic Processes and their Applications 125}, 6 (2015),
  2206--2255.

\bibitem{DelarueMenozzi2006}
{\sc Delarue, F., and Menozzi, S.}
\newblock A forward-backward stochastic algorithm for quasi-linear {PDE}s.
\newblock {\em Ann. Appl. Probab. 16}, 1 (2006), 140--184.

\bibitem{delarue2008interpolated}
{\sc Delarue, F., and Menozzi, S.}
\newblock An interpolated stochastic algorithm for quasi-linear {PDE}s.
\newblock {\em Mathematics of Computation 77}, 261 (2008), 125--158.

\bibitem{delong2013backward}
{\sc Delong, {\L}.}
\newblock {\em Backward stochastic differential equations with jumps and their
  actuarial and financial applications}.
\newblock Springer, 2013.

\bibitem{DouglasMaProtter}
{\sc Douglas, Jr., J., Ma, J., and Protter, P.}
\newblock Numerical methods for forward-backward stochastic differential
  equations.
\newblock {\em Ann. Appl. Probab. 6}, 3 (1996), 940--968.

\bibitem{EHanJentzen2017CMStat}
{\sc E, W., Han, J., and Jentzen, A.}
\newblock Deep learning-based numerical methods for high-dimensional parabolic
  partial differential equations and backward stochastic differential
  equations.
\newblock {\em Commun. Math. Stat. 5}, 4 (2017), 349--380.

\bibitem{han2020algorithms}
{\sc E, W., Han, J., and Jentzen, A.}
\newblock Algorithms for {S}olving {H}igh {D}imensional {PDE}s: {F}rom
  {N}onlinear {M}onte {C}arlo to {M}achine {L}earning.
\newblock {\em arXiv:2008.13333\/} (2020).

\bibitem{EHutzenthalerJentzenKruse2016}
{\sc E, W., Hutzenthaler, M., Jentzen, A., and Kruse, T.}
\newblock Multilevel {P}icard iterations for solving smooth semilinear
  parabolic heat equations.
\newblock {\em arXiv:1607.03295\/} (2016).
\newblock To appear in Springer Nature Partial Differential Equations and
  Applications.

\bibitem{EHutzenthalerJentzenKruse2017}
{\sc E, W., Hutzenthaler, M., Jentzen, A., and Kruse, T.}
\newblock On multilevel {P}icard numerical approximations for high-dimensional
  nonlinear parabolic partial differential equations and high-dimensional
  nonlinear backward stochastic differential equations.
\newblock {\em Journal of Scientific Computing 79}, 3 (2019), 1534--1571.

\bibitem{ElKarouiPengQuenez1997}
{\sc El~Karoui, N., Peng, S., and Quenez, M.~C.}
\newblock Backward stochastic differential equations in finance.
\newblock {\em Math. Finance 7}, 1 (1997), 1--71.

\bibitem{FuZhaoZhou}
{\sc Fu, Y., Zhao, W., and Zhou, T.}
\newblock Efficient spectral sparse grid approximations for solving
  multi-dimensional forward backward {SDE}s.
\newblock {\em Discrete Contin. Dyn. Syst. Ser. B 22}, 9 (2017), 3439--3458.

\bibitem{FujiiTakahashiTakahashi2017}
{\sc Fujii, M., Takahashi, A., and Takahashi, M.}
\newblock Asymptotic {E}xpansion as {P}rior {K}nowledge in {D}eep {L}earning
  {M}ethod for high dimensional {BSDEs}.
\newblock {\em arXiv:1710.07030\/} (2017).

\bibitem{GeissLabart2016}
{\sc Geiss, C., and Labart, C.}
\newblock Simulation of {BSDE}s with jumps by {W}iener chaos expansion.
\newblock {\em Stochastic Process. Appl. 126}, 7 (2016), 2123--2162.

\bibitem{geiss_labart_luoto_2020}
{\sc Geiss, C., Labart, C., and Luoto, A.}
\newblock Mean square rate of convergence for random walk approximation of
  forward-backward {SDE}s.
\newblock {\em Advances in Applied Probability 52}, 3 (2020), 735–771.

\bibitem{geiss2020random}
{\sc Geiss, C., Labart, C., Luoto, A., et~al.}
\newblock Random walk approximation of {BSDE}s with h{\"o}lder continuous
  terminal condition.
\newblock {\em Bernoulli 26}, 1 (2020), 159--190.

\bibitem{giles20019generalised}
{\sc Giles, M.~B., Jentzen, A., and Welti, T.}
\newblock Generalised multilevel {P}icard approximations.
\newblock {\em arXiv:1911.03188\/} (2019).

\bibitem{gobet2007error}
{\sc Gobet, E., and Labart, C.}
\newblock Error expansion for the discretization of backward stochastic
  differential equations.
\newblock {\em Stochastic processes and their applications 117}, 7 (2007),
  803--829.

\bibitem{GobetLabart2010}
{\sc Gobet, E., and Labart, C.}
\newblock Solving {BSDE} with adaptive control variate.
\newblock {\em SIAM J. Numer. Anal. 48}, 1 (2010), 257--277.

\bibitem{GobetLemor2008Numerical}
{\sc Gobet, E., and Lemor, J.-P.}
\newblock Numerical simulation of {BSDE}s using empirical regression methods:
  theory and practice.
\newblock {\em arXiv:0806.4447\/} (2008).

\bibitem{GobetLemorWarin2005}
{\sc Gobet, E., Lemor, J.-P., and Warin, X.}
\newblock A regression-based {M}onte {C}arlo method to solve backward
  stochastic differential equations.
\newblock {\em Ann. Appl. Probab. 15}, 3 (2005), 2172--2202.

\bibitem{GobetLopezSalasTurkedjiev2016}
{\sc Gobet, E., L\'opez-Salas, J.~G., Turkedjiev, P., and V\'azquez, C.}
\newblock Stratified regression {M}onte-{C}arlo scheme for semilinear {PDE}s
  and {BSDE}s with large scale parallelization on {GPU}s.
\newblock {\em SIAM J. Sci. Comput. 38}, 6 (2016), C652--C677.

\bibitem{GobetTurkedjiev2016}
{\sc Gobet, E., and Turkedjiev, P.}
\newblock Approximation of backward stochastic differential equations using
  {M}alliavin weights and least-squares regression.
\newblock {\em Bernoulli 22}, 1 (2016), 530--562.

\bibitem{GobetTurkedjiev2016MathComp}
{\sc Gobet, E., and Turkedjiev, P.}
\newblock Linear regression {MDP} scheme for discrete backward stochastic
  differential equations under general conditions.
\newblock {\em Math. Comp. 85}, 299 (2016), 1359--1391.

\bibitem{h98}
{\sc Heinrich, S.}
\newblock Monte {C}arlo complexity of global solution of integral equations.
\newblock {\em J. Complexity 14}, 2 (1998), 151--175.

\bibitem{Heinrich01}
{\sc Heinrich, S.}
\newblock Multilevel {M}onte {C}arlo {M}ethods.
\newblock In {\em Large-Scale Scientific Computing}, vol.~2179 of {\em Lecture
  Notes in Computer Science}. Springer, 2001, pp.~58--67.

\bibitem{heinrich1999monte}
{\sc Heinrich, S., and Sindambiwe, E.}
\newblock {M}onte {C}arlo complexity of parametric integration.
\newblock {\em Journal of Complexity 15}, 3 (1999), 317--341.

\bibitem{Labordere2012}
{\sc Henry-Labord{\`e}re, P.}
\newblock Counterparty risk valuation: a marked branching diffusion approach.
\newblock {\em arXiv:1203.2369\/} (2012).

\bibitem{HenryLabordere2017}
{\sc Henry-Labord{\`e}re, P.}
\newblock Deep {Primal-Dual Algorithm for BSDEs}: {A}pplications of {M}achine
  {L}earning to {CVA and IM}.
\newblock {\em Available at SSRN: http://dx.doi.org/10.2139/ssrn.3071506\/}
  (2017).

\bibitem{HenryLabordereOudjaneTanTouziWarin2016}
{\sc Henry-Labord{\`e}re, P., Oudjane, N., Tan, X., Touzi, N., and Warin, X.}
\newblock Branching diffusion representation of semilinear {PDE}s and {M}onte
  {C}arlo approximation.
\newblock {\em Annales de l'Institut Henri Poincar{\'e}, Probabilit{\'e}s et
  Statistiques 55}, 1 (2019), 184--210.

\bibitem{LabordereTanTouzi2014}
{\sc Henry-Labord{\`e}re, P., Tan, X., and Touzi, N.}
\newblock A numerical algorithm for a class of {BSDE}s via the branching
  process.
\newblock {\em Stochastic Process. Appl. 124}, 2 (2014), 1112--1140.

\bibitem{hu2011malliavin}
{\sc Hu, Y., Nualart, D., and Song, X.}
\newblock Malliavin calculus for backward stochastic differential equations and
  application to numerical solutions.
\newblock {\em The Annals of Applied Probability 21}, 6 (2011), 2379--2423.

\bibitem{HuijskensRuijterOosterlee2016}
{\sc Huijskens, T.~P., Ruijter, M.~J., and Oosterlee, C.~W.}
\newblock Efficient numerical {F}ourier methods for coupled forward-backward
  {SDE}s.
\newblock {\em J. Comput. Appl. Math. 296\/} (2016), 593--612.

\bibitem{hjk2019overcoming}
{\sc Hutzenthaler, M., Jentzen, A., and Kruse, T.}
\newblock Overcoming the curse of dimensionality in the numerical approximation
  of parabolic partial differential equations with gradient-dependent
  nonlinearities.
\newblock {\em Found. Comput. Math.\/} (2021), 1--62.

\bibitem{HJKN20}
{\sc Hutzenthaler, M., Jentzen, A., Kruse, T., and Nguyen, T.~A.}
\newblock Multilevel {P}icard approximations for high-dimensional semilinear
  second-order {PDE}s with {L}ipschitz nonlinearities.
\newblock {\em arXiv:2009.02484\/} (2020).

\bibitem{HJKNW2018}
{\sc Hutzenthaler, M., Jentzen, A., Kruse, T., Nguyen, T.~A., and von
  Wurstemberger, P.}
\newblock Overcoming the curse of dimensionality in the numerical approximation
  of semilinear parabolic partial differential equations.
\newblock {\em Proceedings of the Royal Society A 476}, 2244 (2020), 20190630.

\bibitem{hutzenthaler2019overcoming}
{\sc Hutzenthaler, M., Jentzen, A., and von Wurstemberger, P.}
\newblock Overcoming the curse of dimensionality in the approximative pricing
  of financial derivatives with default risks.
\newblock {\em Electron. J. Probab. 25\/} (2020), Paper No. 101, 73.

\bibitem{HutzenthalerKruse2017}
{\sc Hutzenthaler, M., and Kruse, T.}
\newblock Multilevel {P}icard approximations of high-dimensional semilinear
  parabolic differential equations with gradient-dependent nonlinearities.
\newblock {\em SIAM Journal on Numerical Analysis 58}, 2 (2020), 929--961.

\bibitem{imkeller2010results}
{\sc Imkeller, P., Dos~Reis, G., and Zhang, J.}
\newblock Results on numerics for {FBSDE} with drivers of quadratic growth.
\newblock In {\em Contemporary Quantitative Finance}. Springer, 2010,
  pp.~159--182.

\bibitem{LabartLelong2013}
{\sc Labart, C., and Lelong, J.}
\newblock A parallel algorithm for solving {BSDE}s.
\newblock {\em Monte Carlo Methods Appl. 19}, 1 (2013), 11--39.

\bibitem{le2017particle}
{\sc Le~Cavil, A., Oudjane, N., and Russo, F.}
\newblock Particle system algorithm and chaos propagation related to
  non-conservative {M}c{K}ean type stochastic differential equations.
\newblock {\em Stochastics and Partial Differential Equations: Analysis and
  Computations 5}, 1 (2017), 1--37.

\bibitem{le2018monte}
{\sc Le~Cavil, A., Oudjane, N., and Russo, F.}
\newblock {M}onte-{C}arlo algorithms for a forward {F}eynman--{K}ac-type
  representation for semilinear nonconservative partial differential equations.
\newblock {\em Monte Carlo Methods and Applications 24}, 1 (2018), 55--70.

\bibitem{le2019forward}
{\sc Le~Cavil, A., Oudjane, N., and Russo, F.}
\newblock Forward {F}eynman-{K}ac type representation for semilinear
  non-conservative partial differential equations.
\newblock {\em Stochastics 91}, 8 (2019), 1206--1248.

\bibitem{LemorGobetWarin2006}
{\sc Lemor, J.-P., Gobet, E., and Warin, X.}
\newblock Rate of convergence of an empirical regression method for solving
  generalized backward stochastic differential equations.
\newblock {\em Bernoulli 12}, 5 (2006), 889--916.

\bibitem{LionnetDosReisSzpruch2015}
{\sc Lionnet, A., dos Reis, G., and Szpruch, L.}
\newblock Time discretization of {FBSDE} with polynomial growth drivers and
  reaction-diffusion {PDE}s.
\newblock {\em Ann. Appl. Probab. 25}, 5 (2015), 2563--2625.

\bibitem{MaProtterSanMartin2002}
{\sc Ma, J., Protter, P., San~Mart\'\i{n}, J., and Torres, S.}
\newblock Numerical method for backward stochastic differential equations.
\newblock {\em Ann. Appl. Probab. 12}, 1 (2002), 302--316.

\bibitem{MaProtterYong1994}
{\sc Ma, J., Protter, P., and Yong, J.~M.}
\newblock Solving forward-backward stochastic differential equations
  explicitly---a four step scheme.
\newblock {\em Probab. Theory Related Fields 98}, 3 (1994), 339--359.

\bibitem{MaYong1999}
{\sc Ma, J., and Yong, J.}
\newblock {\em Forward-backward stochastic differential equations and their
  applications}, vol.~1702 of {\em Lecture Notes in Mathematics}.
\newblock Springer-Verlag, Berlin, 1999.

\bibitem{McKean1975}
{\sc McKean, H.~P.}
\newblock Application of {B}rownian motion to the equation of
  {K}olmogorov-{P}etrovskii-{P}iskunov.
\newblock {\em Comm. Pure Appl. Math. 28}, 3 (1975), 323--331.

\bibitem{MilsteinTretyakov2006}
{\sc Milstein, G.~N., and Tretyakov, M.~V.}
\newblock Numerical algorithms for forward-backward stochastic differential
  equations.
\newblock {\em SIAM J. Sci. Comput. 28}, 2 (2006), 561--582.

\bibitem{MilsteinTretyakov2007}
{\sc Milstein, G.~N., and Tretyakov, M.~V.}
\newblock Discretization of forward-backward stochastic differential equations
  and related quasi-linear parabolic equations.
\newblock {\em IMA J. Numer. Anal. 27}, 1 (2007), 24--44.

\bibitem{MR1485004}
{\sc Novak, E., and Ritter, K.}
\newblock The curse of dimension and a universal method for numerical
  integration.
\newblock In {\em Multivariate approximation and splines ({M}annheim, 1996)},
  vol.~125 of {\em Internat. Ser. Numer. Math.} Birkh\"{a}user, Basel, 1997,
  pp.~177--187.

\bibitem{NovakWozniakowski2008I}
{\sc Novak, E., and Wo{\'z}niakowski, H.}
\newblock {\em Tractability of multivariate problems. {V}ol. 1: {L}inear
  information}, vol.~6 of {\em EMS Tracts in Mathematics}.
\newblock European Mathematical Society (EMS), Z\"urich, 2008.

\bibitem{Oek03}
{\sc {\O}ksendal, B.}
\newblock {\em Stochastic differential equations}.
\newblock Universitext. Springer-Verlag, Berlin, 1985.
\newblock An introduction with applications.

\bibitem{pardoux1999bsdes}
{\sc Pardoux, {\'E}.}
\newblock {BSDE}s, weak convergence and homogenization of semilinear {PDE}s.
\newblock In {\em Nonlinear analysis, differential equations and control}.
  Springer, 1999, pp.~503--549.

\bibitem{PardouxPeng1990}
{\sc Pardoux, {\'E}., and Peng, S.}
\newblock Adapted solution of a backward stochastic differential equation.
\newblock {\em Systems Control Lett. 14}, 1 (1990), 55--61.

\bibitem{PardouxPeng1992}
{\sc Pardoux, E., and Peng, S.}
\newblock Backward stochastic differential equations and quasilinear parabolic
  partial differential equations.
\newblock In {\em Stochastic partial differential equations and their
  applications ({C}harlotte, {NC}, 1991)}, vol.~176 of {\em Lect. Notes Control
  Inf. Sci.} Springer, Berlin, 1992, pp.~200--217.

\bibitem{pardoux2014stochastic}
{\sc Pardoux, E., and Ra{\c{s}}canu, A.}
\newblock {\em Stochastic Differential Equations, Backward SDEs, Partial
  Differential Equations}.
\newblock Stochastic Modelling and Applied Probability. Springer International
  Publishing, 2014.

\bibitem{peng1991probabilistic}
{\sc Peng, S.}
\newblock Probabilistic interpretation for systems of quasilinear parabolic
  partial differential equations.
\newblock {\em Stochastics and stochastics reports 37}, 1-2 (1991), 61--74.

\bibitem{pham2009continuous}
{\sc Pham, H.}
\newblock {\em Continuous-time stochastic control and optimization with
  financial applications}, vol.~61.
\newblock Springer Science \& Business Media, 2009.

\bibitem{Pham2015}
{\sc Pham, H.}
\newblock Feynman-{K}ac representation of fully nonlinear {PDE}s and
  applications.
\newblock {\em Acta Math. Vietnam. 40}, 2 (2015), 255--269.

\bibitem{RasulovRaimoveMascagni2010}
{\sc Rasulov, A., Raimova, G., and Mascagni, M.}
\newblock Monte {C}arlo solution of {C}auchy problem for a nonlinear parabolic
  equation.
\newblock {\em Math. Comput. Simulation 80}, 6 (2010), 1118--1123.

\bibitem{richou2012markovian}
{\sc Richou, A.}
\newblock Markovian quadratic and superquadratic {BSDE}s with an unbounded
  terminal condition.
\newblock {\em Stochastic Processes and their Applications 122}, 9 (2012),
  3173--3208.

\bibitem{richou2011numerical}
{\sc Richou, A., et~al.}
\newblock Numerical simulation of {BSDE}s with drivers of quadratic growth.
\newblock {\em The Annals of Applied Probability 21}, 5 (2011), 1933--1964.

\bibitem{RuijterOosterlee2015}
{\sc Ruijter, M.~J., and Oosterlee, C.~W.}
\newblock A {F}ourier cosine method for an efficient computation of solutions
  to {BSDE}s.
\newblock {\em SIAM J. Sci. Comput. 37}, 2 (2015), A859--A889.

\bibitem{RuijterOosterlee2016}
{\sc Ruijter, M.~J., and Oosterlee, C.~W.}
\newblock Numerical {F}ourier method and second-order {T}aylor scheme for
  backward {SDE}s in finance.
\newblock {\em Appl. Numer. Math. 103\/} (2016), 1--26.

\bibitem{SkorohodBranchingDiffusion1964}
{\sc Skorohod, A.~V.}
\newblock Branching diffusion processes.
\newblock {\em Teor. Verojatnost. i Primenen. 9\/} (1964), 492--497.

\bibitem{teng2020multi}
{\sc Teng, L., Lapitckii, A., and G{\"u}nther, M.}
\newblock A multi-step scheme based on cubic spline for solving backward
  stochastic differential equations.
\newblock {\em Applied Numerical Mathematics 150\/} (2020), 117--138.

\bibitem{touzi2012optimal}
{\sc Touzi, N.}
\newblock {\em Optimal stochastic control, stochastic target problems, and
  backward SDE}, vol.~29.
\newblock Springer Science \& Business Media, 2012.

\bibitem{Turkedjiev2015}
{\sc Turkedjiev, P.}
\newblock Two algorithms for the discrete time approximation of {M}arkovian
  backward stochastic differential equations under local conditions.
\newblock {\em Electron. J. Probab. 20\/} (2015), no. 50, 49.

\bibitem{warin2017variations}
{\sc Warin, X.}
\newblock Variations on branching methods for non linear {PDEs}.
\newblock {\em arXiv:1701.07660\/} (2017).

\bibitem{Watanabe1965Branching}
{\sc Watanabe, S.}
\newblock On the branching process for {B}rownian particles with an absorbing
  boundary.
\newblock {\em J. Math. Kyoto Univ. 4\/} (1965), 385--398.

\bibitem{yong1999stochastic}
{\sc Yong, J., and Zhou, X.~Y.}
\newblock {\em Stochastic controls: {H}amiltonian systems and {HJB} equations},
  vol.~43.
\newblock Springer Science \& Business Media, 1999.

\bibitem{ZhangGunzburgerZhao2013}
{\sc Zhang, G., Gunzburger, M., and Zhao, W.}
\newblock A sparse-grid method for multi-dimensional backward stochastic
  differential equations.
\newblock {\em J. Comput. Math. 31}, 3 (2013), 221--248.

\bibitem{Zhang2004}
{\sc Zhang, J.}
\newblock A numerical scheme for {BSDE}s.
\newblock {\em Ann. Appl. Probab. 14}, 1 (2004), 459--488.

\bibitem{Zha17}
{\sc Zhang, J.}
\newblock {\em {Backward Stochastic Differential Equations. From Linear to
  Fully Nonlinear Theory}}.
\newblock Springer, 2017.

\bibitem{zhao2006new}
{\sc Zhao, W., Chen, L., and Peng, S.}
\newblock A new kind of accurate numerical method for backward stochastic
  differential equations.
\newblock {\em SIAM Journal on Scientific Computing 28}, 4 (2006), 1563--1581.

\bibitem{zhao2014new}
{\sc Zhao, W., Fu, Y., and Zhou, T.}
\newblock New kinds of high-order multistep schemes for coupled forward
  backward stochastic differential equations.
\newblock {\em SIAM Journal on Scientific Computing 36}, 4 (2014),
  A1731--A1751.

\bibitem{zhao2012generalized}
{\sc Zhao, W., Li, Y., and Zhang, G.}
\newblock A generalized $theta $-scheme for solving backward stochastic
  differential equations.
\newblock {\em Discrete \& Continuous Dynamical Systems-B 17}, 5 (2012), 1585.

\bibitem{zhao2009error}
{\sc Zhao, W., Wang, J., and Peng, S.}
\newblock Error estimates of the $theta $-scheme for backward stochastic
  differential equations.
\newblock {\em Discrete \& Continuous Dynamical Systems-B 12}, 4 (2009), 905.

\bibitem{zhao2010stable}
{\sc Zhao, W., Zhang, G., and Ju, L.}
\newblock A stable multistep scheme for solving backward stochastic
  differential equations.
\newblock {\em SIAM Journal on Numerical Analysis 48}, 4 (2010), 1369--1394.

\end{thebibliography}
}
\end{document}